\documentclass{amsart}

\usepackage{comment,color,amssymb,amsthm,amsmath,bussproofs}
\usepackage{graphics} 

\newdimen{\auxlength}

\DeclareSymbolFont{forpolishl}{T1}{cmr}{m}{n}
\DeclareMathSymbol{\mathrmL}{0}{forpolishl}{'212}

\newcommand{\alg}{\mathbf}
\newcommand{\class}{\mathsf}
\newcommand{\logic}{\mathrm}
\newcommand{\tuple}{\overline}

\newcommand{\set}[2]{\{ #1 \mid #2 \}}
\newcommand{\pair}[2]{\langle #1, #2 \rangle}

\newcommand{\assign}{:=}
\newcommand{\equals}{\approx}
\newcommand{\inequals}{\leq}
\newcommand{\bs}{\backslash}

\newcommand{\allset}{\emptyset}

\newcommand{\CL}{\logic{CL}}
\newcommand{\IL}{\logic{IL}}

\newcommand{\FL}{\logic{FL}}
\newcommand{\FLe}{\FL_{\mathrm{e}}}
\newcommand{\FLew}{\FL_{\mathrm{ew}}}
\newcommand{\FLen}{\FL_{\mathrm{e}}^{n}}
\newcommand{\FLewn}{\FL_{\mathrm{ew}}^{n}}

\newcommand{\K}{\logic{K}}
\newcommand{\Kfour}{\logic{K4}}
\newcommand{\Knfour}{\logic{Kn.4}}
\newcommand{\Knfourfive}{\logic{Kn.45}}
\newcommand{\Sfour}{\logic{S}4}
\newcommand{\Sfive}{\logic{S}5}

\newcommand{\IK}{\logic{IK}}

\newcommand{\IKnfour}{\logic{IKn.4}}
\newcommand{\IKnfourfive}{\logic{IKn.45}}
\newcommand{\ISfour}{\logic{IS}4}

\newcommand{\WSfive}{\logic{WS}5}
\newcommand{\MIPC}{\logic{MIPC}}

\newcommand{\Luk}{\mathrmL}
\newcommand{\Lukinfty}{\Luk_{\infty}}
\newcommand{\Product}{\Pi}
\newcommand{\Productinfty}{\Product_{\infty}}
\newcommand{\BL}{\logic{BL}}
\newcommand{\BLinfty}{\BL_{\infty}}
\newcommand{\SBL}{\logic{SBL}}

\newcommand{\Kthree}{\logic{K}_{3}}

\newcommand{\FLalg}{\class{FL}}
\newcommand{\IKalg}{\class{IK}}

\newcommand{\logleq}{\leq}

\DeclareMathOperator{\Fm}{Fm}
\DeclareMathOperator{\Th}{Th}
\DeclareMathOperator{\Thm}{Thm}
\DeclareMathOperator{\Var}{Var}

\DeclareMathOperator{\Fg}{Fg}
\newcommand{\FmAlg}{\mathop{\alg{Fm}}}

\newcommand{\ssc}[1]{\alpha_{\Th}(#1)}

\newcommand{\ddtfamily}{\Phi}
\newcommand{\ddtset}{\mathrm{I}}

\newcommand{\dilfamily}{\Psi}
\newcommand{\dilset}{\mathrm{I}}
\newcommand{\dilsettwo}{\mathrm{J}}

\newcommand{\lemfamily}{\dilfamily}
\newcommand{\lemset}{\dilset}

\newcommand{\ilfamily}{\Psi}
\newcommand{\ilset}{\mathrm{I}}
\newcommand{\ilsettwo}{\mathrm{J}}
\newcommand{\ilsetthree}{\mathrm{K}}

\newcommand{\protoset}{\Delta}
\newcommand{\pcpset}{\sqcup}
\newcommand{\lem}{\mathbf{lem}}

\newcommand{\antitheorem}{\Pi}
\newcommand{\antiformula}{\pi}
\newcommand{\antiset}{P}

\newcommand{\Btwo}{\alg{B}_{\alg{2}}}

\newcommand{\genbot}{\boldsymbol{\bot}}
\newcommand{\genneg}{\mathord{\sim}}
\newcommand{\genrightarrow}{\mathbin{\Rightarrow}}

\newcommand{\smallleq}{{\scriptscriptstyle\leq}}

\newcommand{\negbs}{\neg}
\newcommand{\negs}{\reflectbox{$\neg$}}

\newtheorem{theorem}{Theorem}[section]
\newtheorem{lemma}[theorem]{Lemma}
\newtheorem{proposition}[theorem]{Proposition}
\newtheorem{definition}[theorem]{Definition}
\newtheorem{corollary}[theorem]{Corollary}
\newtheorem{example}[theorem]{Example}
\newtheorem{fact}[theorem]{Fact}

\begin{document}

\author[T. L\'{a}vi\v{c}ka]{Tom\'{a}\v{s} L\'{a}vi\v{c}ka}
\address[Tom\'{a}\v{s} L\'{a}vi\v{c}ka]{Institute of Information Theory and Automation, Czech Academy of Sciences \\ Pod Vod\'{a}renskou v\v{e}\v{z}\'{i} 4, Praha 8, 182 00, Czechia}
\email{lavicka.thomas@@gmail.com}
\author[A. P\v{r}enosil]{Adam P\v{r}enosil}
\address[Adam~P\v{r}enosil]{Department of Mathematics, Vanderbilt University \\ 1326 Stevenson Center, Nashville, TN 37240, USA}
\email{adam.prenosil@@vanderbilt.edu}
\thanks{The research of the first author was supported by project no.\ GA17-04630S of the Czech Science Foundation.}



\title{Semisimplicity, Glivenko~theorems, and the excluded middle}

\begin{abstract}
  We formulate a general, signature-independent form of the law of the excluded middle and prove that a logic is semisimple if and only if it enjoys this law, provided that it satisfies a weak form of the so-called inconsistency lemma of Raftery. We~then show that this equivalence can be used to provide simple syntactic proofs of the theorems of Kowalski and Kracht characterizing the semisimple varieties of $\FLew$-algebras and Boolean algebras with operators, and to extend them to $\FLe$-algebras and Heyting algebras with operators. Moreover, under stronger assumptions this correspondence works at the level of individual models: the semisimple models of such a logic are precisely those which \mbox{satisfy} an axiomatic form of the law of the excluded middle, and a Glivenko-like connection ob\-tains between the logic and its extension by the axiom of the excluded middle. This in particular subsumes the well-known Glivenko theorems relating intuitionistic and classical logic and the modal logics $\Sfour$ and $\Sfive$. As a consequence, we~also obtain a description of the subclassical substructural logics which are Glivenko related to classical logic.
\end{abstract}

\maketitle

\section{Introduction}
\label{sec: intro}

  This paper investigates the relationship between three fundamental notions of logic and algebra, namely semisimplicity, the~law of the excluded middle (LEM), and Glivenko theorems, which provide double negation translations between logics. These are all shown to be part of a single coherent circle of ideas: in the first part of the paper we prove that the semisimplicity of a~logic is equivalent to some form of the LEM, while in the second part we establish a Glivenko-like connection between a logic and what we call its semisimple companion, which is in many cases axiomatized by the LEM.

  Both the LEM and Glivenko theorems presuppose some notion of \mbox{negation}. For the purposes of this paper, we take the defining feature of negation to be a weak form of the so-called inconsistency lemma (IL) of Raftery~\cite{raftery13inconsistency}, which in its strong form states that $\varphi$ is inconsistent with $\Gamma$ if and only if $\neg \varphi$ is provable from~$\Gamma$. If~such a negation is available in a logic $\logic{L}$, then an axiomatic extension of~$\logic{L}$ is semisimple if and only if it enjoys the LEM. This~provides a useful strategy for describing the semisimple axiomatic extensions of~$\logic{L}$, which for \mbox{algebraizable} logics such as the sub\-structural logic $\FLew$ or the global modal logic $\K$ amounts to describing the semisimple subvarieties of the corresponding \mbox{varieties} of \mbox{algebras}. We shall see that such non-trivial algebraic questions concerning semisimplicity can be answered using a purely syntactic approach via the LEM, and more importantly that doing so will result in brief and transparent proofs.

  The study of the semisimple extensions of a given logic $\logic{L}$ naturally leads one to consider the semisimple companion of $\logic{L}$, which we define as the logic $\ssc{\logic{L}}$ determined by the maximal consistent theories of $\logic{L}$. The semisimple companion is often axiomatized relative to $\logic{L}$ by an axiomatic form of the LEM. It can be thought of as a dual counterpart to the so-called structural (or admissible) completion of $\logic{L}$, which is the logic determined by the smallest theory of $\logic{L}$, i.e.\ the set of theorems of~$\logic{L}$. Just like the rules valid in the structural completion may be described in terms of the theorems of $\logic{L}$ as the so-called admissible rules of $\logic{L}$, we describe the rules valid in the semisimple companion in terms of the antitheorems of $\logic{L}$ as what we call the antiadmissible rules of $\logic{L}$. The inconsistency lemma then allows us to reformulate antiadmissibility as a Glivenko-like theorem connecting $\logic{L}$ and $\ssc{\logic{L}}$ by means of a double negation translation. Examples of this phenomenon abound among substructural and modal logics.

  The present paper therefore brings together several strands of research within non-classical logic: the~study of inconsistency lemmas initiated by Raftery~\cite{raftery13inconsistency}, who established the equivalence between semisimplicity and the LEM (more precisely, what he called the classical IL) assuming a strong form of the IL, the descriptions of the semisimple varieties of $\FLew$-algebras and modal algebras (Boolean algebras with operators) due to Kowalski and Kracht~\cite{kowalski04,kowalski+kracht06}, and the investigation of Glivenko theorems due to Cignoli \& Torrens~\cite{cignoli+torrens03,cignoli+torrens04} in the context of extensions of integral residuated structures, Galatos \& Ono~\cite{galatos+ono06} in the context of substructural logics, and Bezhanishvili~\cite{bezhanishvili01} in the context of intuitionistic modal logics. We extend various results from these papers and put them into a \mbox{uniform} perspective. The theory of structural completions and admissible rules~\cite{pogorzelski71,bergman91,rybakov97} also finds a dual counterpart in our study of semisimple companions and antiadmissible rules.


  The general approach to Glivenko theorems due to Torrens~\cite{torrens08}, on the other hand, is not directly comparable to the approach we take in the present paper. Torrens understands Glivenko theorems to be relative to a unary formula which represents a homomorphism onto an algebra of ``regular'' elements, while we under\-stand Glivenko theorems to be relative to a family of sets of unary formulas related to inconsistency by means of an IL. Both approaches have their merits: ours does not directly yield the known Glivenko theorem relating H\'{a}jek's basic fuzzy logic $\BL$ and the infinite-valued {\L}ukasiewicz logic $\Luk$ (because $\BL$ does not satisfy a strong enough IL), while the approach of Torrens does not directly yield the known Glivenko theorem relating the global modal logics $\Sfour$ and $\Sfive$ (because $\neg \Box \neg \Box x$ does not represent a homomorphism onto an algebra of regular elements). We do, however, obtain a ``local'' Glivenko theorem relating $\BL$ to the infinitary version of {\L}ukasiewicz logic $\Lukinfty$.

  It is important to observe that although we choose to formulate our results in terms of Hilbert-style consequence relations (consequence relations on formulas), they can be extended to the so-called $k$-deductive systems of Blok \& Pigozzi~\cite{blok+pigozzi97} (consequence relations on $k$-tuples of formulas), and even further to universal Horn logic without equality~\cite{elgueta94}. These in particular subsume the equational consequence relations of (generalized) quasivarieties, which are classes of algebras axiomatized by (possibly infinitary) implications between equalities in a given set of variables. The extension to the $k$-deductive setting does not involve any substantial mathematical work, however, it does come with an increase in the complexity of notation. We~thus opt for formulating our results in logical terms and using the notion of algebraizability to transfer them to the algebraic realm. The reader familiar with $k$-deductive systems will see that our results can instead immediately be translated into results concerning equational consequence relations and (generalized) quasivarieties without any need for an intermediate step in the form of a Hilbert-style consequence relation.

\subsection{Semisimple algebras and logics}

  Having outlined the main themes of the present paper, let us now consider them in more detail, starting with semisimplicity. Recall that an algebra is called \emph{simple} if it has precisely two congruences, namely the equality relation and the total relation, and it is called \emph{semisimple} if it is a subdirect product of simple algebras, or equivalently if the equality relation is the inter\-section of all maximal non-trivial congruences. A~variety (equational class) of algebras $\class{K}$ is semisimple if each algebra in $\class{K}$ is semisimple. Boolean algebras, distributive lattices, and De~Morgan algebras are all semisimple varieties, while Heyting algebras and modal algebras (Boolean algebras with operators) are examples of varieties which are not semisimple. If we replace congruences by $\class{K}$-congruences (congruences $\theta$ on $\alg{A}$ such that $\alg{A} / \theta \in \class{K}$), this definition of semisimplicity extends to all classes of algebras in a given signature closed under subdirect products.

  Our notion of a semisimple logic, by contrast, will be syntactic. The~maximal non-trivial theories of a logic are called \emph{simple}, a theory is \emph{semisimple} if it is an intersection of simple theories, and a logic is semisimple if all of its theories~are. A~theory here is a set of formulas closed under the consequence relation of the logic. Under some mild assumptions, this extends to a semantic form of semi\-simplicity: each model of the logic is a subdirect product of simple models.

  If a logic $\logic{L}$ enjoys a tight connection to a class of algebras $\class{K}$, then the semantic semisimplicity of $\logic{L}$ is equivalent to the semisimplicity of $\class{K}$ in the algebraist's sense. In technical terms, this holds if $\logic{L}$ is (at least weakly) algebraizable and $\class{K}$ forms the algebraic counterpart of~$\logic{L}$. We can thus use the results proved here to answer algebraic questions concerning semisimplicity. We focus on cases where $\class{K}$ forms a variety, but the same reasoning applies to (generalized) quasivarieties.

  We consider two kinds of problems involving semisimplicity. Firstly:
\begin{align*}
  \text{given a variety $\class{K}$, describe its semisimple algebras.}
\end{align*}
  For example, a Heyting algebra is semisimple if and only if it is a Boolean algebra. Similarly, an $\Sfour$ modal algebra (i.e.\ a Boolean algebra with a topological interior operator) is semisimple if and only if it is an $\Sfive$ modal algebra (an $\Sfour$ modal algebra satisfying $x \inequals \Box \Diamond x$). A related but distinct problem is:
\begin{align*}
  \text{given a variety $\class{K}$, describe its semisimple subvarieties.}
\end{align*}
  For example, a result of Kowalski~\cite{kowalski04} states that a variety of $\FLew$-algebras (i.e.\ bounded commutative integral residuated lattices) is semisimple if and only if it validates the equation $x \vee \neg (x^{n}) \equals 1$ for some $n$, while an analogous result of Kowalski \& Kracht~\cite{kowalski+kracht06} states that a variety of modal algebras is semisimple if and only if it is validates the equational axioms of weak $n$-transitivity and \mbox{$n$-cyclicity} for some $n$. However, an individual $\FLew$-algebra or modal algebra may well be semisimple without validating any of these equations.

  In case the variety $\class{K}$ forms the algebraic counterpart of an (at least weakly) algebraizable logic $\logic{L}$, the first of these problems is equivalent to:
\begin{align*}
  \text{describe the semisimple models of $\logic{L}$}.
\end{align*}
  Similarly, the second one is then equivalent to:
\begin{align*}
  \text{describe the semisimple axiomatic extensions of $\logic{L}$}.
\end{align*}
  We shall see that, using the equivalence between semisimplicity and the LEM, both types of problems can be attacked by entirely syntactic methods. To see this strategy in action, the~interested reader may skip ahead to Subsection~\ref{subsec: applications}, where we extend the results of Kowalski and Kracht to $\FLe$-algebras and modal Heyting algebras. Note that the original proof of Kowalski for $\FLew$-algebras involves several pages of algebraic computations, which are far from straight\-forward, while our proof takes about two paragraphs (given the equivalence between semisimplicity and the LEM).

\subsection{The law of the excluded middle}

  Let us now explain what we mean by the LEM. We take it to be a syntactic principle which, roughly speaking, states that if in a given context $\Gamma$ a formula is derivable from both $\varphi$ and the negation of $\varphi$, then it is already derivable from $\Gamma$. For~example, classical logic enjoys the LEM in the~form
\begin{align*}
  \Gamma, \varphi \vdash_{\CL} \psi \text{ and } \Gamma, \neg \varphi \vdash_{\CL} \psi & \implies \Gamma \vdash_{\CL} \psi,
\end{align*}
  while the global modal logic $\Sfive$ enjoys the LEM in the form
\begin{align*}
  \Gamma, \varphi \vdash_{\Sfive} \psi \text{ and } \Gamma, \neg \Box \varphi \vdash_{\Sfive} \psi & \implies \Gamma \vdash_{\CL} \psi.
\end{align*}
  More generally, a whole family of formulas may play the role of a negation. Such principles will be called \emph{local} LEMs, as opposed to the \emph{global} LEMs shown above. For example, the infinitary {\L}ukasiewicz logic $\Lukinfty$ enjoys the following local LEM:
\begin{align*}
  \Gamma, \varphi \vdash_{\Lukinfty} \psi \text{ and } \Gamma, \neg (\varphi^{n}) \vdash_{\Lukinfty} \psi \text{ for each } n \in \omega & \implies \Gamma \vdash_{\Lukinfty} \psi.
\end{align*}
  Although the LEM is formulated as an implication between valid rules, under suitable assumptions it can be expressed in axiomatic form.

  The connection between semisimplicity and the LEM presupposes that the negation satisfies a so-called inconsistency lemma (IL). This is an equivalence which relates inconsistency and validity in a manner analogous to the following equivalence for classical logic:
\begin{align*}
  \Gamma, \varphi \vdash_{\CL} \allset & \iff \Gamma \vdash_{\CL} \neg \varphi.
\end{align*}
  Here $\Gamma, \varphi \vdash_{\CL} \allset$ represents the classical inconsistency of the set of formulas $\Gamma, \varphi$. For example, the global modal logic $\Sfour$ enjoys the IL
\begin{align*}
  \Gamma, \varphi \vdash_{\Sfour} \allset & \iff \Gamma \vdash_{\Sfour} \neg \Box \varphi,
\end{align*}
  which indicates that the connective $\neg \Box x$ (rather than merely $\neg x$) functions as a negation in $\Sfour$, while the $(n+1)$-valued {\L}ukasiewicz logic $\Luk_{n+1}$ enjoys the IL
\begin{align*}
  \Gamma, \varphi \vdash_{\Luk_{n+1}} \allset & \iff \Gamma \vdash_{\Luk_{n+1}} \neg (\varphi^{n}).
\end{align*}

  As with LEMs, we call such equivalences \emph{global} ILs, in contrast to \emph{local} ILs such as the one enjoyed by the substructural logic $\FLew$ and its axiomatic extensions, as well as the infinitary {\L}ukasiewicz logic $\Lukinfty$:
\begin{align*}
  \Gamma, \varphi \vdash_{\FLew} \allset & \iff \Gamma \vdash_{\FLew} \neg (\varphi^{n}) \text{ for some } n \in \omega.
\end{align*}
  \emph{Parametrized} local ILs, enjoyed e.g.\ by $\FL$, are more general still.

  Such equivalences (in their global form) were first studied as conditions in their own right by Raftery~\cite{raftery13inconsistency}. In addition to ordinary global ILs, Raftery also considered \emph{classical} global ILs. These combine an ordinary global IL with what we call a \emph{dual} IL, generalizing the classical equivalence
\begin{align*}
  \Gamma \vdash_{\CL} \varphi & \iff \Gamma, \neg \varphi \vdash_{\CL} \allset.
\end{align*}
  Equivalently, a classical IL combines an ordinary IL and with a LEM. Raftery then proved that for each algebraizable logic $\logic{L}$ whose algebraic counterpart is a quasivariety $\class{K}$, the classical global IL for the logic $\logic{L}$ corresponds precisely to the filtrality of the quasivariety $\class{K}$, i.e.\ to relative semisimplicity plus the equational definability of principal relative congruences in $\class{K}$.

  We are now ready to state our main results relating the semisimplicity of a logic and the LEM (Theorems~\ref{thm: semisimplicity} and~\ref{thm: semantic semisimplicity}), which can be thought of as extending Raftery's results to the (parametrized) local form of the classical IL: if~$\logic{L}$ enjoys a parametrized local IL, then
\begin{align*}
  \text{all theories of $\logic{L}$ are semisimple if and only if $\logic{L}$ satisfies the LEM}.
\end{align*}
  If $\logic{L}$ enjoys a local IL, then this can be upgraded to:
\begin{align*}
  \text{all models of $\logic{L}$ are semisimple if and only if $\logic{L}$ satisfies the LEM}.
\end{align*}
  Since the (parametrized) local IL is inherited by all axiomatic extensions of $\logic{L}$, these theorems will enable us to characterize the semisimple axiomatic extensions of the substructural logic $\FLe$ and the intuitionistic modal logic $\IK$.

\subsection{Semisimple companions}

  To tackle the second problem involving semisimplicity, namely describing the semisimple models of a given logic, we study the \emph{(syntactic) semisimple companion} $\ssc{\logic{L}}$ of a logic $\logic{L}$, defined as the logic determined by the simple theories of $\logic{L}$. That is, $\ssc{\logic{L}}$ is the logic of all matrices of the form $\pair{\FmAlg \logic{L}}{T}$ where $T$ ranges over simple theories of $\logic{L}$. Our next result (Theorem~\ref{thm: models of semisimple companion}) states that if $\logic{L}$ enjoys a local IL, then
\begin{align*}
  \text{a model of $\logic{L}$ is semisimple if and only if it is a model of $\ssc{\logic{L}}$}.
\end{align*}
  The problem of describing the semisimple models of $\logic{L}$ thus reduces to the problem of axiomatizing $\ssc{\logic{L}}$.

  For instance, the semisimple companion of H\'{a}jek's basic fuzzy logic $\BL$ (which inherits a local IL from~$\FLew$) is the infinitary {\L}ukasiewicz logic $\Lukinfty$. This is because the algebraic counterpart of $\BL$ is the variety of $\BL$-algebras, semisimple $\BL$-algebras coincide with semisimple MV-algebras (see~\cite{turunen99}), and these in turn form the algebraic counterpart of $\Lukinfty$ (as opposed to the finitary logic $\Luk$, whose algebraic counterpart is the variety of MV-algebras). In this case, we used our ability to identify the semisimple algebras in the given variety to describe $\ssc{\BL}$. However, we can also proceed in the opposite direction.

  The problem of axiomatizing the semisimple companion $\ssc{\logic{L}}$ can be solved in an entirely mechanical manner whenever $\logic{L}$ enjoys a global IL and the LEM can be expressed in axiomatic form. This occurs if $\logic{L}$ has either a well-behaved disjunction or a well-behaved implication. More precisely, we assume either a global \emph{deduction--detachment theorem} (DDT), generalizing the equivalence
\begin{align*}
  \Gamma, \varphi \vdash_{\CL} \psi & \iff \Gamma \vdash_{\CL} \varphi \rightarrow \psi,
\end{align*}
  or a global \emph{proof by cases property} (PCP), generalizing the equivalence
\begin{align*}
  \Gamma, \varphi \vdash_{\CL} \chi \text{ and } \Gamma, \psi \vdash_{\CL} \chi \iff \Gamma, \varphi \vee \psi \vdash_{\CL} \chi.
\end{align*}
  For example, the global modal logic $\Sfour$ enjoys a global DDT in the form
\begin{align*}
  \Gamma, \varphi \vdash_{\Sfour} \psi \iff \Gamma \vdash_{\Sfour} \Box \varphi \rightarrow \psi,
\end{align*}
  as well as a global PCP in the form
\begin{align*}
  \Gamma, \varphi \vdash_{\Sfour} \chi \text{ and } \Gamma, \psi \vdash_{\Sfour} \chi \iff \Gamma, \Box \varphi \vee \Box \psi \vdash_{\Sfour} \chi.
\end{align*}
  These indicate that the connectives $\Box x \rightarrow y$ and $\Box x \vee \Box y$ (rather than merely $x \rightarrow y$ and $x \vee y$) function respecitvely as an implication and a disjunction in $\Sfour$.

  In~logics which enjoy the global IL and either the global DDT or the global PCP, the LEM may be expressed by an axiom of one of the two forms
\begin{align*}
  & (p \rightarrow q) \rightarrow ((\neg p \rightarrow q) \rightarrow q) & & \text{or} & & p \vee \neg p,
\end{align*}
  where the negation, implication, and disjunction are the connectives which occur in the IL, DDT, and PCP. These need not coincide with the connectives which are customarily denoted $\neg x$, $x \rightarrow y$, and $x \vee y$ in the given logic: in $\Sfour$ these are respectively $\neg \Box x$, $\Box x \rightarrow y$, and $\Box x \vee \Box y$.

  Whenever such an axiomatic form of the LEM is available, the semisimple companion of $\logic{L}$ is the extension of $\logic{L}$ by the axiom of the excluded middle (Proposition~\ref{prop: axiomatization of semisimple companion}). It~immediately follows (Theorem~\ref{thm: semisimple models}) that
\begin{align*}
  \text{a model of $\logic{L}$ is semisimple if and only if it is satisfies the (axiomatic) LEM}.
\end{align*}
  This subsumes the facts that semisimple Heyting algebras are precisely Boolean algebras, and semisimple $\Sfour$ modal algebras are precisely $\Sfive$ modal algebras. (Recall that $\Sfive$ is the extension of $\Sfour$ by the axiom $p \rightarrow \Box \Diamond p$, which is equivalent in $\Sfour$ to the LEM in the form $\Box p \vee \Box \neg \Box p$, thus $\ssc{\Sfive} = \Sfour$.) These two particular facts are of course not difficult to prove directly, but the merit of our approach lies in providing a uniform proof of such facts in any logic with a well-behaved negation and implication or disjunction.

\subsection{Glivenko theorems}

  The framework outlined above now allows us to spell out a precise connection between semisimplicity and Glivenko theorems. The semisimple companion can be thought of as a dual counterpart to the so-called structural completion of a logic: the semisimple companion of $\logic{L}$ is the logic determined by the largest non-trivial theories of $\logic{L}$, while the structural completion of $\logic{L}$ is the logic determined by the smallest theory of $\logic{L}$, i.e.\ by the matrix $\pair{\FmAlg \logic{L}}{\Thm \logic{L}}$. Just~like the structural completion is the strongest logic with the same theorems, the antistructural completion of a compact logic is the strongest logic with the same antitheorems. This relates the study of semisimple companions to the established topic of structural completeness (see e.g.~\cite{pogorzelski71,bergman91,rybakov97}). Indeed, we may call~$\ssc{\logic{L}}$ the \emph{anti\-structural completion} of~$\logic{L}$.

  In particular, the notion of an admissible rule (see~\cite{rybakov97}) has a natural counterpart in the theory of semisimple companions. Recall that a rule $\Gamma \vdash \varphi$ is valid in the structural completion of~$\logic{L}$ if it is \emph{admissible}:
\begin{align*}
  \emptyset \vdash_{\logic{L}} \sigma[\Gamma] & \implies \emptyset \vdash_{\logic{L}} \sigma(\varphi) \text{ for each substitution $\sigma$.}
\end{align*}
  Similarly (Proposition~\ref{prop: antiadmissible rules}), a rule $\Gamma \vdash \varphi$ is valid in the antistructural completion of a compact logic $\logic{L}$, i.e.\ in $\ssc{\logic{L}}$, if it is \emph{antiadmissible}:
\begin{align*}
  \sigma(\varphi), \Delta \vdash_{\logic{L}} \allset & \implies \sigma[\Gamma], \Delta \vdash_{\logic{L}} \allset \text{ for each $\Delta \subseteq \Fm \logic{L}$ and each substitution $\sigma$.}
\end{align*}
  For compact logics which enjoy a local IL (Proposition~\ref{prop: simply antiadmissible rules}), this simplifies to:
\begin{align*}
  \varphi, \Delta \vdash_{\logic{L}} \allset & \implies \Gamma, \Delta \vdash_{\logic{L}} \allset \text{ for each $\Delta \subseteq \Fm \logic{L}$.}
\end{align*}

  A Glivenko theorem relating $\logic{L}$ and $\ssc{\logic{L}}$ now results if we apply the IL to this simplified definition of antiadmissibility and recall that $\Gamma \vdash_{\ssc{\logic{L}}} \varphi$ if and only if $\Gamma \vdash \varphi$ is antiadmissible in $\logic{L}$. This subsumes as special cases the familiar Glivenko theorems for intuitionistic logic and $\Sfour$ (see~\cite{glivenko29,matsumoto55}), or more precisely their extensions concerning valid rules rather than theorems:
\begin{align*}
  \Gamma \makebox[2.2em]{$\vdash_{\CL}$} \varphi & \iff \Gamma \makebox[2em]{$\vdash_{\IL}$} \neg \neg \varphi, \\
  \Gamma \makebox[2.2em]{$\vdash_{\Sfive}$} \varphi & \iff \Gamma \makebox[2em]{$\vdash_{\Sfour}$} \neg \Box \neg \Box \varphi.
\end{align*}
  For a logic such as $\FLew$ which only enjoys a local IL with respect to the family of formulas $\neg (x^{n})$, we obtain what we might call a \emph{local} Glivenko theorem:
\begin{align*}
  \Gamma \vdash_{\ssc{\FLew}} \varphi & \iff \Gamma \vdash_{\FLew} \set{\neg (\neg \varphi^{n})^{f(n)}}{n \in \omega} \text{ for some } f\colon \omega \to \omega.
\end{align*}
  In particular, a local Glivenko theorem connects $\BL$ and $\ssc{\BL} =\Lukinfty$:
\begin{align*}
  \Gamma \vdash_{\Lukinfty} \varphi & \iff \Gamma \vdash_{\BL} \set{\neg (\neg \varphi^{n})^{f(n)}}{n \in \omega} \text{ for some } f\colon \omega \to \omega.
\end{align*}
  This Glivenko theorem can perhaps more naturally be expressed as:
\begin{align*}
  \Gamma \vdash_{\Lukinfty} \varphi & \iff \text{ for each } n \in \omega \text{ there is some } k \in \omega \text{ such that } \Gamma \vdash_{\BL} \neg (\neg \varphi^{n})^{k}.
\end{align*}
  In case $\Gamma$ is finite, this is superseded by the global Glivenko theorem due to Cignoli and Torrens~\cite{cignoli+torrens03}:
\begin{align*}
  \Gamma \vdash_{\Lukinfty} \varphi & \iff \Gamma \vdash_{\Luk} \varphi \iff \Gamma \vdash_{\BL} \neg \neg \varphi.
\end{align*}

  Note that this approach to Glivenko theorems, which relies on the relationship between negation and inconsistency, differs substantially from the approach of Cignoli and Torrens~\cite{cignoli+torrens03,torrens08} and of Galatos \& Ono~\cite{galatos+ono06}, which relies on a suitable term-definable projection onto an algebra of regular elements.

\section{Preliminaries}
\label{sec: prelim}

  We now introduce the necessary preliminaries. We recall some basic notions of abstract algebraic logic concerning logics and their models. The reader unfamiliar with this framework may find the recent textbook~\cite{font16} helpful. We then introduce our running examples: substructural and normal modal logics. Finally, we review some notational conventions used in this paper.

\pagebreak

\subsection{Logics}

  A \emph{logic} $\logic{L}$ is a structural closure operator on the absolutely free algebra (the algebra of formulas) $\FmAlg \logic{L}$ on a given set of variables $\Var \logic{L}$, where structurality means that if $T$ is a closed set of formulas, i.e.\ a \emph{theory} of $\logic{L}$, then so is $\sigma^{-1}[T]$ for any substitution (homomorphism) $\sigma\colon \FmAlg \logic{L} \to \FmAlg \logic{L}$. Equivalently, we can think of a logic as a consequence relation, denoted $\Gamma \vdash_{\logic{L}} \varphi$, between sets of formulas $\Gamma$ and formulas $\varphi$ such that
\begin{itemize}
\item $\varphi \vdash_{\logic{L}} \varphi$,
\item if $\Gamma \vdash_{\logic{L}} \varphi$, then $\Gamma, \Delta \vdash_{\logic{L}} \varphi$,
\item if $\Gamma \vdash_{\logic{L}} \Phi$ and $\Phi \vdash_{\logic{L}} \varphi$, then $\Gamma \vdash_{\logic{L}} \varphi$, and
\item if $\Gamma \vdash_{\logic{L}} \varphi$, then $\sigma[\Gamma] \vdash_{\logic{L}} \sigma(\varphi)$ for each substitution $\sigma$.
\end{itemize}
  Here $\Gamma \vdash_{\logic{L}} \Phi$ abbreviates the claim that $\Gamma \vdash_{\logic{L}} \varphi$ for each $\varphi \in \Phi$. Given $\logic{L}$ as a structural closure operator, we define $\Gamma \vdash_{\logic{L}} \varphi$ to hold if and only if $\varphi$ is in the closure of $\Gamma$. Conversely, given $\logic{L}$ as a consequence relation, we define the closed sets (theories) of $\logic{L}$ to be the sets of formulas $T$ such that $T \vdash_{\logic{L}} \varphi$ implies $\varphi \in T$. The~set of all theories of $\logic{L}$ forms a complete lattice~$\Th \logic{L}$.

  A logic $\logic{L}$ is \emph{finitary}, or more generally $\kappa$-ary, if $\Gamma \vdash_{\logic{L}} \varphi$ implies $\Gamma' \vdash_{\logic{L}} \varphi$ for some finite $\Gamma' \subseteq \Gamma$, or more generally for some $\Gamma' \subseteq \Gamma$ such that $| \Gamma' | < \kappa$.

  A pair $\pair{\alg{A}}{F}$ consisting of an algebra $\alg{A}$ and a set $F \subseteq A$ is called a \emph{matrix}. A~matrix is a \emph{model} of a logic $\logic{L}$ if for each homomorphism $h\colon \FmAlg \logic{L} \to \alg{A}$ the set $h^{-1}[F]$ is a theory of $\logic{L}$, in other words if $h$ designates $\varphi$ whenever it designates $\Gamma$. Here we say that $h$ designates a set of formulas $\Phi$ in $\pair{\alg{A}}{F}$ if $h[\Phi] \subseteq F$. If $\pair{\alg{A}}{F}$ is a model of $\logic{L}$, we also say that $F$ is a \emph{filter} of $\logic{L}$ on $\alg{A}$. For each algebra~$\alg{A}$ the set of all filters of $\logic{L}$ on $\alg{A}$ forms a closure system. The filter \emph{generated} by $X \subseteq A$ is then denoted $\Fg_{\logic{L}}^{\alg{A}} X$. Mimicking the notation for formulas, we also write $X \vdash_{\logic{L}}^{\alg{A}} a$ for $a \in \Fg_{\logic{L}}^{\alg{A}} X$.

  A matrix $\pair{\alg{A}}{F}$, and by extension the filter $F$, is called \emph{trivial} if $F = A$. In~particular, the trivial theory of $\logic{L}$ is the theory $\Fm \logic{L}$. A \emph{simple} theory of $\logic{L}$ is a maximal non-trivial theory, while a \emph{semisimple} theory of $\logic{L}$ is an intersection of simple theories of $\logic{L}$. A logic~$\logic{L}$ is then called semisimple if each theory of $\logic{L}$ is semisimple. Every semisimple logic is \emph{coatomic}: each non-trivial theory is included in a simple theory.

  An \emph{antitheorem} of a logic $\logic{L}$ is a set of formulas $\Gamma$ which cannot be jointly designated in any non-trivial model. In other words, for each non-trivial model $\pair{\alg{A}}{F}$ of~$\logic{L}$ and each homomorphism $h\colon \FmAlg \logic{L} \to \alg{A}$ it is not the case that $h[\Gamma] \subseteq F$. The claim that $\Gamma$ is an antitheorem of $\logic{L}$ shall be abbreviated by $\Gamma \vdash_{\logic{L}} \allset$.\footnote{We take $\vdash \allset$ to be a single indivisible symbol, rather than analyzing it as $\vdash \Phi$ for ${\Phi = \emptyset}$. (The latter would be appropriate if sets of formulas $\Phi$ to the right of the turnstile were given a disjunctive rather than conjunctive reading.) While this notation clashes with the convention that $\Gamma \vdash_{\logic{L}} \Phi$ means $\Gamma \vdash_{\logic{L}} \varphi$ for each ${\varphi \in \Phi}$, no~confusion is likely to occur. The~motivation behind this notation stems from the duality between theorems and antitheorems.} It is straight\-forward to prove the following analogues of monotonicity, cut, and structurality for antitheorems:
\begin{itemize}
\item if $\Gamma \vdash_{\logic{L}} \allset$, then $\Gamma \vdash_{\logic{L}} \varphi$,
\item if $\Gamma \vdash_{\logic{L}} \allset$, then $\Gamma, \Delta \vdash_{\logic{L}} \allset$,
\item if $\Gamma \vdash_{\logic{L}} \Phi$ and $\Phi \vdash_{\logic{L}} \allset$, then $\Gamma \vdash_{\logic{L}} \allset$, and
\item if $\Gamma \vdash_{\logic{L}} \allset$, then $\sigma[\Gamma] \vdash_{\logic{L}} \allset$ for each substitution $\sigma$.
\end{itemize}


  A logic $\logic{L}$ is \emph{compact}, or more generally \emph{$\kappa$-compact}, if $\Gamma \vdash_{\logic{L}} \Fm \logic{L}$ implies $\Gamma' \vdash_{\logic{L}} \Fm \logic{L}$ for some finite $\Gamma' \subseteq \Gamma$, or more generaly for some $\Gamma' \subseteq \Gamma$ such that $|\Gamma'| < \kappa$. In~particular, each compact theory has a finite antitheorem, and conversely each finitary logic with an antitheorem is compact. For finitary logics the compactness of $\logic{L}$ coincides with the lattice-theoretic compactness of $\Fm \logic{L}$ in the lattice $\Th \logic{L}$. Observe that each compact logic is coatomic.

  A logic need not have antitheorems: for example, in the positive fragment of classical logic (whose signature consists of $\wedge$, $\vee$, $\top$, $\rightarrow$) every formula can be designated in the two-element Boolean matrix $\pair{\Btwo}{\{ 1 \}}$ by the homomorphism sending each variable to $1$. Thus $\Gamma \vdash_{\logic{L}} \allset$ is a stronger claim than $\Gamma \vdash_{\logic{L}} \Fm \logic{L}$.

\begin{fact}
  If a variable $p$ does not occur in $\Gamma$, then $\Gamma \vdash_{\logic{L}} \allset$ if and only if~$\Gamma \vdash_{\logic{L}} p$.
\end{fact}

\begin{proof}
  If~$h\colon \FmAlg \logic{L} \to \alg{A}$ designates $\Gamma$ in a model $\pair{\alg{A}}{F}$ of $\logic{L}$, then for each $a \in \alg{A}$ there is a homomorphism $g_{a}\colon \FmAlg \logic{L} \to \alg{A}$ such that $g_{a}[\Gamma] \subseteq F$ and $g_{a}(p) = a$. Thus $F = A$ whenever $\Gamma$ can be designated in $\pair{\alg{A}}{F}$.
\end{proof}


\begin{fact}
  $\Gamma \vdash_{\logic{L}} \allset$ if and only if $\sigma[\Gamma] \vdash_{\logic{L}} \Fm \logic{L}$ for each substitution $\sigma$.
\end{fact}

\begin{proof}
  Right to left, let $\sigma$ and $\tau$ be substitutions such that no formula in $\sigma[\Fm \logic{L}]$ contains the variable $p$ and $\tau(\sigma(\varphi)) = \varphi$ for each $\varphi$. Then $\sigma[\Gamma] \vdash_{\logic{L}} p$, so $\sigma[\Gamma] \vdash_{\logic{L}} \allset$ by the previous fact. Applying $\tau$ now yields that $\Gamma = \tau[\sigma[\Gamma]] \vdash_{\logic{L}} \allset$.
\end{proof}

\subsection{Running examples: substructural logics}
\label{sec: running examples substructural}

  Let us now introduce the two families of logics which will serve as running examples throughout the~\mbox{paper}. We first consider substructural logics, which we take to be the axiomatic extensions of the Full~Lambek calculus $\FL$. A standard reference text for substructural logics is~\cite{residuated07}. The signature of $\FL$ consists of:
\begin{itemize}
\item the lattice meet and join $\wedge$ and $\vee$,
\item the top and bottom constants $\top$ and $\bot$,
\item the monoidal multiplication and unit $\cdot$ and $1$,
\item the residuals of multiplication $\bs$ and $/$, and
\item the constant $0$.
\end{itemize}
  The traditional inclusion of the constant $0$ in the signature will be of little consequence to us (in fact, of no consequence except for the very last theorem of this paper). The~reader may well decide to drop this constant from the signature. The presence of $\bot$, on the other hand, will be crucial.

  We define the consequence relation of $\FL$ in terms of its algebraic semantics given by $\FL$-algebras. These are algebras in the above signature such that $\langle A, \wedge, \vee, \top, \bot \rangle$ is a bounded lattice, $\langle A, \cdot, 1 \rangle$ is a monoid, and moreover the residuation law holds:
\begin{align*}
  x \leq z / y \iff x \cdot y \leq z \iff y \leq x \bs z.
\end{align*}
  $\FL$-algebras form a variety. The subvariety of $\FLe$-algebras imposes the axiom of commutativity (also known as exchange) $x \cdot y \equals y \cdot x$, or equivalently $x \bs y \equals y / x$. The subvariety of $\FLew$-algebras moreover imposes the axiom of integrality $1 \equals \top$ (also known as weakening), or equivalently $x \cdot y \inequals x \wedge y$. Imposing the axiom of contraction $x \inequals x \cdot x$ on $\FLew$-algebras yields the variety of Heyting algebras.

  Throughout the paper we use the standard notation $x^{0} \assign 1$ and $x^{n+1} \assign x^{n} \cdot x$. The~varieties of $\FLen$ and $\FLewn$ algebras are then defined as the subvarieties of $\FLe$ and $\FLew$ which impose the axiom of $n$-contraction $(1 \wedge x)^{n+1} \equals (1 \wedge x)^{n}$. The notation $\neg x$ abbreviates $x \bs \bot$, with $\neg x^{n}$ interpreted as $\neg (x^{n})$.

  The consequence relation of the logic $\FL$ is defined in terms of the equational consequence relation $\vDash_{\FLalg}$ of $\FL$-algebras as follows:
\begin{align*}
  \Gamma \vdash_{\FL} \varphi & \iff \set{1 \inequals \gamma}{\gamma \in \Gamma} \vDash_{\FLalg} 1 \inequals \varphi.
\end{align*}
  Analogous equivalences relate the logics $\FLe$, $\FLew$, $\FLen$, and $\FLewn$ with the corresponding classes algebras. In the terminology of abstract algebraic logic, these logics are \emph{algebraizable}, and their \emph{algebraic counterparts} are the varieties of $\FL$-algebras, $\FLe$-algebras, and $\FLew$-algebras. We shall not need to define these terms precisely for the purposes of this paper.

  In $\FLe$-algebras we use $x \rightarrow y$ to denote the element $x \bs y = y / x$. In $\FLe$ and its extensions, we thus in effect replace the two residuals of multiplication in the signature of $\FL$ by a single residual $\rightarrow$.

  Among the best-known extensions of $\FLew$ are the finite-valued {\L}ukasiewicz logics $\Luk_{n+1}$, the finitary {\L}ukasiewicz logic $\Luk$, and its infinitary counterpart $\Lukinfty$. These may all be defined in terms of a single matrix and its finite submatrices, namely the standard {\L}ukasiewicz chain $[0, 1]_{\Luk}$ with the filter $\{ 1 \}$. This is the real interval $[0, 1]$ equipped with the standard lattice structure and the operations
\begin{align*}
  x \cdot y & \assign \max(x + y - 1, 0), \\
  x \rightarrow y & \assign \min(1 - x + y, 1).
\end{align*}
  In terms of these, we may define some further operations:
\begin{align*}
  \neg x & \assign x \rightarrow \bot = 1 - x, & 0 x & \assign 0, \\
  x \oplus y & \assign \neg (\neg y \cdot \neg x) = \min(x + y, 1), & (n+1) x & \assign x \oplus n x.
\end{align*}
  The matrix $\pair{[0, 1]_{\Luk}}{\{ 1 \}}$ determines the infinitary {\L}ukasiewicz logic $\Lukinfty$, of which the finitary {\L}ukasiewicz logic $\Luk$ is the finitary companion. That is,
\begin{align*}
  \Gamma \vdash_{\Luk} \varphi & \iff \Gamma' \vdash_{\Lukinfty} \varphi \text{ for some finite } \Gamma' \subseteq \Gamma.
\end{align*}
  The logics $\Luk_{n+1}$ are then determined by the submatrices over $\{ \frac{0}{n}, \frac{1}{n}, \dots, \frac{n-1}{n}, \frac{n}{n} \}$.

  The failure of finitarity in $\Lukinfty$ is witnessed by the rule
\begin{align*}
  \set{\neg p \rightarrow q^{n}}{n \in \omega} \vdash_{\Lukinfty} p \vee q,
\end{align*}
  which in fact axiomatizes $\Lukinfty$ relative to $\Luk$. Nevertheless, it is known that $\Lukinfty$ is compact~\cite{tnorms95}, along with other related logics such as the infinitary version of the product logic $\Productinfty$~\cite{cintula+navara04} or the infinitary version $\BLinfty$ of H\'{a}jek's basic fuzzy logic.\footnote{The logic $\Productinfty$ is determined by the standard product chain $[0, 1]_{\Product}$, while $\BLinfty$ is determined by the standard $\BL$-chains, i.e.\ $\BL$-algebras over the unit interval $[0, 1]$. The compactness of $\BLinfty$ follows from the compactness of $\Lukinfty$ and the decomposition of standard $\BL$-chains as ordinal sums of {\L}ukasiewicz, product, and G\"{o}del components \cite[Theorem~2.16]{metcalfe+olivetti+gabbay08}. The key observation is that if a finite set $\Gamma$ is satisfiable in the standard product chain $[0, 1]_{\Product}$ or the standard G\"{o}del chain $[0, 1]_{\mathrm{G}}$, then it is satisfiable in the standard {\L}ukasiewicz chain $[0, 1]_{\Luk}$.} Such logics provide a good motivation for taking care to formulate our theory for compact rather than just finitary logics.


\subsection{Running examples: normal modal logics}
\label{subsec: running examples modal}

  Just like substructural logics are axiomatic extensions of $\FL$, classical normal modal logics are axiomatic extensions of the basic global modal logic $\K$, and intuitionistic modal logics are axiomatic extensions of its intuitionistic counter\-part~$\IK$. Although many such logics enjoy a relational (Kripke) semantics, we shall define the consequence relations of $\K$ and $\IK$ in terms of their algebraic semantics given by what we call modal algebras and modal Heyting algebras.

  An introduction to classical modal logic may be found in~\cite{chagrov+zakharyashchev97, blackburn+derijke+venema02}. For an overview of different approaches to intuitionistic modal logic, the reader may consult~\cite{simpson94}. Our motivation for considering intuitionistic modal logics is to show how easy it is to extend the existing results for classical modal logics given the machinery introduced here. However, no familiarity with intuitionistic modal logic will be required: throughout the paper the reader may easily substitute a corresponding classical modal logic whenever we discuss an intuitionistic one.

  We define modal algebras as Boolean algebras with a unary operator $\Box$ which satisfies the equations $\Box (x \wedge y) \equals \Box x \wedge \Box y$ and $\Box \top \equals \top$. Modal~Heyting algebras are Heyting algebras equipped unary operators $\Box$ and $\Diamond$ which satisfy the following equations:
\begin{align*}
  \Box (x \wedge y) & \equals \Box x \wedge \Box y & \Box \top & \equals \top & \Box (x \rightarrow y) & \inequals \Diamond a \rightarrow \Diamond b \\
  \Diamond (x \vee y) & \equals \Diamond x \vee \Diamond y & \Diamond \bot & \equals \bot & \Diamond x \rightarrow \Box y & \inequals \Box (x \rightarrow y)
\end{align*}
  Each modal algebra is a modal Heyting algebra if we take $\Diamond x \assign \neg \Box \neg x$. However, the two operators are in general not interdefinable in modal Heyting algebras.

  Let us remark here that although the axiom $(\Diamond p \rightarrow \Box p) \rightarrow \Box (p \rightarrow q)$ is usually taken to be part of the basic intuitionistic modal logic, we do not use this axiom anywhere. In other words, we might even take our basic intuitionistic modal logic to be slightly weaker and drop the inequality $\Diamond x \rightarrow \Box y \inequals \Box (x \rightarrow y)$ from the above definition of a modal Heyting algebra.

  Throughout the paper we use the notation
\begin{align*}
  \Box_{n} x & \assign x \wedge \Box x \wedge \cdots \wedge \Box^{n} x, \\
  \Diamond_{n} x & \assign x \vee \Diamond x \vee \cdots \vee \Diamond^{n} x,
\end{align*}
  where $\Box^{n}$ and $\Diamond^{n}$ denote a string of $n$ boxes or diamonds. The~varieties of $\IKnfour$-algebras, $\IKnfourfive$-algebras, and $\ISfour$-algebras are defined by imposing the following equations on modal Heyting algebras for~$n \in \omega$:
\begin{itemize}
\item $\IKnfour$-algebras: $\Box_{n} x \inequals \Box_{n+1} x$ (weak $n$-transitivity),
\item $\IKnfourfive$-algebras: $\IKnfour$-algebras with $1 \equals x \vee \Box_{1} \neg \Box_{n} x$ ($n$-cyclicity),
\item $\ISfour$-algebras: $\Box x \inequals x$ (reflexivity) and $\Box x \inequals \Box \Box x$ (transitivity).
\end{itemize}
  $\Knfour$-algebras, $\Knfourfive$-algebras, and $\Sfour$-algebras are the intersections of the above varieties with the variety of modal (Boolean) algebras. Moreover, $\Sfive$-algebras are $\Sfour$-algebras which satisfy $x \inequals \Box \Diamond x$. The dual equations $\Diamond_{n+1} x \inequals \Diamond_{n} x$, $x \inequals \Diamond x$, and $\Diamond \Diamond x \inequals \Diamond x$ may be added to the above definitions of $\IKnfour$ and $\ISfour$ without changing any of our results.

  The consequence relation of the logic $\IK$ is defined in terms of the equational consequence relation $\vDash_{\IKalg}$ of modal Heyting algebras as follows:
\begin{align*}
  \Gamma \vdash_{\IK} \varphi & \iff \set{1 \equals \gamma}{\gamma \in \Gamma} \vDash_{\IKalg} 1 \equals \varphi.
\end{align*}
  Analogous equivalences relate the other intuitionistic modal logics introduced above with the corresponding classes of modal Heyting algebras. These logics are, in the terminology of abstract algebraic logic, algebraizable with their algebraic counterparts being the varieties of the same name.

  Although we shall not need to use Kripke semantics in this paper, for the benefit of the reader let us recall~that
\begin{itemize}
\item $\Sfour$ is the global logic of reflexive and transitive Kripke frames, and
\item $\Sfive$ is the global logic of reflexive, transitive, and symmetric Kripke frames.
\end{itemize}
 To define the Kripke semantics for $\Knfour$, let us introduce the notation $R^{n}$ for the $n$-fold relational composition of a binary relation $R$ on a set $W$: $R^{0}$ is the equality relation on~$W$ and $R^{n+1} = R^{n} \circ R$. We then use $R^{\smallleq n}$ to denote the relation $\bigcup_{0 \leq i \leq n} R^{i}$, i.e.\ the relation of being accessible in at most $n$ steps. The~relation $R$ will be called weakly $n$-transitive if $R^{\smallleq (n+1)} \subseteq R^{\smallleq n}$, and $n$-cyclic if $u R w$ implies $w R^{\smallleq n} u$. Then
\begin{itemize}
\item $\Knfour$ is the logic of weakly $n$-transitive Kripke frames.
\item $\Knfourfive$ is the logic of $n$-cyclic weakly $n$-transitive Kripke frames.
\end{itemize}
  Similar claims could be made for the intuitonistic versions of these logics.

\subsection{Notational conventions}
\label{subsec: notational conventions}

  We now review some notational conventions employed throughout the paper. For a given logic $\logic{L}$, $\kappa$ ranges over all cardinals such that ${1 < \kappa \leq |{\Var \logic{L}}|^{+}}$, possibly finite, while $\alpha$ ranges over all ordinals $0 < \alpha < \kappa$, and $\tuple{\varphi}$ and $\tuple{p}$ range over $\alpha$-tuples of formulas and variables. A set of formulas in the variables $\tuple{p}, q, \tuple{r}$ is denoted e.g.\ $\ddtset(\tuple{p}, q, \tuple{r})$, in which case $\ddtset(\tuple{\varphi}, \psi, \tuple{\chi})$ denotes the set of formulas obtained by substituting $\tuple{\varphi}$ for $\tuple{p}$, $\psi$ for $q$, and $\tuple{\chi}$ for~$\tuple{r}$. If $\Phi$ is a set of formulas of cardinality at most equal to $|\alpha|$, then $\ddtset(\Phi, \psi, \tuple{\chi})$ abbreviates $\ddtset(\tuple{\varphi}, \psi, \tuple{\chi})$ for some $\alpha$-tuple of formulas $\tuple{\varphi}$ consisting precisely of the elements of~$\Phi$, some of them possibly repeated. For each $\alpha$-tuple of formulas $\tuple{\varphi}$ and each substitution $\sigma$, the tuple obtained by applying $\sigma$ to each element of $\tuple{\varphi}$ is denoted~$\sigma(\tuple{\varphi})$. Analogous conventions hold for $\alpha$-tuples $\tuple{a}$ of elements of any given algebra. We generally avoid explicitly stating the lengths of tuples to avoid cluttering the text with technicalities which can be easily inferred from the context.

\section{Semisimplicity and the law of the excluded middle}
\label{sec: ils and ddts}

  The aim of this section is to show that under suitable conditions, semisimplicity and the law of the excluded middle (LEM) are equivalent. Theorem~\ref{thm: semisimplicity} provides the syntactic version of this equivalence, while Theorem~\ref{thm: semantic semisimplicity} provides a semantic version, which is then used in Subsection~\ref{subsec: applications} to describe the semisimple varieties of $\FLe$-algebras and modal Heyting algebras.

  We first introduce the syntactic principles which will play a crucial role in our frame\-work: deduction--detachment theorems (DDTs), inconsistency lemmas (ILs), dual inconsistency lemmas (dual~ILs), and simple inconsistency lemmas (simple ILs). Dual ILs in fact turn out to be nothing but LEMs in disguise. As~the reader will readily observe, these principles are preserved under axiomatic extensions as well as linguistic reducts, provided that the reducts contain all the syntactic material required to express the principle in question. Our methods will be entirely syntactic, with the exception of results which relate syntactic and semantic semisimplicity (Subsection~\ref{subsec: semantic semisimplicity}).

  Although we only investigate DDTs, ILs, and dual ILs from a syntactic point of view, the reader should be aware that these principles also have natural semantic correlates. For~DDTs such correlates are well known, as the theory of DDTs is a classical part of abstract algebraic logic (see e.g.~\cite{czelakowski01,font16}). By~contrast, the theory of ILs has not been systematically investigated so far, with the exception of the global IL studied by Raftery~\cite{raftery13inconsistency}. Since developing this theory is outside the scope of the present paper, let us merely remark here that a theory of ILs parallel to the existing theory of DDTs can be built where the known semantic characterizations of so-called global, local, and parametrized local DDTs have analogues usually obtained by restricting to simple theories or models at the appropriate places.

  The facts about DDTs stated here are well known and their proofs can be found in the literature (see again~\cite{czelakowski01,font16}), although we extend the standard framework slightly by allowing for infinitary DDTs. The facts about ILs are not present in the literature. However, many of them are extensions of the results of Raftery~\cite{raftery13inconsistency} on global~ILs in finitary protoalgebraic logics to a more general setting. In~such cases we provide a reference to the corresponding proposition in~\cite{raftery13inconsistency}. Some of these extensions are straightforward, while in other cases some subtleties arise. Moreover, while our proofs are invariable syntactic in nature, some of Raftery's proofs are semantic in nature and involve protoalgebraicity.

\subsection{Deduction theorems}
\label{subsec: ddts}

  A deduction--detachment theorem (DDT) for a given logic is, informally speaking, an equivalence which allows us to move formulas from the left of the turnstile to the right in a manner analogous to the following equivalence for classical logic:
\begin{align*}
  \Gamma, \varphi \vdash_{\CL} \psi \iff \Gamma \vdash_{\CL} \varphi \rightarrow \psi.
\end{align*}
  In finitary logics, such an equivalence allows us to reduce questions of validity to questions of theoremhood.

  Many familiar logics enjoy a DDT.  For example, the above equivalence holds for intuitionistic logic, while a more general one holds for the $n$-contractive Full~Lambek calculus with exchange $\FLen$ and its axiomatic extensions:
\begin{align*}
  \Gamma, \varphi \vdash_{\FLen} \psi \iff \Gamma \vdash_{\FLen} (1 \wedge \varphi)^{n} \rightarrow \psi.
\end{align*}
  A similar equivalence holds for the global modal logic $\IKnfour$ and its axiomatic extensions, including the classical modal logic $\Kfour$:
\begin{align*}
  \Gamma, \varphi \vdash_{\IKnfour} \psi \iff \Gamma \vdash_{\IKnfour} \Box_{n} \varphi \rightarrow \psi.
\end{align*}
  The above equivalences all involve a single formula schema to the right of the turnstile, no matter what $\Gamma$, $\varphi$, or $\psi$ are. Such DDTs are called global DDTs.

  A more complicated DDT holds for the Full~Lambek calculus with exchange $\FLe$ and its axiomatic extensions:
\begin{align*}
  \Gamma, \varphi \vdash_{\FLe} \psi \iff \Gamma \vdash_{\FLe} (1 \wedge \varphi)^{n} \rightarrow \psi \text{ for some } n \in \omega.
\end{align*}
  For $\FLew$ and its axiomatic extensions this simplifies to:
\begin{align*}
  \Gamma, \varphi \vdash_{\FLew} \psi \iff \Gamma \vdash_{\FLew} \varphi^{n} \rightarrow \psi \text{ for some } n \in \omega.
\end{align*}
  Similarly, the global modal logic $\IK$ and its axiomatic extensions, including the classical modal logic $\K$, enjoy the following equivalence:
\begin{align*}
  \Gamma, \varphi \vdash_{\IK} \psi \iff \Gamma \vdash_{\IK} \Box_{n} \varphi \rightarrow \psi \text{ for some } n \in \omega.
\end{align*}
  Recall that $\Box_{n} x$ stands for $x \wedge \Box x \wedge \cdots \wedge \Box^{n} x$, where $\Box^{n}$ abbreviates a string of $n$ boxes. Such equivalences, which feature a whole family of formulas (or more generally sets of formulas) to the right of the turnstile, are called local~DDTs.

  Some logics, such as the Full Lambek calculus~$\FL$ do not enjoy even a local form of the DDT, but they satisfy a still more general form of the DDT, called a para\-metrized local DDT (see~\cite[Chapter~2.4]{residuated07}). Other variants of the DDT not considered in the present paper also exist, such as so-called contextual deduction theorems~\cite{raftery11}. Having sketched some of the most important examples of DDTs, let us now state the relevant definitions properly.

\begin{definition} \label{def: ddt}
  A logic $\logic{L}$ enjoys the \emph{$\kappa$-ary parametrized local DDT}, where $\kappa \leq | {\Var \logic{L}} |^{+}$, if for each ordinal $0 < \alpha < \kappa$ there is a family of sets $\ddtfamily_{\alpha}(\tuple{p}, q, \tuple{r})$, where $\tuple{p}$ is an $\alpha$-tuple of variables, such that for each set of formulas $\Gamma$ and each $\alpha$-tuple of formulas $\tuple{\varphi}$:
\begin{align*}
  \Gamma, \tuple{\varphi} \vdash_{\logic{L}} \psi & \iff \Gamma \vdash_{\logic{L}} \ddtset(\tuple{\varphi}, \psi, \tuple{\pi}) \text{ for some } \ddtset(\tuple{p}, q, \tuple{r}) \in \ddtfamily_{\alpha}(\tuple{p}, q, \tuple{r}) \text{ and some } \tuple{\pi}.
\end{align*}
  It enjoys the \emph{$\kappa$-ary local DDT} if each $\ddtfamily_{\alpha}$ can be chosen of the form $\ddtfamily_{\alpha}(\tuple{p}, q)$. It~enjoys the \emph{$\kappa$-ary global DDT} if each $\ddtfamily_{\alpha}$ can be chosen of the form $\{ \ddtset(\tuple{p}, q) \}$.
\end{definition}

  The whole collection of families $\ddtfamily_{\alpha}$ will be referred to as the DDT family~$\ddtfamily$. Moreover,
\begin{itemize}
\item by the \emph{unary DDT} we mean the $2$-ary DDT,\footnote{In other words, the unary DDT allows us to transport a single formula from the left of the turnstile to the right. One might alternatively tweak the definition of the $\kappa$-ary DDT in case $\kappa$ is finite to make sure that it is the $1$-ary DDT rather than the $2$-ary DDT which transports a single formula across the turnstile. We opt against doing so as it would be rather tedious to split the definition of each syntactic principle into two cases depending on whether $\kappa$ is finite.}
\item by the \emph{finitary DDT} we mean the $\omega$-ary DDT, and
\item by the \emph{unrestricted DDT} we mean the $\kappa$-ary DDT for $\kappa = | {\Var \logic{L}} |^{+}$.
\end{itemize}
  The technical condition ${\kappa \leq | {\Var \logic{L}} |^{+}}$, which accounts for the fact that there are no tuples of distinct variables of length~$| {\Var \logic{L}} |^{+}$, will be always be assumed when we talk about $\kappa$-ary DDTs, or indeed any other $\kappa$-ary syntactic principles.\footnote{If $| {\Fm \logic{L}} | = | {\Var \logic{L}} |$, which is the case for most ``everyday'' logics, then there are no tuples of distinct formulas of length greater than $| {\Var \logic{L}} |$ anyway.}

  For a finitary ($\kappa$-ary) logic we may assume without loss of generality that each of the families $\ddtfamily_{\alpha}$ consists of finite sets (sets of cardinality less than~$\kappa$) only. A logic which enjoys a unary global (local, parametrized local) DDT in fact enjoys the finitary global (local, parametrized local) local DDT.\footnote{This is trivial for global and local DDTs. For parametrized local DDTs the presence of parameters complicates things, but the implication holds by Fact~\ref{fact: protoimplication}.} Moreover, a $\kappa$-ary logic which enjoys a $\kappa$-ary (parametrized) local DDT in fact enjoys the unrestricted (parametrized) local DDT. 

  Only unary (or equivalently, finitary) DDTs are usually considered in the literature. This omission is entirely justified when it comes to global DDTs in a setting where each connective is finitary: the~reader may easily observe that even classical logic fails to enjoy the infinitary global DDT. However, as a consequence of finitarity, classical logic does enjoy the unrestricted local DDT in the form
\begin{align*}
  \Gamma, \tuple{\varphi} \vdash_{\CL} \psi & \iff \Gamma \vdash_{\CL} \left( \bigwedge \tuple{\varphi}' \right) \rightarrow \psi \text{ for some finite subtuple } \tuple{\varphi}' \text{ of } \tuple{\varphi}.
\end{align*}
  More interestingly, the infinitary {\L}ukasiewicz logic $\Lukinfty$ enjoys a more complicated form of the unrestricted local DDT. Since ILs and dual ILs are instrumental in obtaining the infinitary local DDT of $\Lukinfty$, we postpone the formulation of this unrestricted DDT until after these have been introduced (see~Example~\ref{ex: infinitary ddt for lukinfty}).

  Our reasons for considering infinitary DDTs are three-fold. Firstly, we gain some generality with no corresponding increase in the complexity of our proofs. Secondly, building an arity into the definition of a DDT makes the parallels between the theory of DDTs and the theory of ILs more explicit. A unary IL need not induce a finitary IL, therefore one is led to consider ILs of different arities anyway, and no substantial simplicity is gained by restricting to finitary ILs. Finally, for infinitary protoalgebraic logics it is the unrestricted DDT rather than the finitary DDT which corresponds to a natural semantic condition, namely the Filter Extension Property (see~\cite[Proposition~6.95]{lavicka18thesis}).

  Although we did not provide an example of a parametrized local DDT in action, let us recall here that logics with such a DDT admit a more practical characterization in terms of the existence of a set of formulas which satisfies the axiom of Reflexivity and the rule of Modus Ponens. Most ``everyday'' logics with an implication connective therefore enjoy at least this form of the DDT.

\begin{fact} \label{fact: protoimplication}
  The following are equivalent for each logic $\logic{L}$:
\begin{enumerate}
\item $\logic{L}$ enjoys the unary parametrized local DDT.
\item $\logic{L}$ enjoys the unrestricted parametrized local DDT.
\item $\logic{L}$ has a protoimplication set, i.e.\ a set of formulas $\protoset(p, q)$ in two variables such that $\emptyset \vdash_{\logic{L}} \protoset(p, p)$ and $p, \protoset(p, q) \vdash_{\logic{L}} q$.
\end{enumerate}
\end{fact}

\begin{proof}
  The equivalence between (1) and (3) is well known, see e.g.\ Theorems~6.7 and 6.22 of \cite{font16}. The equivalence between conditions (2)~and~(3) is not stated explicitly there, but it holds by a trivial modification of the proof of Theorem 6.22 of \cite{font16}.
\end{proof}

  If a logic enjoys a parametrized local DDT w.r.t.\ two families $\ddtfamily$ and~$\ddtfamily'$, then these are equivalent in the following sense: for each $\ddtset(\tuple{p}, q, \tuple{r}) \in \ddtfamily_{\alpha}$ there are $\ddtset'(\tuple{p}, q, \tuple{r}') \in \ddtfamily'$ and $\tuple{\pi}'$ such that $\ddtset(\tuple{p}, q, \tuple{r}) \vdash_{\logic{L}} \ddtset'(\tuple{p}, q, \tuple{\pi}')$, and conversely for each $\ddtset'(\tuple{p}, q, \tuple{r}') \in \ddtfamily'_{\alpha}$ there are $\ddtset(\tuple{p}, q, \tuple{r}) \in \ddtfamily$ and $\tuple{\pi}$ such that $\ddtset'(\tuple{p}, q, \tuple{r}') \vdash_{\logic{L}} \ddtset(\tuple{p}, q, \tuple{\pi})$.

\begin{fact} \label{fact: all ddt families equivalent}
  Let $\logic{L}$ be a logic which enjoys the $\kappa$-ary parametrized local DDT w.r.t.\ a family $\ddtfamily$. Then $\logic{L}$ enjoys the $\kappa$-ary parameterized local DDT w.r.t.\ a family $\ddtfamily'$ if and only if $\ddtfamily$ and $\ddtfamily'$ are equivalent.
\end{fact}

\begin{proof}
  We omit the proof of this known fact, since an analogous fact will be proved for~ILs.
\end{proof}

  An equivalent formulation of the (parametrized) local DDT is what we call the (surjective) substitution swapping property.

\begin{definition} \label{def: swapping}
  A logic $\logic{L}$ enjoys \emph{$\kappa$-ary (surjective) substitution swapping}, where $\kappa \leq | {\Var \logic{L}} |^{+}$, if for each ordinal $0 < \alpha < \kappa$, each (surjective) substitution~$\sigma$, each $\alpha$-tuple of formulas $\tuple{\varphi}$, each~$\psi$, and each $\logic{L}$-theory~$T$
\begin{align*}
  \sigma(\tuple{\varphi}), T \vdash_{\logic{L}} \sigma(\psi) \implies \tuple{\varphi}, \sigma^{-1}[T] \vdash_{\logic{L}} \psi.
\end{align*}
\end{definition}

\begin{proposition} \label{prop: swapping}
  A logic $\logic{L}$ enjoys the $\kappa$-ary (parametrized) \mbox{local} DDT if and only if it enjoys $\kappa$-ary (surjective) substitution swapping, for $\kappa \leq | {\Var \logic{L}} |^{+}$.
\end{proposition}

\begin{proof}
  We again omit the proof of this known fact (see e.g.~\cite[Section~6.2]{font16}), since an analogous fact will be proved for~ILs.
\end{proof}

  Although in this paper we shall not be concerned with the semantic correlates of the various forms of the DDT, let us briefly state what they are for the benefit of the algebraically minded reader. If $\logic{L}$ is an algebraizable logic and $\class{K}$ is its equivalent algebraic semantics, then $\logic{L}$ always enjoys the parametrized local DDT. Moreover, modulo some finitarity and finiteness assumptions, it enjoys the local DDT if and only if $\class{K}$ enjoys the Relative Congruence Extension Property (RCEP), and it enjoys the global DDT if and only if $\class{K}$ has Equationally Definable Relative Principal Congruences (EDPRC).

\subsection{Inconsistency lemmas}
\label{subsec: ils}

  While DDTs allow us to move formulas from the left of the turnstile to the right in any context, inconsistency lemmas (ILs) have a more limited scope: they allow us to do so if the premises are inconsistent. That is, they generalize the following equivalence for classical logic:
\begin{align*}
  \Gamma, \varphi_{1}, \dots, \varphi_{k} \vdash_{\CL} \allset \iff \Gamma \vdash_{\logic{L}} \neg (\varphi_{1} \wedge \dots \wedge \varphi_{k}).
\end{align*}
  This is essentially the DDT for $\CL$ with the \emph{falsum} constant in place of~$\psi$. In~compact logics, such an equivalence allows us to reduce questions of inconsistency to questions of validity. It will come as no surprise to the reader that such equivalences also come in global, local, and parametrized local forms.

  Each logic with a DDT and a \emph{falsum} constant representing an inconsistent proposition enjoys an IL. Fragments of such logics which drop the implication but retain the negation provide examples of logics which enjoy the IL but not a DDT. For example, (the conjunction--negation--product--unit fragment of) $\FLen$ and its axiomatic extensions enjoy the IL in the form
\begin{align*}
  \Gamma, \varphi_{1}, \dots, \varphi_{k} \vdash_{\FLen} \allset & \iff \Gamma \vdash_{\FLen} \neg (1 \wedge \varphi_{1} \wedge \dots \wedge \varphi_{k})^{n},
\end{align*}
  while (the negation--conjunction--box fragment of) $\IKnfour$ and its axiomatic extensions enjoy the IL in the form
\begin{align*}
  \Gamma, \varphi_{1}, \dots, \varphi_{k} \vdash_{\IKnfour} \allset & \iff \Gamma \vdash_{\IKnfour} \neg \Box_{n} (\varphi_{1} \wedge \dots \wedge \varphi_{k}).
\end{align*}
  The local form of the IL is analogous to the local form of the DDT. For~example, in $\FLe$ and its axiomatic extensions, it takes the form
\begin{align*}
  \Gamma, \varphi_{1}, \dots, \varphi_{k} \vdash_{\FLe} \allset & \iff \Gamma \vdash_{\FLe} \neg (1 \wedge \varphi_{1} \wedge \dots \wedge \varphi_{k})^{n} \text{ for some } n \in \omega.
\end{align*}
  Similarly, in $\IK$ and its axiomatic extensions, it takes the form
\begin{align*}
  \Gamma, \varphi_{1}, \dots, \varphi_{k} \vdash_{\IK} \allset & \iff \Gamma \vdash_{\IK} \neg \Box_{n} (\varphi_{1} \wedge \dots \wedge \varphi_{k}) \text{ for some } n \in \omega.
\end{align*}

  A substructural logic may well enjoy a ``better'' IL than the one which comes from its DDT. For instance, the~product fuzzy logic~$\Product$ only enjoys the local DDT inherited from $\FLew$, but it enjoys the global IL of classical logic:
\begin{align*}
  \Gamma, \varphi_{1}, \dots, \varphi_{k} \vdash_{\Product} \allset & \iff \Gamma \vdash_{\Product} \neg (\varphi_{1} \wedge \dots \wedge \varphi_{k}).
\end{align*}
  A different example is the infinitary {\L}ukasiewicz logic $\Lukinfty$, which enjoys a complicated local DDT (see Example~\ref{ex: ddt for lukinfty}), but by compactness it inherits the fairly simple unrestricted local IL of the finitary {\L}ukasiewicz logic $\Luk$ (i.e.\ of $\FLew$). These are both local principles, but the local IL enjoyed by $\Lukinfty$ is substantially simpler than the local DDT enjoyed by $\Lukinfty$.

\begin{definition} \label{def: il}
  A logic $\logic{L}$ enjoys the \emph{$\kappa$-ary parametrized local IL}, where $\kappa \leq | {\Var \logic{L}} |^{+}$, if for each ordinal $0 < \alpha < \kappa$ there is a family of sets $\ilfamily_{\alpha}(\tuple{p}, \tuple{r})$, where $\tuple{p}$ is an $\alpha$-tuple of variables, such that for each set of formulas $\Gamma$ and each $\alpha$-tuple of formulas~$\tuple{\varphi}$:
\begin{align*}
  \Gamma, \tuple{\varphi} \vdash_{\logic{L}} \Fm \logic{L} & \iff \Gamma \vdash_{\logic{L}} \ilset_{\alpha}(\tuple{\varphi}, \tuple{\pi}) \text{ for some } \ilset(\tuple{p}, \tuple{r}) \in \ilfamily_{\alpha} \text{ and some } \tuple{\pi}.
\end{align*}
  It enjoys the \emph{$\kappa$-ary local IL} if each $\ilfamily_{\alpha}$ can be chosen of the form $\ilfamily_{\alpha}(\tuple{p})$. It~enjoys the \emph{$\kappa$-ary global IL} if each~$\ilfamily_{\alpha}$ can be chosen of the form $\{ \ilset(\tuple{p}) \}$.
\end{definition}

\begin{fact}
  Each logic with the parametrized local IL has an antitheorem.
\end{fact}

\begin{proof}
  Consider any $\ilset \in \ilfamily_{1}$. Then $\ilset(p, \tuple{p}) \vdash_{\logic{L}} \ilset(p, \tuple{p})$, where $\tuple{p}$ is a tuple consisting entirely of copies of the variable $p$, therefore $p, \ilset(p, \tuple{p}) \vdash_{\logic{L}} \Fm \logic{L}$. It~follows that $p, \ilset(p, \tuple{p})$ is an antitheorem.
\end{proof}

  The remarks which follow the definition of DDTs (Definition~\ref{def: ddt}) also apply, \emph{mutatis mutandis}, to ILs. In particular,
\begin{itemize}
\item by the \emph{unary IL} we mean the $2$-ary IL,
\item by the \emph{finitary IL} we mean the $\omega$-ary IL, and
\item by the \emph{unrestricted IL} we mean the $\kappa$-ary IL for $\kappa = | {\Var \logic{L}} |^{+}$.
\end{itemize}
  The condition $\kappa \leq | {\Var \logic{L}} |^{+}$ will again be assumed when talking about $\kappa$-ary ILs.

  If a logic is $\kappa$-compact, then we may assume without loss of generality that the families $\ilfamily_{\alpha}$ consist of sets of cardinality less than~$\kappa$. Each unary IL induces a finitary IL provided that the logic in question has a conjunction: a~set of formulas $\bigwedge(p, q)$ in two variables such that $\bigwedge(p, q) \vdash_{\logic{L}} p$ and $\bigwedge(p, q) \vdash_{\logic{L}} q$ and $p, q \vdash_{\logic{L}} \bigwedge(p, q)$. Moreover, a $\kappa$-compact logic with a $\kappa$-ary (parametrized) local IL in fact enjoys the unrestricted (parametrized) local IL. 

  If a logic enjoys the IL w.r.t.\ two families, then the two families are in a suitable sense equivalent (see~\cite[Remark~3.2]{raftery13inconsistency}). Namely, we say that two families of sets of formulas $\ilfamily_{\alpha}(\tuple{p}, \tuple{q})$ and $\ilfamily'_{\alpha}(\tuple{p}, \tuple{q})$ for $0 < \alpha < \kappa$ are equivalent if for each $\ilset \in \ilfamily_{\alpha}$ there are $\ilset' \in \ilfamily'_{\alpha}$ and $\tuple{\pi}'$ such that $\ilset(\tuple{p}, \tuple{q}) \vdash_{\logic{L}} \ilset'(\tuple{p}, \tuple{\pi}')$, and conversely for each $\ilset' \in \ilfamily'_{\alpha}$ there are $\ilset \in \ilfamily_{\alpha}$ and $\tuple{\pi}$ such that $\ilset'(\tuple{p}, \tuple{q}) \vdash_{\logic{L}} \ilset(\tuple{p}, \tuple{\pi})$.

  Observe that if $\ilfamily$ is a local inconsistency family, i.e.\ it only consists of sets of the form $\ilset(\tuple{p})$, and $\ilfamily'$ is equivalent to $\ilfamily$, then $\ilfamily'$ can be transformed into a local inconsistency family replacing each $\ilset'(\tuple{p}, \tuple{q}) \in \ilfamily'_{\alpha}$ by $\ilset'(\tuple{p}, \sigma(\tuple{\pi}'))$ where $\tuple{\pi}'$ is the tuple whose existence is asserted by the equivalence between $\ilfamily$ and $\ilfamily'$ and $\sigma$ is any substitution such that $\sigma(\tuple{p}) = \tuple{p}$ and $\sigma(\tuple{q}) \subseteq \tuple{p}$. Moreover, if $\ilfamily$ is a global inconsistency family, i.e.\ it consists of a single set of the form $\ilset(\tuple{p})$, then $\ilfamily'$ may also be restricted to a single set of the form $\ilset(\tuple{p})$.

\begin{fact} \label{fact: all il families equivalent}
  Let $\logic{L}$ be a logic which enjoys the $\kappa$-ary parametrized local IL w.r.t. a family~$\ilfamily$. Then $\logic{L}$ enjoys the $\kappa$-ary parametrized local IL w.r.t.\ a family~$\ilfamily'$ if and only if $\ilfamily$ and $\ilfamily'$ are equivalent.
\end{fact}

\begin{proof}
  Right to left, if $\Gamma \vdash_{\logic{L}} \ilset(\tuple{\varphi}, \tuple{\pi})$ for $\ilset \in \ilfamily_{\alpha}$, then by structurality $\ilset(\tuple{\varphi}, \tuple{\pi}) \vdash_{\logic{L}} \ilset'(\tuple{\varphi}, \sigma(\tuple{\pi}'))$ for some $\ilset' \in \ilfamily'_{\alpha}$ where $\sigma$ is a substitution such that $\sigma(\tuple{p}) = \tuple{\varphi}$ and $\sigma(\tuple{q}) = \tuple{\pi}$. Thus $\Gamma \vdash_{\logic{L}} \ilset'(\tuple{\varphi}, \sigma(\tuple{\pi}'))$ by cut. Conversely, there is a substitution $\sigma'$ such that $\Gamma \vdash_{\logic{L}} \ilset'(\tuple{\varphi}, \tuple{\pi}')$ for $\ilset' \in \ilfamily_{\alpha}$ implies $\Gamma \vdash_{\logic{L}} \ilset(\tuple{\varphi}, \sigma'(\tuple{\pi}))$ for some $\ilset \in \ilfamily_{\alpha}$. The conditions $\Gamma \vdash_{\logic{L}} \ilset(\tuple{\varphi}, \tuple{\pi})$ for some $\tuple{\pi}$ and $\Gamma \vdash_{\logic{L}} \ilset'(\tuple{\varphi}, \tuple{\pi}')$ for some $\tuple{\pi}'$ are therefore equivalent.

  Left to right, by the IL $\ilset(\tuple{p}, \tuple{q}) \vdash_{\logic{L}} \ilset(\tuple{p}, \tuple{q})$ for $\ilset \in \ilfamily_{\alpha}$ implies $\tuple{p}, \ilfamily(\tuple{p}, \tuple{q}) \vdash_{\logic{L}} \allset$, hence by the IL there are $\ilset' \in \ilfamily'_{\alpha}$ and $\tuple{\pi}'$ such that $\ilfamily(\tuple{p}, \tuple{q}) \vdash_{\logic{L}} \ilset'(\tuple{p}, \tuple{\pi}')$. The~other implication is proved by switching the roles of $\ilfamily$ and $\ilfamily'$.
\end{proof}

  Each local DDT induces a corresponding local IL, provided that the logic in question has an antitheorem (see~\cite[Corollary~3.9]{raftery13inconsistency}). In the parametrized case, this holds if the antitheorem is small enough.

\begin{fact}
  Each logic which enjoys the $\kappa$-ary global (local, para\-metrized local) DDT also enjoys the $\kappa$-ary global (local, parametrized local)~IL, provided that it has an antitheorem (in the parametrized local~case, of cardinality at most $| {\Var \logic{L}} |$).
\end{fact}

\begin{proof}
  If a logic has an antitheorem, then by structurality for antitheorems it has an antitheorem of cardinality at most $| {\Var \logic{L}} |$ in one variable only, say $\antitheorem(p_{1})$ for $p_{1} \in \tuple{p}$. Suppose that $\logic{L}$ enjoys the parametrized local DDT w.r.t.~$\ddtfamily$. We define the family $\ilfamily_{\alpha}$ as follows:
\begin{align*}
  \ilfamily_{\alpha} & \assign \set{\ilset_{f}}{f\colon \antitheorem \to \ddtfamily_{\alpha}}, \text{ where} \\
  \ilset_{f}(\tuple{p}, \tuple{r}) & \assign \bigcup_{\antiformula \in \antitheorem} f(\antiformula)(\tuple{p}, \antiformula(p_{1}), \tuple{r}_{\antiformula}),
\end{align*}
  where $\tuple{r}$ obtained by collecting the tuples $\tuple{r}_{\antiformula}$ into a single tuple, and $\tuple{r}_{\antiformula}$ are tuples of distinct variables such that $\tuple{r}_{\antiformula}$ and $\tuple{r}_{\antiformula'}$ are disjoint for distinct $\antiformula$ and $\antiformula'$, and they do not share any variables with~$\tuple{p}$. In the parametrized local case, such tuples exist by the assumption on the cardinality of~$\antitheorem$.

  Then ${\tuple{\varphi}, \ilset_{f}(\antiformula)(\tuple{\varphi}, \antiformula(\varphi_{1}), \tuple{r}_{\antiformula}) \vdash_{\logic{L}} \antiformula(\varphi_{1})}$ for each $\antiformula \in \antitheorem$ by the DDT, therefore ${\tuple{\varphi}, \ilset_{f}(\tuple{\varphi}, \tuple{\pi}) \vdash_{\logic{L}} \antitheorem(\varphi_{1})}$. But by structurality for antitheorems $\antitheorem(\varphi_{1}) \vdash_{\logic{L}} \allset$, hence $\tuple{\varphi}, \ilset_{f}(\tuple{\varphi}, \tuple{\pi}) \vdash_{\logic{L}} \allset$. Conversely, if $\Gamma, \tuple{\varphi} \vdash_{\logic{L}} \allset$, then $\Gamma, \tuple{\varphi} \vdash_{\logic{L}} \antitheorem(\varphi_{1})$ and by the DDT for each $\antiformula \in \antitheorem$ there are $\ddtset \in \ddtfamily$ and $\tuple{\rho}_{\antiformula}$ such that $\Gamma \vdash_{\logic{L}} \ddtset(\tuple{\varphi}, \antiformula(\varphi_{1}), \tuple{\rho}_{\antiformula})$. Therefore $\Gamma \vdash_{\logic{L}} \ilset_{f}(\tuple{\varphi}, \tuple{\rho})$, where $f\colon \antitheorem \to \ddtfamily_{\alpha}$ is the function which assigns the set $\ddtset$ to $\antiformula$ and $\tuple{\rho}$ is a tuple obtained by collecting all the tuples $\tuple{\rho}_{\antiformula}$ into a single tuple. Here we use the assumption about the disjointness of the tuples $\tuple{r}_{\antiformula}$.

  If no parameters occur in $\ddtfamily$, then the assumption on the cardinality of $\ddtfamily_{\alpha}$ for $0 < \alpha < \kappa$ is not needed, and the resulting family $\ilfamily$ also does not contain parameters. Moreover, if $\ddtfamily_{\alpha}$ is a singleton, so is $\ilfamily_{\alpha}$.
\end{proof}

  The technical assumption about the cardinality of an antitheorem is satisfied whenever $| {\Fm \logic{L}} | = | {\Var \logic{L}} |$, which is true for most ``everyday'' logics. However, the assumption will fail e.g.\ in a logic which has countable many variables but each antitheorem contains uncountably many constants.

  The (surjective) substitution swapping property, which is equivalent to the (para\-metrized) local DDT, has a direct analogue for (parametrized) local ILs.

\begin{definition} \label{def: swapping for antitheorems}
  A logic $\logic{L}$ with an antitheorem enjoys \emph{$\kappa$-ary (surjective) sub\-stitution swapping for antitheorems}, where $\kappa \leq |{\Var \logic{L}}|^{+}$, if for each ordinal $0 < \alpha < \kappa$, each (surjective) substitution~$\sigma$, each $\alpha$-tuple of formulas $\tuple{\varphi}$, and~each theory $T$ of $\logic{L}$
\begin{align*}
  \sigma(\tuple{\varphi}), T \vdash_{\logic{L}} \allset \implies \tuple{\varphi}, \sigma^{-1}[T] \vdash_{\logic{L}} \allset.
\end{align*}
\end{definition}

\begin{proposition} \label{prop: swapping for antitheorems}
  A logic $\logic{L}$ enjoys the $\kappa$-ary (parametrized) \mbox{local} IL if and only if it has an antitheorem and it enjoys $\kappa$-ary (surjective) substitution swapping for antitheorems, provided that $\kappa \leq | {\Var \logic{L}} |^{+}$.
\end{proposition}

\begin{proof}
  Suppose that $\logic{L}$ enjoys the parametrized local IL. Then $\logic{L}$ has an antitheorem. Now suppose that $\sigma(\tuple{\varphi}), T \vdash_{\logic{L}} \allset$ for some $\logic{L}$-theory $T$ and some surjective substitution~$\sigma$. By the IL w.r.t.\ a family $\ilfamily$ there are $\ilset \in \ilfamily$ and $\tuple{\pi}$ such that $T \vdash_{\logic{L}} \ilset(\sigma (\tuple{\varphi}), \tuple{\pi})$. Because $\sigma$ is surjective, there are formulas $\tuple{\rho}$ such that $\tuple{\pi} = \sigma(\tuple{\rho})$, therefore $T \vdash_{\logic{L}} \ilset(\sigma (\tuple{\varphi}), \sigma (\tuple{\rho}))$ and $T \vdash_{\logic{L}} \sigma[\ilset(\tuple{\varphi}, \tuple{\rho})]$. But then $\sigma^{-1}[T] \vdash_{\logic{L}} \ilset(\tuple{\varphi}, \tuple{\rho})$ and by the IL $\tuple{\varphi}, \sigma^{-1}[T] \vdash_{\logic{L}} \allset$.

  Conversely, suppose that $\logic{L}$ has an antitheorem and enjoys surjective substitution swapping for antitheorems. We define the family $\ilfamily$ so that
\begin{align*}
  \ilset(\tuple{p}, \tuple{q}) \in \ilfamily_{\alpha} & \iff \tuple{p}, \ilset(\tuple{p}, \tuple{q}) \vdash_{\logic{L}} \allset.
\end{align*}
  If $\Gamma \vdash_{\logic{L}} \ilset(\tuple{\varphi}, \tuple{\pi})$, then $\Gamma, \tuple{\varphi} \vdash_{\logic{L}} \allset$ because $\tuple{\varphi}, \ilset(\tuple{\varphi}, \tuple{\pi}) \vdash_{\logic{L}} \allset$ by structurality for antitheorems. Conversely, suppose that $\Gamma, \tuple{\varphi} \vdash_{\logic{L}} \allset$ and let $\sigma$ be a surjective substitution such that $\sigma(\tuple{p}) = \tuple{\varphi}$. Such a substitution exists because $\tuple{p}$ is an $\alpha$-tuple for some $0 < \alpha < \kappa \leq | {\Var \logic{L}} |^{+}$, i.e.\ $\alpha$ has cardinality at most $| {\Var \logic{L}} |$. Then by surjective substitution swapping for antitheorems $\sigma^{-1}[T], \tuple{p} \vdash_{\logic{L}} \allset$ where $T$ is the $\logic{L}$-theory generated by $\Gamma$, hence $\sigma^{-1}[T] \in \ilset(\tuple{p}, \tuple{q})$. Applying $\sigma$ yields that $T = \sigma[\sigma^{-1}[T]] \in \ilset(\tuple{\varphi}, \sigma(\tuple{q}))$, hence $\Gamma \vdash_{\logic{L}} \ilset(\tuple{\varphi}, \tuple{\pi})$ for $\tuple{\pi} \assign \sigma(\tuple{q})$.

  Now suppose that $\logic{L}$ enjoys the local IL and ${\sigma(\tuple{\varphi}), T \vdash_{\logic{L}} \allset}$ for some $\logic{L}$-theory $T$. By the IL w.r.t.~$\ilfamily$ there is some $\ilset \in \ilfamily$ such that $T \vdash_{\logic{L}} \ilset(\sigma (\tuple{\varphi}))$, i.e.\ $T \vdash_{\logic{L}} \sigma[\ilset(\tuple{\varphi})]$. But then $\sigma^{-1}[T] \vdash_{\logic{L}} \ilset(\tuple{\varphi})$, hence $\tuple{\varphi}, \sigma^{-1}[T] \vdash_{\logic{L}} \allset$.

  Conversely, suppose that $\logic{L}$ enjoys substitution swapping for antitheorems. We define the family $\ilfamily$ so~that
\begin{align*}
  \ilset(\tuple{p}) \in \ilfamily_{\alpha} & \iff \tuple{p}, \ilset(\tuple{p}) \vdash_{\logic{L}} \allset.
\end{align*}
  As before, $\Gamma \vdash_{\logic{L}} \ilset(\tuple{\varphi})$ implies that $\Gamma, \tuple{\varphi} \vdash_{\logic{L}} \allset$. Conversely, suppose that $\Gamma, \tuple{\varphi} \vdash_{\logic{L}} \allset$ and let $\sigma$ be a substitution such that $\sigma(\tuple{p}) = \tuple{\varphi}$ and $\sigma(q) = \varphi_{1} \in \tuple{\varphi}$. Let $\tau$ be a substitution such that $\tau(\tuple{p}) = \tuple{p}$ and $\tau(q) = p_{1}$ otherwise. Then by substitution swapping for antitheorems $\sigma^{-1}[T], \tuple{p} \vdash_{\logic{L}} \allset$ where $T$ is the $\logic{L}$-theory generated by~$\Gamma$. By structurality for antitheorems, we may apply~$\tau$ to obtain $\tau[\sigma^{-1}[T]], \tuple{p} \vdash_{\logic{L}} \allset$. Thus $\tau[\sigma^{-1}[T]] \in \ilfamily_{\alpha}(\tuple{p})$. Applying~$\sigma$ again yields $\sigma[\tau[\sigma^{-1}[T]]], \tuple{\varphi} \vdash_{\logic{L}} \allset$. Since $\Gamma \vdash_{\logic{L}} T$, it now suffices to prove that $\sigma[\tau[\sigma^{-1}[T]]] \subseteq T$. Consider therefore a formula $\alpha(\tuple{p}, \tuple{q}) \in \sigma[\tau[\sigma^{-1}[T]]]$. Then there is a formula $\beta(\tuple{p}, \tuple{q})$ such that $\beta(\tuple{\varphi}, \varphi_{1}) = \sigma (\beta(\tuple{p}, \tuple{q})) \in T$ and $\alpha(\tuple{p}, \tuple{q}) = \sigma(\tau(\beta(\tuple{p}, \tuple{q}))) = \sigma(\beta(\tuple{p}, p_{1})) = \beta(\tuple{\varphi}, \varphi_{1})$ where $\varphi_{1} = \sigma(p_{1})$. Thus $\alpha(\tuple{p}, \tuple{q}) \in T$.
\end{proof}

\subsection{Dual and classical inconsistency lemmas}
\label{subsec: dils}

  Just like ordinary ILs allow us to reduce questions of inconsistency in a compact logic to questions of validity, dual ILs are syntactic principles which allow for the converse reduction, in a manner analogous to the following equivalence for classical logic:
\begin{align*}
  \Gamma \vdash_{\CL} \varphi \iff \Gamma, \neg \varphi \vdash_{\logic{L}} \allset.
\end{align*}
  More generally, the finite-valued {\L}ukasiewicz logics $\Luk_{n}$ satisfy the dual IL
\begin{align*}
  \Gamma \vdash_{\Luk_{n+1}} \varphi \iff \Gamma, \neg \varphi^{n} \vdash_{\Luk_{n+1}} \allset,
\end{align*}
  while in $\Sfive$ the dual IL takes the form
\begin{align*}
  \Gamma \vdash_{\Sfive} \varphi \iff \Gamma, \neg \Box \varphi \vdash_{\Sfive} \allset,
\end{align*}
  Most of our examples of logics with the IL, however, do not enjoy the dual IL: intuitionistic logic does not, neither do $\FLe$, $\FLen$, $\K$, or $\Sfour$.

  The dual local IL follows a quantifier pattern dual to that of ordinary local~ILs: a universal quantifier takes the place of the existential one. This corresponds to the fact that while in both cases the IL family informally has a disjunctive reading, the IL family occurs to the right of the turnstile, while the dual IL family occurs to the left of the turnstile.

  For example, the infinitary {\L}ukasiewicz logic $\Lukinfty$ enjoys the dual local IL
\begin{align*}
  \Gamma \vdash_{\Lukinfty} \varphi \iff \Gamma, \neg \varphi^{n} \vdash_{\Lukinfty} \allset \text{ for each } n \in \omega.
\end{align*}
  This is an immediate consequence of the fact that for each element~$a$ of the standard {\L}ukasiewicz chain on the real interval $[0, 1]$ either $a = 1$ or $a^{n} = 0$ for some $n$. The finitary {\L}ukasiewicz logic $\Luk$, on the other hand, only enjoys this dual local IL for finite sets of formulas $\Gamma$: the above equivalence fails for $\Luk$ if we take $\Gamma$ to be $\set{\neg \varphi \rightarrow \varphi^{n}}{n \in \omega}$. We shall see that this failure of the dual local IL corresponds to the fact that $\Lukinfty$ is semisimple while $\Luk$ is not.

\begin{definition} \label{def: dil}
  A logic $\logic{L}$ enjoys the \emph{$\kappa$-ary dual parametrized local IL}, where $\kappa \leq | {\Var \logic{L}} |^{+}$, if for each ordinal $0 < \alpha < \kappa$ there is a family of sets of formulas $\dilfamily_{\alpha}(\tuple{p}, \tuple{q})$, where $\tuple{p}$ is an $\alpha$-tuple of variables, such that for each set of formulas $\Gamma$ and each $\alpha$-tuple of formulas $\tuple{\varphi}$
\begin{align*}
  \Gamma \vdash_{\logic{L}} \tuple{\varphi} \iff \Gamma, \dilset(\tuple{\varphi}, \tuple{\pi}) \vdash_{\logic{L}} \Fm \logic{L} \text{ for each set } \dilset \in \dilfamily \text{ and each tuple } \tuple{\pi}.
\end{align*}
  It enjoys the \emph{$\kappa$-ary dual local IL} if each $\dilfamily_{\alpha}$ can be chosen of the form $\dilfamily_{\alpha}(\tuple{p})$. It enjoys the \emph{$\kappa$-ary dual global IL} if each $\dilfamily_{\alpha}$ can be chosen of the form $\{ \dilset(\tuple{p}) \}$.
\end{definition}

\begin{fact}
  Each logic with the dual parametrized local IL has an antitheorem.
\end{fact}

\begin{proof}
  The dual IL implies that $p, \dilset(p, \tuple{p}) \vdash_{\logic{L}} \Fm \logic{L}$, where $\tuple{p}$ is a tuple consisting entirely of copies of the variable $p$. Thus $p, \dilset(p, \tuple{p})$ is an antitheorem.
\end{proof}

  For dual (parametrized) local ILs there is no need to take note of the arity of the dual IL, unlike in the case of the dual global IL.

\begin{fact} \label{fact: dil arity}
  Each logic with the unary dual (parametrized) local IL has the unrestricted dual (parametrized) local IL.
\end{fact}

\begin{proof}
  It suffices to take $\ilfamily_{\alpha}(\tuple{p}, \tuple{q}) \assign \bigcup_{p_{i} \in \tuple{p}} \ilfamily_{\alpha}(p_{i}, \tuple{q})$.
\end{proof}

  A suitable dual local IL allows us to strengthen the $\kappa$-compactness of a logic to $\kappa$-arity, provided that the cardinal $\kappa$ is a regular.

\begin{fact} \label{fact: kappa compact to kappa ary}
  Let $\kappa$ be a regular infinite cardinal. If a $\kappa$-compact logic $\logic{L}$ enjoys the dual local IL w.r.t.~$\ilfamily$ such that $| \ilfamily_{1} | < \kappa$, then $\logic{L}$ is a $\kappa$-ary logic.
\end{fact}

\begin{proof}
  If $\Gamma \vdash_{\logic{L}} \varphi$, then $\Gamma, \ilset(\varphi) \vdash_{\logic{L}} \allset$ for all $\ilset \in \ilfamily_{1}$. For each $\ilset \in \ilfamily_{1}$, then is therefore $\Gamma_{\ilset} \subseteq \Gamma$ of cardinality less than $\kappa$ such that $\Gamma, \ilset(\varphi) \vdash_{\logic{L}} \allset$. It~follows from the regularity of $\kappa$ that $\Gamma' \assign \bigcup_{\ilset \in \ilfamily_{1}} \Gamma_{\ilset}$ is a set of cardinality less than $\kappa$ such that $\Gamma', \ilset(\varphi) \vdash_{\logic{L}} \allset$. Thus $\Gamma' \vdash_{\logic{L}} \varphi$ by the dual local IL.
\end{proof}

  The dual IL becomes an especially powerful principle in conjunction with the ordinary~IL. Following Raftery~\cite{raftery13inconsistency}, we call such a combination a classical IL.

\begin{definition} \label{def: cil}
  A logic enjoys the \emph{$\kappa$-ary classical IL} w.r.t.\ a family $\ilfamily$ if it enjoys both the ordinary and the dual $\kappa$-ary IL w.r.t.\ $\ilfamily$.
\end{definition}

  We now show that if a logic enjoys an ordinary IL with respect to a family $\ilfamily$ as well as \emph{some} form of the dual IL w.r.t.\ \emph{some} family, then it in fact it enjoys a classical IL w.r.t.~the family $\ilfamily$. This simple observation will turn out to be crucial when applying the framework developed here to specific logics.

\begin{proposition} \label{prop: il and dil imply cil}
  Let $\logic{L}$ be a logic which enjoys the $\kappa$-ary parametrized local IL w.r.t.\ a family~$\ilfamily$. If~$\logic{L}$ enjoys the dual parametrized local IL, then it enjoys the $\kappa$-ary classical parametrized local IL w.r.t.~$\ilfamily$.
\end{proposition}

\begin{proof}
  The left-to-right direction of the dual IL follows immediately from the IL applied to $\ilset(\tuple{\varphi}, \tuple{\pi}) \vdash_{\logic{L}} \ilset(\tuple{\varphi}, \tuple{\pi})$. Conversely, suppose that $\Gamma \nvdash_{\logic{L}} \psi$ for some $\psi \in \tuple{\varphi}$. Since $\logic{L}$ enjoys the unary dual IL w.r.t.\ some family $\dilfamily'$, then there are $\dilset' \in \dilfamily'_{1}$ and $\tuple{\pi}$ such that $\Gamma, \dilset'(\psi, \tuple{\pi}) \nvdash_{\logic{L}} \allset$. But the dual IL applied to $\psi \vdash_{\logic{L}} \psi$ yields that $\dilset'(\psi, \tuple{\pi}), \psi \vdash_{\logic{L}} \allset$. By monotonicity $\dilset'(\psi, \tuple{\pi}), \tuple{\varphi} \vdash_{\logic{L}} \allset$, hence by the IL there are $\ilset \in \ilfamily_{\alpha}$ and $\tuple{\rho}$ such that $\dilset'(\psi, \tuple{\pi}) \vdash_{\logic{L}} \ilset(\tuple{\varphi}, \tuple{\rho})$. Thus $\Gamma, \dilset'(\psi, \tuple{\pi}) \nvdash_{\logic{L}} \allset$ implies that $\Gamma, \ilset(\tuple{\varphi}, \tuple{\rho}) \nvdash_{\logic{L}} \allset$.
\end{proof}

  The mere fact that a logic enjoys of a global IL therefore immediately upgrades a dual parametrized local IL to a dual global IL. Similarly, a finitary parametrized local IL immediately upgrades a unary dual IL to a finitary one.

\begin{corollary} \label{cor: il and dil imply cil}
  If a logic enjoys the $\kappa$-ary parametrized local IL and the unary dual local (global) IL, then it enjoys the $\kappa$-ary classical local (global) IL.
\end{corollary}

  Recalling the notion of equivalence of IL families (Fact~\ref{fact: all il families equivalent}), it follows that the classical IL is preserved under equivalence of IL families (see~\cite[p.\,400]{raftery13inconsistency}).

\begin{fact} \label{fact: all cil families are equivalent}
  If $\logic{L}$ enjoys a classical IL w.r.t.\ $\ilfamily$, then it enjoys a classical IL w.r.t.\ $\ilfamily'$ if and only if $\ilfamily$ and $\ilfamily'$ are equivalent.
\end{fact}

  The power of a dual parametrized local IL, a seemingly weak principle on its own, therefore lies in upgrading ordinary ILs to classical ILs, which then imply full-blooded DDTs (see \cite[Corollary~4.4]{raftery13inconsistency}).

\begin{proposition} \label{prop: cil implies ddt}
  Let $\logic{L}$ be a logic which enjoys the classical $(\kappa+\lambda)$-ary local (global) IL w.r.t.\ $\ilfamily$ such that each set in $\ilfamily_{1}$ has cardinality less than $\lambda+1$. Then $\logic{L}$ enjoys the $\kappa$-ary local (global)~DDT.
\end{proposition}

\begin{proof}
  Fix some $\alpha$ such that $0 < \alpha < \kappa$. Let us define the family of functions
\begin{align*}
  \mathcal{F} \assign \set{f\colon \ilfamily_{1} \to \bigcup_{n \in \omega} \ilfamily_{n}}{f(\ilsettwo) \in \ilfamily_{| {\ilsettwo} | + |\alpha|} \text{ for each } \ilsettwo \in \ilfamily_{1}}.
\end{align*}
  For each $f \in \mathcal{F}$ we define the set of formulas $\ddtset_{f}(\tuple{p}, q)$ as follows:
\begin{align*}
  \ddtset_{f}(\tuple{p}, q) \assign \bigcup_{\ilsettwo \in \ilfamily_{1}} f(\ilsettwo)(\tuple{p} \cup \ilsettwo(q)).
\end{align*}
  We claim that $\logic{L}$ enjoys the local DDT with respect to $\ddtfamily \assign \set{\ddtset_{f}}{f \in \mathcal{F}}$. The~claim for the classical global IL then follows immediately from the fact that $\mathcal{F}$ and therefore $\ddtfamily$ is a singleton if each $\ilfamily_{n}$ is a singleton.

  If $\Gamma, \tuple{\varphi} \vdash_{\logic{L}} \psi$, then $\Gamma, \tuple{\varphi}, \ilsettwo(\psi) \vdash_{\logic{L}} \allset$ for each $\ilsettwo \in \ilfamily_{1}$ by the dual IL. Consider some $\ilsettwo \in \ilfamily_{1}$. By assumption $| {\ilsettwo} | < \kappa$. Then by the IL there is $\ilsetthree \in \ilfamily_{| {\ilsettwo} | + |\alpha|}$ such that $\Gamma \vdash_{\logic{L}} \ilsetthree(\tuple{\varphi} \cup \ilsettwo(\psi))$. Let $f \in \mathcal{F}$ be the function which to each $\ilsettwo \in \ilfamily_{1}$ assigns this $\ilsetthree \in \ilfamily_{| {\ilsettwo} |+|\alpha|}$. Then $\Gamma \vdash_{\logic{L}} \ddtset_{f}(\varphi, \psi)$ for some $f \in \mathcal{F}$. Conversely, suppose that $\Gamma \vdash_{\logic{L}} \ddtset_{f}(\varphi, \psi)$. Then $\Gamma, \varphi, \ilsettwo(\psi) \vdash_{\logic{L}} \allset$ for each $\ilsettwo \in \ilfamily_{1}$ by the IL, therefore $\Gamma, \varphi \vdash_{\logic{L}} \psi$ by the dual IL.
\end{proof}

\begin{proposition} \label{prop: cplil implies plddt}
  Let $\logic{L}$ be a logic which enjoys the classical $(\kappa+1)$-ary global (local, parametrized local) IL w.r.t.\ $\ilfamily$ such that each set in $\ilfamily_{1}$ has cardinality less than $\kappa$. Then $\logic{L}$ has a protoimplication set.
\end{proposition}

\begin{proof}
  (Recall the definition of a protoimplication set from Fact~\ref{fact: protoimplication}.) By~the dual IL $p \vdash_{\logic{L}} p $ implies ${p, \ilset(p, \tuple{\pi}) \vdash_{\logic{L}} \allset}$ for each $\ilset \in \ilfamily$ and each tuple $\tuple{\pi}$. Thus by the assumption concerning the cardinality of sets in $\ilfamily_{1}$, for each $\ilset \in \ilfamily$ and each tuple $\tuple{\pi}$ there are $\ilsettwo \in \ilfamily_{|\ilset|+1}$ and $\tuple{\rho}$ such that 
\begin{align*}
  \emptyset \vdash_{\logic{L}} \ilsettwo(\{ p \} \cup \ilset(p, \tuple{\pi}), \tuple{\rho}).
\end{align*}
  Let $\Gamma$ be the union of the sets of formulas $\ilsettwo(\{ p \} \cup \ilset(q, \tuple{\pi}), \tuple{\rho})$ where $q$ is distinct from $p$ and $\ilsettwo(\{ p \} \cup \ilset(p, \tuple{\pi}), \tuple{\rho})$ ranges over all the sets obtained above. Then by structurality $\emptyset \vdash_{\logic{L}} \sigma[\Gamma]$ where $\sigma$ is the substitution such that $\sigma(p) = p$ and $\sigma(r) = q$ for every variable $r$ other than $p$. On the other hand, $p, \ilset(q, \tuple{\pi}), \Gamma \vdash_{\logic{L}} \allset$ for each $\ilset \in \ilfamily$ and each $\tuple{\pi}$ by the IL, therefore $p, \Gamma \vdash_{\logic{L}} q$ by the dual IL. Structurality now implies that $p, \sigma[\Gamma] \vdash_{\logic{L}} q$, thus $\sigma[\Gamma]$ is the required protoimplication set.
\end{proof}

\begin{corollary}
  Each $\kappa$-compact logic with the $\kappa$-ary classical global (local, para\-metrized local) IL enjoys the $\kappa$-ary global (local, parametrized local) DDT, where $\kappa \leq |{\Var \logic{L}}|^{+}$ is an infinite cardinal.
\end{corollary}

  In the local and global cases, our proof that a classical IL yields a DDT was constructive. This recipe for constructing DDTs from classical ILs is in fact eminently practical. To illustrate this, let us apply it to obtain a local DDT for the infinitary {\L}ukasiewicz logic $\Lukinfty$, which to the best of our knowledge was previously unknown. This also yields the only example known to us of a local DDT where the DDT family necessarily consists of infinite sets.

\begin{example} \label{ex: ddt for lukinfty}
  $\Lukinfty$ enjoys the unary local DDT:
\begin{align*}
  \Gamma, \varphi \vdash_{\Lukinfty} \psi & \iff \Gamma \vdash_{\Lukinfty} \set{f(n) (\varphi \rightarrow \psi^{n})}{n \in \omega} \text{ for some } f\colon \omega \to \omega,
\end{align*}
  or equivalently:
\begin{align*}
  \Gamma, \varphi \vdash_{\Lukinfty} \psi & \iff \text{ for each $n \in \omega$ there is $k \in \omega$ such that } \Gamma \vdash_{\Lukinfty} k(\varphi \rightarrow \psi^{n}).
\end{align*}
  However, it does not enjoy the unary LDDT w.r.t.\ any family of finite sets.
\end{example}

\begin{proof}
  By compactness, $\Lukinfty$ enjoys the same local IL as its finitary counter\-part~$\Luk$. The above recipe for constructing the local DDT now yields the DDT family consisting of the sets
\begin{align*}
  \set{\neg (p \cdot \neg q^{n})^{f(n)}}{n \in \omega} \text{ for } f\colon \omega \to \omega,
\end{align*}
  and each such set is equivalent, formula by formula, to $\set{f(n) (p \rightarrow q^{n})}{n \in \omega}$.

  If $\Lukinfty$ satisfied the local DDT w.r.t.\ a family $\ddtfamily$ of finite sets, then we claim that it satisfies the local DDT w.r.t.\ the family $\set{p^{n} \rightarrow q}{n > 0}$. Recall that this is a DDT family for $\Luk$, and $\Luk$ is precisely the finitary companion of $\Lukinfty$. If $\Gamma \vdash_{\Lukinfty} \varphi^{n} \rightarrow \psi$, then $\Gamma, \varphi \vdash_{\Lukinfty} \psi$ because $\varphi, \varphi^{n} \rightarrow \psi \vdash_{\Lukinfty} \psi$. Conversely, $\varphi, \ddtset(\varphi, \psi) \vdash_{\Lukinfty} \psi$ for each $\ddtset \in \ddtfamily$, therefore $\varphi, \ddtset(\varphi, \psi) \vdash_{\Luk} \psi$ by the finiteness assumption and $\ddtset(\varphi, \psi) \vdash_{\Luk} \varphi^{n} \rightarrow \psi$ for some $n > 0$ by the local DDT for $\Luk$. Thus $\ddtset(\varphi, \psi) \vdash_{\Lukinfty} \varphi^{n} \rightarrow \psi$ for some $n > 0$. On the other hand, $\varphi^{n} \rightarrow \psi \vdash_{\Lukinfty} \ddtset(\varphi, \psi)$ for some $\ddtset \in \ddtfamily$ by the local DDT w.r.t.~$\ddtfamily$ applied to $\varphi, \varphi^{n} \rightarrow \psi \vdash_{\Lukinfty} \psi$. Therefore $\ddtfamily$ and $\set{p^{n} \rightarrow q}{n > 0}$ are equivalent families and $\Lukinfty$ enjoys the local DDT w.r.t.\ the latter.

  We now show that $\Lukinfty$ does not satisfy the local DDT w.r.t.~this family. Let
\begin{align*}
  \Gamma & \assign \set{p^{i+1} \rightarrow (q_{i+1} \rightarrow q_{i}), \neg q_{0} \rightarrow q_{i}^{i}}{i \in \omega}.
\end{align*}
  Then $\Gamma, p \vdash_{\Lukinfty} q_{0}$. This is because $\Gamma, p \vdash_{\Lukinfty} q_{i+1} \rightarrow q_{i}$ for each $i \in \omega$, hence $\Gamma, p \vdash_{\Lukinfty} \neg q_{0} \rightarrow q_{0}^{i}$ for each $i \in \omega$, but $\set{\neg q_{0} \rightarrow q_{0}^{i}}{i \in \omega} \vdash_{\Lukinfty} q_{0}$.

  We prove that $\Gamma \nvdash_{\Lukinfty} p^{n} \rightarrow q_{0}$ for each $n > 0$. Given an $\varepsilon \in (0, 1)$, consider the valuation $v$ such that
\begin{align*}
  v(p) & \assign 1 - \frac{\varepsilon}{n+1}, \\
  v(q_{i}) & \assign \min \left( 1, 1 - \varepsilon + \frac{1 + \dots + i}{n+1} \varepsilon \right).
\end{align*}
  Then $v(p^{n}) = 1 - \frac{n \varepsilon}{n+1} \nleq 1 -\varepsilon = v(q_{0})$, hence $v(p^{n} - q) < 1$. Moreover, $v(q_{i+1} - q_{i}) \leq \frac{i+1}{n+1} \varepsilon = 1 - v(p^{i+1})$, thus $v(p^{i+1}) \leq v(q_{i+1} \rightarrow q_{i})$. It follows that $v(p^{i+1} \rightarrow (q_{i+1} \rightarrow q_{i})) = 1$. Finally, observe that $v(q_{i}) < 1$ only if $1 + \dots + i = \frac{i (i+1)}{2} \leq n$, i.e.\ only if $i \leq \sqrt{2n}$. But then $v(q_{i}^{i}) \geq 1 - i \varepsilon \geq 1 - \sqrt{2n} \varepsilon$. Taking any $\varepsilon \leq 1 - \sqrt{2n} \varepsilon$, e.g.\ $\varepsilon \assign \frac{1}{2\sqrt{2n}}$, now yields $v(\neg q_{0}) = \varepsilon \leq 1 -\sqrt{2n} \varepsilon \leq v(q_{i}^{i})$ for each $i \in \omega$. Thus $v(\neg q_{0} \rightarrow q_{i}^{i}) = 1$ for each $i \in \omega$.
\end{proof}

  The logic $\Lukinfty$ in fact enjoys an unrestricted local DDT.

\begin{example} \label{ex: infinitary ddt for lukinfty}
  The infinitary {\L}ukasiewicz logic $\Lukinfty$ enjoys the unrestricted local DDT:
\begin{align*}
  \Gamma, \tuple{\varphi} \vdash_{\Lukinfty} \psi \iff ~\, & \text{for each $n$ there is $k$ and a finite subtuple $\tuple{\varphi}'$ of $\tuple{\varphi}$ such that } \\
  & \Gamma \vdash_{\Lukinfty} k \left( \bigwedge \tuple{\varphi}' \rightarrow \psi^{n} \right).
\end{align*}
\end{example}

\begin{proof}
  By compactness, $\Lukinfty$ has the same finitary local IL as $\Luk$, which again by compactness extends to an unrestricted local IL as follows:
\begin{align*}
  \Gamma, \tuple{\varphi} \vdash_{\Lukinfty} \allset & \iff \Gamma \vdash_{\Lukinfty} \neg \left( \bigwedge \tuple{\varphi}' \right)^{n} \text{ for some finite subtuple } \tuple{\varphi}' \text{ of } \tuple{\varphi}.
\end{align*}
  (The conjunction $\bigwedge \tuple{\varphi}'$ may equivalently be replaced by a product.) The recipe for constructing DDTs from classical ILs now yields the stated DDT.
\end{proof}

\subsection{Simple inconsistency lemmas}

  The ordinary IL and the dual IL in fact coincide when restricted to simple theories $T$, thanks to the equivalences
\begin{align*}
  T \vdash_{\logic{L}} \varphi & \iff T, \varphi  \nvdash_{\logic{L}} \allset, \\
  T \vdash_{\logic{L}} \ilset(\varphi) & \iff T, \ilset(\varphi) \nvdash_{\logic{L}} \allset.
\end{align*}
  This common restriction of the IL and dual IL to simple theories will be called the simple IL, although it might equally well be called the simple dual IL.

\begin{definition} \label{def: sil}
  A logic $\logic{L}$ enjoys the \emph{$\kappa$-ary simple parametrized local IL}, where $\kappa \leq | {\Var \logic{L}} |^{+}$, if for each ordinal $0 < \alpha < \kappa$ there is a family of sets of formulas $\ilfamily_{\alpha}(\tuple{p}, \tuple{q})$, where $\tuple{p}$ is an $\alpha$-tuple of variables, such that for each simple theory $T$ of $\logic{L}$ and each $\alpha$-tuple $\tuple{\varphi}$ we have $\tuple{\varphi}, \ilset(\tuple{\varphi}, \tuple{\pi}) \vdash_{\logic{L}} \allset$ for each $\tuple{\pi}$ and
\begin{align*}
  T, \tuple{\varphi} \vdash_{\logic{L}} \allset \implies T \vdash_{\logic{L}} \ilset(\tuple{\varphi}, \tuple{\pi}) \vdash_{\logic{L}} \allset \text{ for some set } \ilset \in \ilfamily \text{ and some tuple } \tuple{\pi},
\end{align*}
  or equivalently 
\begin{align*}
  T,\ilset(\tuple{\varphi}, \tuple{\pi}) \vdash_{\logic{L}} \allset \text{ for each set } \ilset \in \ilfamily \text{ and each tuple } \tuple{\pi} \implies T \vdash_{\logic{L}} \tuple{\varphi}.
\end{align*}
  It enjoys the \emph{$\kappa$-ary simple local IL} if each $\ilfamily_{\alpha}$ can be chosen of the form $\ilfamily_{\alpha}(\tuple{p})$. It enjoys the \emph{$\kappa$-ary simple global IL} if each $\ilfamily_{\alpha}$ can be chosen of the form $\{ \ilset(\tuple{p}) \}$.
\end{definition}

  Just like with dual (parametrized) local ILs, there is again no need to take note of the arity of a simple (parametrized) local IL, since each unary simple (parametrized) local IL extends to an unrestricted one.

  While the ordinary (parametrized) local IL is equivalent to the (surjective) substitution swapping property for antitheorems, the simple (parametrized) \mbox{local} IL is equivalent to a weaker condition, namely the closure of the \mbox{family} of simple theories under preimages with respect to (surjective) substitutions. For~many of our proofs this weaker condition will be sufficient.

\begin{proposition}
  The following are equivalent for each coatomic logic $\logic{L}$ with an antitheorem:
\begin{enumerate}
\item $\logic{L}$ enjoys the simple parametrized local IL.
\item If $\sigma$ is surjective and $T$ is a simple theory of $\logic{L}$, then so is~$\sigma^{-1}[T]$.
\end{enumerate}
\end{proposition}

\begin{proof}
  (1) $\Rightarrow$ (2): if $\sigma^{-1}[T] \nvdash_{\logic{L}} \varphi$ for some simple theory $T$, then $T \nvdash_{\logic{L}} \sigma(\varphi)$, hence $T, \sigma(\varphi) \vdash_{\logic{L}} \allset$ by the simplicity of $T$ and $T \vdash_{\logic{L}} \ilset(\sigma(\varphi), \tuple{\pi})$ for some $\ilset \in \ilfamily_{1}$ and $\tuple{\pi}$ by the simple local IL. By surjectivity $\tuple{\pi} = \sigma(\tuple{\rho})$ for some $\tuple{\rho}$, so $T \vdash_{\logic{L}} \sigma[\ilset(\varphi, \tuple{\rho})]$. It follows that $\sigma^{-1}[T] \vdash_{\logic{L}} \ilset(\varphi, \tuple{\rho})$ and $\sigma^{-1}[T], \varphi \vdash_{\logic{L}} \allset$. Thus $\sigma^{-1}[T]$ is simple.

  (2) $\Rightarrow$ (1): take $\ilfamily_{1} \assign \set{\ilset(p, \tuple{q})}{p, \ilset(p, \tuple{q}) \vdash_{\logic{L}} \allset}$ and asume that $T, \varphi \vdash_{\logic{L}} \allset$ for some simple theory $T$. Consider a surjective substitution such that $\sigma(p) = \varphi$. Then $T \nvdash_{\logic{L}} \varphi = \sigma(p)$, hence $\sigma^{-1}[T] \nvdash_{\logic{L}} p$. But the theory $\sigma^{-1}[T]$ is simple by assumption, hence $\sigma^{-1}[T], p \vdash_{\logic{L}} \allset$ and $\sigma^{-1}[T] \in \ilfamily_{1}$. That is, $\sigma^{-1}[T] = \ilset(p, \tuple{q})$ for some $\ilset \in \ilfamily_{1}$. Then $T \vdash_{\logic{L}} \sigma[\sigma^{-1}[T]] = \ilset(\varphi, \sigma(\tuple{q}))$.
\end{proof}

\begin{proposition}
  The following are equivalent for each coatomic logic $\logic{L}$ with an antitheorem:
\begin{enumerate}
\item $\logic{L}$ enjoys the simple local IL.
\item If $\sigma$ is a substitution and $T$ a simple theory of $\logic{L}$, then so is $\sigma^{-1}[T]$.
\end{enumerate}
\end{proposition}

\begin{proof}
  (1) $\Rightarrow$ (2): if $\sigma^{-1}[T] \nvdash_{\logic{L}} \varphi$ for some simple theory $T$, then $T \nvdash_{\logic{L}} \sigma(\varphi)$, hence $T, \sigma(\varphi) \vdash_{\logic{L}} \allset$ by the simplicity of $T$ and $T \vdash_{\logic{L}} \ilset(\sigma(\varphi)) = \sigma[\ilset(\varphi)]$ for some $\ilset \in \ilfamily_{1}$ by the simple local IL. It follows that $\sigma^{-1}[T] \vdash_{\logic{L}} \ilset(\varphi)$ and $\sigma^{-1}[T], \varphi \vdash_{\logic{L}} \allset$. Thus $\sigma^{-1}[T]$ is simple.

  (2) $\Rightarrow$ (1): take $\ilfamily_{1} \assign \set{\ilset(p)}{p, \ilset(p) \vdash_{\logic{L}} \allset}$ and asume that $T, \varphi \vdash_{\logic{L}} \allset$ for some simple theory $T$. Consider the substitutions $\sigma(q) = \varphi$ and $\tau(q) = p$ for each variable $q$. Then $T \nvdash_{\logic{L}} \varphi = \sigma(p)$, so $\sigma^{-1}[T] \nvdash_{\logic{L}} p$, but the theory $\sigma^{-1}[T]$ is simple by assumption, hence $\sigma^{-1}[T], p \vdash_{\logic{L}} \allset$ and $\tau[\sigma^{-1}[T]] \in \ilfamily_{1}$. Moreover, $T \vdash_{\logic{L}} \sigma[\tau[\sigma^{-1}[T]]]$ by the same reasoning as in the proof of Proposition~\ref{prop: swapping for antitheorems}.
\end{proof}

  The strong three-valued Kleene logic $\Kthree$ is an example of a logic with the simple global IL but not even the parametrized local IL. We provide a sketch of the proof for readers familiar with $\Kthree$ (see e.g.~\cite[Ch.\,7]{priest08intro}).

\begin{fact}
  The strong three-valued Kleene logic $\Kthree$ enjoys the simple global IL w.r.t.~$\neg p$ but not the parametrized local IL.
\end{fact}

\begin{proof}
  Recall that a literal is a formula of the form $p$ or $\neg p$ where~$p$ is a variable. If $\Gamma$ is a set of literals and $\varphi$ is a literal, then $\Gamma, \varphi \vdash_{\Kthree} \allset$ if and only if $\Gamma \vdash_{\Kthree} \neg \varphi$. Using the disjunctive normal form and the proof by cases property
\begin{align*}
  \Gamma, \varphi \vdash_{\Kthree} \chi \text{ and } \Gamma, \psi \vdash_{\Kthree} \chi & \iff \Gamma, \varphi \vee \psi \vdash_{\Kthree} \chi,
\end{align*}
  this equivalence extends to arbitrary sets of formulas $\Gamma$. This now implies that each simple theory $T$ contains either $p$ or $\neg p$ for each variable $p$. By induction over the complexity of $\varphi$, it follows that $T$ contains either $\varphi$ or $\neg \varphi$ for each~$\varphi$. The simple global IL follows: $T, \varphi \vdash_{\Kthree} \allset \implies T \nvdash_{\Kthree} \varphi \implies T \vdash_{\Kthree} \neg \varphi$.

  On the other hand, if $\Kthree$ enjoys the parametrized local IL w.r.t.~a family $\ilfamily$, then $p, \ilset(p, \tuple{q}) \vdash_{\Kthree} \allset$ for each $\ilset \in \ilfamily_{1}$ and each $\tuple{q}$. It follows that $\ilset(p, \tuple{q}) \vdash_{\Kthree} \neg p$, e.g.\ by using the disjunctive normal form. But $p \wedge \neg p \vdash_{\Kthree} \allset$ and $\emptyset \nvdash_{\Kthree} \neg (p \wedge \neg p)$.
\end{proof}

\subsection{The law of the excluded middle}
\label{subsec: lem}

  The dual IL can also be restated in the form of a law of the excluded middle (LEM). Here we take the LEM to be a principle which states that if in some context $\Gamma$ we can deduce~$\psi$ from both $\varphi$ and the negation of $\varphi$, then we may deduce $\psi$ from $\Gamma$ itself. Although we take the LEM to be a meta-rule (an~implication between valid rules) rather than an axiom or a rule, we shall see later that the LEM can often be reduced to the validity of some set of axioms or rules.

  The canonical example of the LEM is of course the one for classical~logic:
\begin{prooftree}
  \def\fCenter{\vdash_{\CL}}
  \Axiom$\Gamma, \varphi \fCenter \psi$
  \Axiom$\Gamma, \neg \varphi \fCenter \psi$
  \BinaryInf$\Gamma \fCenter \psi$
\end{prooftree}
  (The horizontal line above represents an implication from the premises to the conclusion.) More generally, the finite-valued {\L}ukasiewicz logics~$\Luk_{n+1}$ enjoy the LEM in the form:
\begin{prooftree}
  \def\fCenter{\vdash_{\Luk_{n+1}}}
  \Axiom$\Gamma, \varphi \fCenter \psi$
  \Axiom$\Gamma, \neg \varphi^{n} \fCenter \psi$
  \BinaryInf$\Gamma \fCenter \psi$
\end{prooftree}
  The~global modal logic~$\Sfive$ enjoys the LEM in the form:
\begin{prooftree}
  \def\fCenter{\vdash_{\Sfive}}
  \Axiom$\Gamma, \varphi \fCenter \psi$
  \Axiom$\Gamma, \neg \Box \varphi \fCenter \psi$
  \BinaryInf$\Gamma \fCenter \psi$
\end{prooftree}

  The above implications are all examples of global LEMs. The local form of the LEM follows the same quantifier pattern as the local form of the dual IL. For example, the infinitary {\L}ukasiewicz logic $\Lukinfty$ enjoys the local LEM
\begin{prooftree}
  \def\fCenter{\vdash_{\Lukinfty}}
  \Axiom$\Gamma, \varphi \fCenter \psi$
  \Axiom$\Gamma, \neg \varphi^{n} \fCenter \psi \text{ for each } n \in \omega$
  \BinaryInf$\Gamma \fCenter \psi$
\end{prooftree}
  Let us now provide a precise definition of the LEM in our sense of the term.

\begin{definition}
  A logic $\logic{L}$ enjoys the \emph{$\kappa$-ary parametrized local LEM}, where $\kappa \leq |{\Var \logic{L}}|^{+}$, if for each ordinal $0 < \alpha < \kappa$ there is a family of sets of formulas $\lemfamily_{\alpha}(\tuple{p}, \tuple{q})$, where $\tuple{p}$ is an $\alpha$-tuple of variables, such that
\begin{align*}
  \tuple{\varphi}, \lemset(\tuple{\varphi}, \tuple{\pi}) \vdash_{\logic{L}} \Fm \logic{L} \text{ for each $\lemset \in \lemfamily_{\alpha}$ and each $\tuple{\varphi}$ and $\tuple{\pi}$},
\end{align*}
  and moreover $\Gamma \vdash_{\logic{L}} \psi$ whenever
\begin{align*}
  \Gamma, \tuple{\varphi} \vdash_{\logic{L}} \psi \text{ and } \Gamma, \lemset(\tuple{\varphi}, \tuple{\pi}) \vdash_{\logic{L}} \psi \text{ for each set } \lemset \in \lemfamily \text{ and each tuple } \tuple{\pi}.
\end{align*}
  It enjoys the \emph{local LEM} if each $\lemfamily_{\alpha}$ can be chosen of the form $\lemfamily_{\alpha}(\tuple{p})$. It enjoys the \emph{global LEM} if each $\lemfamily_{\alpha}$ can be chosen of the form $\{ \lemset(\tuple{p}) \}$.
\end{definition}

  It is no accident that the LEMs above correspond precisely to the dual ILs enjoyed by these logics.

\begin{proposition} \label{prop: dil iff lem}
  A logic enjoys the $\kappa$-ary dual parametrized local IL w.r.t.\ the family $\dilfamily$ if and only if it enjoys the $\kappa$-ary parametrized local LEM w.r.t.~$\dilfamily$.
\end{proposition}

\begin{proof}
  If $\logic{L}$ has the LEM w.r.t.\ $\dilfamily$ and $\Gamma, \dilset(\tuple{\varphi}, \tuple{\pi}) \vdash_{\logic{L}} \Fm \logic{L}$ for each $\dilset \in \dilfamily$ and each $\tuple{\pi}$, then for each $\varphi_{i} \in \tuple{\varphi}$ the LEM for $\psi \assign \varphi_{i}$ yields $\Gamma \vdash_{\logic{L}} \varphi_{i}$. The other implication of the dual IL follows from $\tuple{\varphi}, \dilset(\tuple{\varphi}, \tuple{\pi}) \vdash_{\logic{L}} \Fm \logic{L}$ for each $\dilset \in \dilfamily$.

  Conversely, suppose $\logic{L}$ has the dual IL w.r.t.\ $\dilfamily$ and $\Gamma, \tuple{\varphi} \vdash_{\logic{L}} \psi$ and $\Gamma, \dilset(\tuple{\varphi}, \tuple{\pi}) \vdash_{\logic{L}} \psi$ for each $\dilset \in \dilfamily$ and each $\tuple{\pi}$. We need to show that $\Gamma \vdash_{\logic{L}} \psi$.

  By the dual IL $\Gamma, \dilset(\tuple{\varphi}, \tuple{\pi}) \vdash_{\logic{L}} \psi$ yields that $\Gamma, \dilset(\tuple{\varphi}, \tuple{\pi}), \dilsettwo(\psi, \tuple{\rho}) \vdash_{\logic{L}} \allset$ for each $\dilsettwo \in \dilfamily_{1}$ and each $\tuple{\rho}$. By the dual IL this implies $\Gamma, \dilsettwo(\psi, \tuple{\rho}) \vdash_{\logic{L}} \tuple{\varphi}$. Applying the assumption that $\Gamma, \tuple{\varphi} \vdash_{\logic{L}} \psi$ yields $\Gamma, \dilsettwo(\psi, \tuple{\rho}) \vdash_{\logic{L}} \psi$, but $\psi, \dilsettwo(\psi, \tuple{\rho}) \vdash_{\logic{L}} \allset$, therefore $\Gamma, \dilsettwo(\psi, \tuple{\rho}) \vdash_{\logic{L}} \allset$. But by the dual IL this implies that $\Gamma \vdash_{\logic{L}} \psi$.

  The equivalences for dual local and global ILs and the corresponding forms of the LEM are an immediate consequence of the equivalence between the dual parametrized local IL and the parametrized local LEM.
\end{proof}

\begin{corollary}
  A logic enjoys the $\kappa$-ary dual global (local, parametrized local) IL if and only if it enjoys the $\kappa$-ary global (local, parametrized local) LEM.
\end{corollary}

  In particular, for the (parametrized) local form of the LEM there is no need to take note of the arity, since the unary (parametrized) local LEM extends to an unrestricted one by Fact~\ref{fact: dil arity}.

\subsection{Syntactic characterization of semisimplicity}
\label{subsec: syntactic semisimplicity}

  We now show that, assuming the simple IL, semisimplicity is in fact equivalent to the law of the excluded middle (LEM), or equivalently the dual IL. Recall that the simple IL follows both from the ordinary IL and from the dual~IL. We first show that the dual IL then extends from simple to semi\-simple theories.

\begin{proposition}
  If $\logic{L}$ enjoys the $\kappa$-ary simple parametrized local IL w.r.t.\ a family $\ilfamily$, then for each semisimple theory $T$ of $\logic{L}$
\begin{align*}
  T,\ilset(\tuple{\varphi}, \tuple{\pi}) \vdash_{\logic{L}} \allset \text{ for each set } \ilset \in \ilfamily \text{ and each tuple } \tuple{\pi} \implies T \vdash_{\logic{L}} \tuple{\varphi}.
\end{align*}
\end{proposition}

\begin{proof}
  If $T \nvdash_{\logic{L}} \tuple{\varphi}$ and $T$ is semisimple, then there is some semisimple theory $T' \supseteq T$ such that $T' \nvdash_{\logic{L}} \tuple{\varphi}$. By~the simple IL there are $\ilset \in \ilfamily_{\alpha}$ and $\tuple{\pi}$ such that $T', \ilset(\tuple{\varphi}, \tuple{\pi}) \nvdash_{\logic{L}} \allset$, hence also $T, \ilset(\tuple{\varphi}, \tuple{\pi}) \nvdash_{\logic{L}} \allset$.
\end{proof}

\begin{proposition} \label{prop: semisimple and il implies dil}
  Each semisimple logic with the $\kappa$-ary parametrized local IL w.r.t.~$\ilfamily$ enjoys the $\kappa$-ary classical parametrized local IL w.r.t.~$\ilfamily$.
\end{proposition}

\begin{proof}
  The IL implies the simple IL, which implies the dual IL for semisimple theories. But by assumption each theory of the logic is semisimple.
\end{proof}

  Conversely, the dual parametrized local IL implies semisimplicity.

\begin{proposition} \label{prop: dil implies semisimple}
  Each coatomic logic with the dual para\-metrized local IL is semisimple.
\end{proposition}

\begin{proof}
  Suppose that $\Gamma \nvdash_{\logic{L}} \varphi$. Then by the dual IL there are $\dilset \in \dilfamily$ and $\tuple{\pi}$ such that $\Gamma, \dilset(\varphi, \tuple{\pi}) \nvdash_{\logic{L}} \allset$. By~coatomicity $\Gamma, \dilset(\varphi, \tuple{\pi})$ extends to a simple $\logic{L}$-theory $\Delta$. But $\varphi, \dilset(\varphi, \tuple{\pi}) \vdash_{\logic{L}} \allset$ by the dual IL applied to $\varphi \vdash_{\logic{L}} \varphi$, therefore $\Delta \nvdash_{\logic{L}} \varphi$. It~follows that each $\logic{L}$-theory $\Gamma$ is semisimple.
\end{proof}

  Combining these simple observations with the equivalence between the LEM, the dual IL, and the classical IL now yields the following theorem, which is the main syntactic result of the first half of this paper.

\begin{theorem} \label{thm: semisimplicity}
  The following are equivalent for each coatomic logic $\logic{L}$ which enjoys the simple (parametrized) local IL w.r.t.\ $\ilfamily$:
\begin{enumerate}
\item $\logic{L}$ is semisimple.
\item $\logic{L}$ enjoys the (parametrized) local LEM.
\item $\logic{L}$ enjoys the dual (parametrized) local IL.
\item $\logic{L}$ enjoys the (parametrized) local LEM w.r.t.\ $\ilfamily$.
\item $\logic{L}$ enjoys the dual (parametrized) local IL w.r.t.\ $\ilfamily$.
\end{enumerate}
\end{theorem}

  An analogous equivalence for the global LEM and the dual global IL immediately follows.

\begin{theorem}
  The following are equivalent for each coatomic logic $\logic{L}$ which enjoys the $\kappa$-ary simple global local IL w.r.t.\ $\ilfamily$:
\begin{enumerate}
\item $\logic{L}$ is semisimple.
\item $\logic{L}$ enjoys the unary global LEM.
\item $\logic{L}$ enjoys the unary dual global IL.
\item $\logic{L}$ enjoys the $\kappa$-ary global LEM w.r.t.\ $\ilfamily$.
\item $\logic{L}$ enjoys the $\kappa$-ary dual global IL w.r.t.\ $\ilfamily$.
\end{enumerate}
\end{theorem}

   Let us remark here that, while the equivalence between the originally algebraic notion of semisimplicity and the logical notion of the LEM is certainly of interest conceptually, in practical applications it is crucial that we know the specific form of the LEM. It would be quite difficult to apply this theorem to describe, say, the semisimple extensions of a given logic $\logic{L}$, as these could enjoy the LEM with respect to widely varying families~$\ilfamily$. However, as we shall see, the theorem can be very helpful when describing the semisimple \emph{axiomatic} extensions of a logic~$\logic{L}$ which enjoys the IL w.r.t.\ $\ilfamily$, since each of them inherits this specific IL.

\label{subsec: semantic semisimplicity}

  It will not have escaped the reader looking to apply the previous theorem to establish the semisimplicity of some class of algebras that there is a catch to it: we have only defined a syntactic notion of semi\-simplicity (each theory is an intersection of simple theories), whereas the algebraist's notion of semisimplicity is semantic and relates to algebras. Let~us therefore show how to get from the syntactic semisimplicity of an (at~least weakly) algebraizable logic to the semisimplicity of the corresponding class of algebras. This will require somewhat stronger assumptions.



  We call a logic $\logic{L}$ \emph{semantically semisimple} if each $\logic{L}$-filter on each algebra~$\alg{A}$ is an intersection of simple (i.e.\ maximal non-trivial) $\logic{L}$-filters on $\alg{A}$. We call it \emph{semantically coatomic} if the lattice of $\logic{L}$-filters on each algebra is coatomic, i.e.\ if each $\logic{L}$-filter on each algebra~$\alg{A}$ is included in some simple $\logic{L}$-filter on $\alg{A}$. Clearly each semantically semisimple or coatomic logic is also syntactically so.

\begin{fact}
  The following implications hold:
\begin{enumerate}
\item Each finitary logic with an antitheorem is semantically compact.
\item Each semantically compact logic is semantically coatomic.
\end{enumerate}
\end{fact}

\begin{proof}
  Item (1) holds because $\logic{L}$-filter generation yields an algebraic closure operator on each algebra if $\logic{L}$ is finitary, and the existence of a finite antitheorem ensures that the trivial filter on each algebra is finitely generated. Item (2) is a general lattice theoretic fact.
\end{proof}

  The semantic counterparts of the syntactic principles defined in the previous subsections are obtained by replacing the quantification over $\Gamma$, $\tuple{\varphi}$, $\tuple{\pi}$ and $\psi$ with quantification over all algebras $\alg{A}$ (instead of the algebra $\FmAlg \logic{L}$) and all subsets $X \subseteq \alg{A}$, all tuples $\tuple{a}, \tuple{b} \in \alg{A}$, and all elements $c \in \alg{A}$. The \emph{semantic $\kappa$-ary dual parametrized local IL}, for example, states that for each algebra $\alg{A}$, each $X \subseteq \alg{A}$, and each $\alpha$-tuple $\tuple{a} \in \alg{A}$ for $0 < \alpha < \kappa$
\begin{align*}
  X \vdash_{\logic{L}}^{\alg{A}} \tuple{a} & \iff X, \dilset(\tuple{a}, \tuple{b}) \vdash_{\logic{L}}^{\alg{A}} A \text{ for each set $\dilset \in \dilfamily_{\alpha}$ and each tuple $\tuple{b}$}.
\end{align*}
  Recall that $X \vdash_{\logic{L}}^{\alg{A}} \tuple{a}$ is our notation for the claim that each element of $\tuple{a}$ belongs to the $\logic{L}$-filter of $\alg{A}$ generated by $X$.

\begin{proposition} \label{prop: semantic dil implies semantically semisimple}
  Each semantically coatomic logic which enjoys the semantic simple dual para\-metrized local IL is semantically semisimple.
\end{proposition}

\begin{proof}
  This is a straightforward semantic analogue of Proposition~\ref{prop: dil implies semisimple}.
\end{proof}

  Given the above implications, in order to prove that a finitary semisimple logic with an antitheorem is semantically semisimple, it suffices to show that the syntactic dual parametrized local IL extends to a semantic one. If~the~logic in question is not finitary but it is at least compact, then one moreover has to show that syntactic compactness extends to semantic compactness. 

  Such theorems, which extend a syntactic property to a corresponding semantic one, are called \emph{transfer theorems} in abstract algebraic logic. The proof of the transfer theorem for the classical local IL will involve several technical lemmas, which are in fact nothing but straightforward modifications of existing proofs found e.g.\ in~\cite{font16}. The reader unfamiliar with the techniques involved in such proofs is advised to skip the details and jump to Theorem~\ref{thm: semantic semisimplicity} and Fact~\ref{fact: logic to algebra semisimplicity}.

  In the following pair of lemmas, recall that the existence of a protoimplication set is equivalent to the parametrized local DDT (Fact~\ref{fact: protoimplication}). Moreover, each logic $\logic{L}$ such that $|{\Fm \logic{L}}| = |{\Var \logic{L}}|$ is $\lambda$-ary for $\lambda = |{\Var \logic{L}}|^{+}$.

\begin{lemma} \label{lemma: filter generation}
  Let $\logic{L}$ be $\lambda$-ary logic for $\lambda = | {\Var \logic{L}} |^{+}$ with a proto\-implication set. Then $X \vdash_{\logic{L}}^{\alg{A}} a$ if and only if there are $\Gamma$, $\varphi$, and a homomorphism $h\colon \FmAlg \logic{L} \to \alg{A}$ such that $h[\Gamma] \subseteq X \cup \Fg_{\logic{L}}^{\alg{A}} \emptyset$, $h(\varphi) = a$, and~${\Gamma \vdash_{\logic{L}} \varphi}$.
\end{lemma}

\begin{proof}
  This is a standard lemma about filter generation for finitary logics (see~\cite[Proposition~6.12]{font16}). The only subtlety involved in generalizing the proof to $\lambda$-ary logics is that given a set of rules $\set{\Gamma_{i} \vdash \varphi_{i}}{i \in I}$ such that $| I | < \lambda$ for each $i \in I$, we need to be able to rename the variables in these rules so that distinct rules use disjoint sets of variables. The assumption that $\lambda = | \Var \logic{L}|^{+}$ takes care of this problem.
\end{proof}

\begin{lemma} \label{lemma: trivial filter generation}
  Let $\logic{L}$ be $\lambda$-ary logic for $\lambda = | {\Var \logic{L}} |^{+}$ which has an antitheorem and a protoimplication set. Then $X \vdash_{\logic{L}}^{\alg{A}} A$ if and only if there is~$\Gamma$ and a homomorphism $h\colon \FmAlg \logic{L} \to \alg{A}$ such that $h[\Gamma] \subseteq X \cup \Fg_{\logic{L}}^{\alg{A}} \emptyset$ and~${\Gamma \vdash_{\logic{L}} \allset}$.
\end{lemma}

\begin{proof}
  The right-to-left implication is straightforward. Conversely, because $\logic{L}$ is $\lambda$-ary and has an antitheorem, it has an antitheorem~$\antitheorem$ of cardinality at most $| {\Var \logic{L}} |$, w.l.o.g.\ in a single variable $p$. Then any homomorphism $g\colon \FmAlg \logic{L} \to \alg{A}$ yields a set $\antiset \assign g[\antitheorem]$ such that $\antiset \vdash_{\logic{L}}^{\alg{A}} A$ and $| {\antiset} | \leq | {\Var \logic{L}} |$. By Lemma~\ref{lemma: filter generation}, for each $a \in \antiset$ there are formulas $\Gamma_{a}, \varphi_{a}$ and a homomorphism $h_{a}\colon \FmAlg \logic{L} \to \alg{A}$ such that $h_{a}[\Gamma_{a}] \subseteq X \cup \Fg_{\logic{L}}^{\alg{A}} \emptyset$, $h_{a}(\varphi_{a}) = a$, and $\Gamma_{a} \vdash_{\logic{L}} \varphi_{a}$. Because $| {\antiset} | \leq | {\Var \logic{L}} |$, we may assume that the formulas $\Gamma_{a}, \varphi_{a}$ and the formulas $\Gamma_{a'}, \varphi_{a'}$ do not share any variables if $a$ and $a'$ are distinct. We~thus obtain a set $\Gamma \assign \bigcup_{a \in \antiset} \Gamma_{a}$, a~tuple of formulas $\tuple{\varphi}$ which consists of the formulas~$\varphi_{a}$, and a homomorphism $h\colon \Fm \logic{L} \to \alg{A}$ such that $h[\Gamma] \subseteq X \cup \Fg_{\logic{L}}^{\alg{A}} \emptyset$, $h(\tuple{\varphi}) = \tuple{a}$, and $\Gamma \vdash_{\logic{L}} \tuple{\varphi}$. We may assume w.l.o.g.\ that $\Gamma$ does not contain the variable~$p$.

  If $\protoset$ is a protoimplication set for $\logic{L}$, let $\Gamma' \assign \Gamma \cup \bigcup_{\antiformula \in \antiset} \protoset(\varphi_{g(\antiformula)}, \antiformula)$ and take $h'\colon \Fm \logic{L} \to \alg{A}$ to be the homomorphism which agrees with $h$ on every variable except $h'(p) = g(p)$, then $h'[\Gamma'] \subseteq X \cup \Fg_{\logic{L}}^{\alg{A}} \emptyset$, since $h'(\varphi_{g(\pi)}) = g(\pi) = h'(\pi)$. Moreover, now $\Gamma' \vdash_{\logic{L}} \emptyset$ because $\Gamma \vdash_{\logic{L}} \tuple{\varphi}$ and $\tuple{\varphi}, \bigcup_{\antiformula \in \antitheorem} \protoset(\varphi_{g(\antiformula)}, \antiformula) \vdash_{\logic{L}} \antitheorem$.
\end{proof}

  The equivalence of syntactic and semantic compactness for logics which satisfy the assumptions of the previous lemma now follows immediately.

\begin{proposition} \label{prop: syntactic to semantic compactness}
  Let $\logic{L}$ be a $\lambda$-ary logic for $\lambda = | {\Var \logic{L}} |^{+}$ which has an antitheorem and a protoimplication set. If $\logic{L}$ is (syntactically) $\kappa$-compact, then it is semantically $\kappa$-compact.
\end{proposition}

  We are now in a position to extend the syntactic classical IL to a semantic one.

\begin{proposition} \label{prop: syntactic to semantic semisimplicity}
  Let $\logic{L}$ be a $\lambda$-ary logic for $\lambda = | { \Var \logic{L} }|^{+}$. If $\logic{L}$ enjoys the syntactic $\kappa$-ary classical local IL w.r.t.\ a family $\ilfamily$ such that $| {\ilfamily_{\alpha}} | \leq | {\Var \logic{L}} |$ for each $\ilfamily_{\alpha}$ and each set in $\ilfamily_{1}$ has cardinality less than $\kappa$, then it enjoys the semantic $\kappa$-ary classical local IL w.r.t.\ $\ilfamily$.
\end{proposition}

\begin{proof}
  Recall that the classical IL implies the DDT (Proposition~\ref{prop: cil implies ddt}), therefore the previous two lemmas apply. We first show the transfer of the IL. One~direction is clear: $X \vdash_{\logic{L}}^{\alg{A}} \ilset(\tuple{a}, \tuple{b})$ for $\ilset \in \ilfamily_{\alpha}$ and $\tuple{b} \subseteq \alg{A}$ implies $X, \tuple{a} \vdash_{\logic{L}}^{\alg{A}} A$ because $\tuple{a}, \ilset(\tuple{a}, \tuple{b}) \vdash_{\logic{L}}^{\alg{A}} A$ by virtue of $\tuple{p}, \ilset(\tuple{p}, \tuple{q})$ being an antitheorem. Conversely, suppose that $X, \tuple{a} \vdash_{\logic{L}}^{\alg{A}} A$ where $\tuple{a} \in \alg{A}$ has length $0 < \alpha < \kappa$. Then by the previous lemma there are $\Gamma$, $\tuple{\varphi}$, and a homomorphism $h\colon \FmAlg \logic{L} \to \alg{A}$ such that $h[\Gamma] \subseteq \Fg_{\logic{L}}^{\alg{A}} \emptyset$, $h(\tuple{\varphi}) = \tuple{a}$, and $\Gamma, \tuple{\varphi} \vdash_{\logic{L}} \allset$. By the IL there are $\ilset \in \ilfamily_{\alpha}$ and $\tuple{\pi}$  such that $\Gamma \vdash_{\logic{L}} \ilset(\tuple{\varphi}, \tuple{\pi})$, thus for $\tuple{b} \assign h(\tuple{\pi})$ we have $\ilset(\tuple{a}, \tuple{b}) = h[\ilset(\tuple{\varphi}, \tuple{\pi})] \subseteq \Fg_{\logic{L}}^{\alg{A}} h[\Gamma] \subseteq \Fg_{\logic{L}}^{\alg{A}} X$.

  For the dual IL, one direction is again clear: $X \vdash_{\logic{L}}^{\alg{A}} \tuple{a}$ implies $X, \ilset(\tuple{a}) \vdash_{\logic{L}}^{\alg{A}} A$ for each $\ilset \in \ilfamily_{\alpha}$ because $\tuple{a}, \ilset(\tuple{a}) \vdash_{\logic{L}}^{\alg{A}} A$. Conversely, suppose that $X, \ilset(\tuple{a}) \vdash_{\logic{L}}^{\alg{A}} A$ for each $\ilset \in \ilfamily_{\alpha}$. Then by the previous fact for each $\ilset \in \ilfamily_{\alpha}$ there are $\Gamma_{\ilset}, \Lambda_{\ilset} \subseteq \Fm \logic{L}$, and $h_{\ilset}\colon \FmAlg \logic{L} \to \alg{A}$ such that $h_{\ilset}[\Gamma_{\ilset}] \subseteq X \cup \Fg_{\logic{L}}^{\alg{A}} \emptyset$ and $h_{\ilset}[\Lambda_{\ilset}] \subseteq \ilset(\tuple{a})$ and $\Gamma_{\ilset}, \Lambda_{\ilset} \vdash_{\logic{L}} \allset$. More precisely, for each $\lambda_{\ilset} \in \Lambda_{\ilset}$ there is $\iota_{\lambda_{\ilset}} \in \ilset$ such that $h(\lambda_{\ilset}) = \iota_{\lambda_{\ilset}}(\tuple{a})$. We~may assume w.l.o.g.\ that there is an $\alpha$-tuple of variables $\tuple{p}$ which do not occur in any $\Gamma_{\ilset}, \Lambda_{\ilset}$. Moreover, since $| {\ilfamily_{\alpha}} | \leq | {\Var \logic{L}} |$ we may assume that the sets $\Gamma_{\ilset}, \Lambda_{\ilset}$ and $\Gamma_{\ilset'}, \Lambda_{\ilset'}$ do not share any variables for distinct $\ilset$ and $\ilset'$. This allows us to define a homomorphism $h\colon \FmAlg \logic{L} \to \alg{A}$ such that $h[\Gamma_{\ilset}] = h_{\ilset}[\Gamma_{\ilset}]$, $h[\Lambda_{\ilset}] = h_{\ilset}[\Lambda_{\ilset}]$, and $h(\tuple{a}) = \tuple{a}$.

  Now consider $\Gamma \assign \bigcup_{\ilset \in \ilfamily_{\alpha}} \Gamma_{\ilset}$ and $\Gamma' \assign \Gamma \cup \bigcup_{\ilset \in \ilfamily_{\alpha}} \bigcup_{\lambda_{\ilset} \in \Lambda_{\ilset}} \protoset(\iota_{\lambda_{\ilset}}(\tuple{p}), \lambda_{\ilset})$. Then $h[\Gamma'] \subseteq X \cup \Fg_{\logic{L}}^{\alg{A}} \emptyset$ because $h(\lambda_{\ilset}) = (\iota_{\lambda_{\ilset}}(\tuple{a})) = h(\iota_{\lambda_{\ilset}}(\tuple{p}))$. Moreover, $\Gamma', \ilset(\tuple{p}) \vdash_{\logic{L}} \Lambda_{\ilset}$ for each $\ilset \in \ilfamily_{\alpha}$, hence $\Gamma', \ilset(\tuple{p}) \vdash_{\logic{L}} \allset$ for each $\tuple{p}$. By the dual local IL it follows that $\Gamma' \vdash_{\logic{L}} \tuple{p}$. Finally, $h(\tuple{p}) = \tuple{a}$ implies that $\tuple{a} \in \Fg_{\logic{L}}^{\alg{A}} X$.
\end{proof}

  Combining the semantic results of this subsection with the characterization of syntactic semi\-simplicity (Theorem~\ref{thm: semisimplicity}) now yields the following theorem, which is the main semantic result of the first half of the paper. It upgrades our previous characterization of syntactic semisimplicity to a characterization of semantic semisimplicity, provided that the logic enjoys a local IL rather than merely a parametrized local one and certain very mild technical conditions are met.

\begin{theorem} \label{thm: semantic semisimplicity}
  The following are equivalent for a semantically coatomic logic $\logic{L}$ which enjoys the $\kappa$-ary local IL w.r.t.~$\ilfamily$, provided that $| {\ilfamily_{\alpha}} | \leq | {\Var \logic{L}} |$ for all~$\ilfamily_{\alpha}$, each set in $\ilfamily_{1}$ has cardinality less than $\kappa$, and $\logic{L}$ is $\lambda$-ary for $\lambda = |{\Var \logic{L}}|^{+}$:
\begin{enumerate}
\item $\logic{L}$ is semantically semisimple.
\item $\logic{L}$ is syntactically semisimple.
\item $\logic{L}$ enjoys the unary local LEM.
\item $\logic{L}$ enjoys the unary dual local IL.
\item $\logic{L}$ enjoys the unary classical local IL.
\item $\logic{L}$ enjoys the semantic $\kappa$-ary local LEM w.r.t.\ $\ilfamily$.
\item $\logic{L}$ enjoys the semantic $\kappa$-ary dual local IL w.r.t.\ $\ilfamily$.
\item $\logic{L}$ enjoys the semantic $\kappa$-ary classical local IL w.r.t.\ $\ilfamily$.
\end{enumerate}
  In particular, the above equivalence holds for each compact $\logic{L}$ with the $\kappa$-ary local IL w.r.t.~$\ilfamily$, provided that $| {\ilfamily_{\alpha}} | \leq | {\Var \logic{L}} |$ for~each~$\ilfamily_{\alpha}$ and each set in $\ilfamily_{1}$ has cardinality less than $\kappa$.
\end{theorem}

\begin{proof}
  The~implication from semantic semisimplicity to the syntactic semi\-simplicity is trivial, and the equivalence between the syntactic conditions has already been established in Theorem~\ref{thm: semisimplicity}. The proof of the equivalence between the syntactic LEM and the syntactic dual IL (Proposition~\ref{prop: dil iff lem}) immediately extends to their semantic counterparts, as does the proof that the syntactic $\kappa$-ary IL and unary dual IL imply the syntactic $\kappa$-ary classical IL (Proposition~\ref{prop: il and dil imply cil}). The~implication from the semantic classical IL to semantic semisimplicity is Proposition~\ref{prop: semantic dil implies semantically semisimple}. Finally, the implication from the syntactic unary classical IL (which implies the syntactic $\kappa$-ary classical IL by Proposition~\ref{prop: il and dil imply cil}) to the semantic $\kappa$-ary classical IL is Proposition~\ref{prop: syntactic to semantic semisimplicity}.

  Each of the conditions implies the unary classical local IL without the use of semantic coatomicity and $\lambda$-arity. But compactness together with the unary classical local IL imply that $\logic{L}$ is $\lambda$-ary for $\lambda = |{\Var \logic{L}}|^{+}$ by Fact~\ref{fact: kappa compact to kappa ary}, using the assumption that $|{\ilfamily_{1}}| < \lambda$. Moreover, compactness together with $\lambda$-arity and the unary parametrized local DDT (which follows from the unary classical local IL by Proposition~\ref{prop: cplil implies plddt}) imply the semantic compactness of $\logic{L}$ by Proposition~\ref{prop: syntactic to semantic compactness}. The~hypotheses about $\lambda$-arity and semantic coatomicity may thus be omitted if $\logic{L}$ is compact.
\end{proof}

  Recall again that each logic with $|{\Fm \logic{L}}| = |{\Var \logic{L}}|$, in particular each logic with at most countably many connectives and constants, is $\lambda$-ary for $\lambda = |{\Var \logic{L}}|^{+}$.

  In order to use the above theorem to prove the semisimplicity of a class of algebras, it remains to connect the semantic semisimplicity of a (weakly) algebraizable logic with the semisimplicity of its algebraic counterpart.

\begin{fact} \label{fact: logic to algebra semisimplicity}
  Let $\logic{L}$ be a (weakly) algebraizable logic and let $\class{K}$ be its algebraic counterpart. Then $\class{K}$ is semisimple if and only if $\logic{L}$ is semantically semisimple.
\end{fact}

\begin{proof}
  By \cite[Proposition~6.117]{font16} the lattice of $\class{K}$-congruences and the lattice of $\logic{L}$-filters on each algebra are isomorphic.
\end{proof}

\subsection{Applications}
\label{subsec: applications}

  We now show how the equivalence between semi\-simplicity and the LEM (Theorem~\ref{thm: semisimplicity}) can be used to describe the semisimple axiomatic extensions of a given logic. We consider two cases: the~semisimple axiomatic extensions of the logics $\FLe$ and~$\IK$. By Theorem~\ref{thm: semantic semisimplicity} and Fact~\ref{fact: logic to algebra semisimplicity}, this problem is equivalent to the problem of describing the semisimple varieties of $\FLe$-algebras and modal Heyting algebras.

  The semisimple varieties of $\FLew$-algebras and modal algebras were already described by Kowalski~\cite{kowalski04} and by Kowalski \& Kracht~\cite{kowalski+kracht06} respectively. The proofs of Kowalski and Kracht are algebraic in nature and, although their overall structure is similar, they involve fairly complicated computations specific to the two varieties. In contrast, having isolated the common core of the two proofs into Theorem~\ref{thm: semisimplicity}, our proofs are relatively brief and purely syntactic.

\begin{theorem}
  An axiomatic extension of $\FLe$ is semisimple if and only if the formula $(1 \wedge p) \vee (1 \wedge \neg (1\wedge p)^{n})$ is a theorem of $\logic{L}$ for some $n \in \omega$.
\end{theorem}

\begin{proof}
  Let $\logic{L}$ be an axiomatic extension of $\FLe$. Then $\logic{L}$ is finitary and it inherits the local DDT and local IL of $\FLe$. By Theorem~\ref{thm: semisimplicity}, the semisimplicity of $\logic{L}$ is therefore equivalent to the LEM w.r.t.\ $\set{\neg (1 \wedge p)^{n}}{n \in \omega}$:
\begin{align*}
  \Gamma, \varphi \vdash_{\logic{L}} \psi \text{ and } \Gamma, \neg (1 \wedge \varphi)^{n} \vdash_{\logic{L}} \psi \text{ for each } n \in \omega & \implies \Gamma \vdash_{\logic{L}} \psi.
\end{align*}
  Suppose first that $\logic{L}$ is semisimple and take $\varphi \assign p$, $\psi \assign q$, and
\begin{align*}
  \Gamma \assign \{ (1 \wedge p) \rightarrow q \} \cup \set{(1 \wedge \neg (1 \wedge p)^{n}) \rightarrow q}{n \in \omega}.
\end{align*}
  The LEM yields that $\Gamma \vdash_{\logic{L}} q$. Since $(1 \wedge \neg (1 \wedge p)^{i}) \rightarrow q \vdash_{\logic{L}} (1 \wedge \neg (1 \wedge p)^{j}) \rightarrow q$ holds for $j \leq i$, finitarity implies that for some $k \in \omega$
\begin{align*}
  (1 \wedge p) \rightarrow q, (1 \wedge \neg (1 \wedge p)^{k}) \rightarrow q \vdash_{\logic{L}} q.
\end{align*}
  Substituting $(1 \wedge p) \vee (1 \wedge \neg (1 \wedge p)^{k})$ for $q$ now yields the desired theorem.

  Conversely, recall that $\logic{L}$ inherits the proof by cases property (PCP) of $\FLe$:
\begin{align*}
  \Gamma, \varphi \vdash_{\logic{L}} \chi \text{ and } \Gamma, \psi \vdash_{\logic{L}} \chi & \iff \Gamma, (1 \wedge \varphi) \vee (1 \wedge \psi) \vdash_{\logic{L}} \chi.
\end{align*}
  It immediately follows that if $(1 \wedge p) \vee (1 \wedge \neg (1 \wedge p)^{n})$ is a theorem of $\logic{L}$, then $\logic{L}$ enjoys the LEM w.r.t.\ $\set{\neg (1 \wedge p)^{n}}{n \in \omega}$.
\end{proof}

  We have in fact proved the above claim for any extension of $\FLe$ (or more generally of any reduct of $\FLe$) which shares the same local DDT.

\begin{corollary}
  A variety of $\FLe$-algebras is semisimple if and only if it satisfies the equation $(1 \wedge x) \vee (1 \wedge \neg (1 \wedge x)^{n}) \equals 1$ for some $n$.
\end{corollary}

  Specializing the above results to $\FLew$ and $\FLew$-algebras allows us to simplify their statement somewhat.

\begin{theorem}
  An axiomatic extension of $\FLew$ is semisimple if and only if $p \vee \neg p^{n}$ is a theorem of $\logic{L}$ for some $n \in \omega$.
\end{theorem}

\begin{corollary}[\cite{kowalski04}]
  A variety of $\FLew$-algebras is semisimple if and only if it satisfies the equation $x \vee \neg x^{n} \equals 1$ for some $n$.
\end{corollary}

  The proof of the analogous theorem for $\IK$ will be slightly more complicated. Let us first show that the axiom of $n$-cyclicity can be given a slightly different but equivalent form in $\IKnfour$.

\begin{fact} \label{fact: k cyclicity}
  Let $\logic{L}$ be an extension of $\IK$ and $k \geq 1$. If $p \vee \Box_{1} \neg \Box_{n} p$ is a theorem of $\logic{L}$, then so is $p \vee \Box_{k} \neg \Box_{kn} p$.
\end{fact}

\begin{proof}
  Suppose that $p \vee \Box_{1} \neg \Box_{n} p$ and $p \vee \Box_{k} \neg \Box_{kn} p$ are theorems. Then substituting $\Box_{kn} p$ for $p$ in the former yields the theorem $\Box_{kn} p \vee \Box_{1} \neg \Box_{n} \Box_{kn} p$, which combined with $p \vee \Box_{k} \neg \Box_{kn} p$ yields the theorem $p \vee \Box_{k} \Box_{1} \neg \Box_{n} \Box_{kn} p$, or equivalently $p \vee \Box_{k+1} \neg \Box_{(k+1)n} p$.
\end{proof}

\begin{fact} \label{fact: cyclicity}
  Let $\logic{L}$ be an extension of $\IKnfour$ for $n \geq 1$. Then $p \vee \Box_{1} \neg \Box_{n} p$ is a theorem of $\logic{L}$ if and only if $\Box_{n} p \vee \Box_{n} \neg \Box_{n} p$ is.
\end{fact}

\begin{proof}
  The right-to-left direction is immediate given the definition of $\Box_{n}$. Conversely, if $p \vee \Box \neg \Box_{n} p$ is a theorem, then so is $p \vee \Box_{n} \neg \Box_{n^{2}} p$ by the previous fact. Substituting $\Box_{n} p$ for $p$ yields that $\Box_{n} p \vee \Box_{n} \neg \Box_{n^{2} + n} p$. Finally, $\Box_{n} p \rightarrow \Box_{n+1} p$ is a theorem because $\logic{L}$ extends $\IKnfour$, therefore $\Box_{n} p \rightarrow \Box_{k} p$ is a theorem for each $k \geq n$. Taking $k \assign n^{2} + n$ now yields the theorem $\Box_{n} p \rightarrow \Box_{n^{2} + n} p$, which combined with $\Box_{n} p \vee \Box_{n} \neg \Box_{n^{2} + n} p$ yields the theorem $\Box_{n} p \vee \Box_{n} \neg \Box_{n} p$.
\end{proof}

\begin{theorem}
  An axiomatic extension of $\IK$ is semisimple if and only if the formulas $p \vee \Box_{1} \neg \Box_{n} p$ and $\Box_{n} p \rightarrow \Box_{n+1} p$ are theorems of $\logic{L}$ for some $n \in \omega$.
\end{theorem}

\begin{proof}
  Let $\logic{L}$ be an axiomatic extension of $\IK$. Then $\logic{L}$ is finitary and it inherits the local DDT and local IL of $\IK$. By Theorem~\ref{thm: semisimplicity}, the semisimplicity of $\logic{L}$ is therefore equivalent to the LEM w.r.t.\ $\set{\neg \Box_{n} p}{n \in \omega}$.

  Suppose first that $\Box_{n} p \rightarrow \Box_{n+1} p$ and $p \vee \Box_{1} \neg \Box_{n} p$ are theorems of $\logic{L}$. Recall that the latter is equivalent to $\Box_{n} p \vee \Box_{n} \neg \Box_{n} p$ by Fact~\ref{fact: cyclicity}. Then the local DDT inherited by $\logic{L}$ from $\IK$ reduces to the global DDT w.r.t.\ $\Box_{n} p \rightarrow q$. It now suffices to show that $\logic{L}$ enjoys the global LEM w.r.t.\ $\neg \Box_{n} p$. Using the global DDT of $\logic{L}$, this global LEM is equivalent to the following implication:
\begin{prooftree}
  \def\fCenter{\vdash_{\logic{L}}}
  \Axiom$\Gamma \fCenter \Box_{n} \varphi \rightarrow  \psi$
  \Axiom$\Gamma \fCenter \Box_{n} \neg \Box_{n} \varphi \rightarrow \psi$
  \BinaryInf$\Gamma \fCenter \psi$
\end{prooftree}
  This implication holds if (and only if) $\Box_{n} p \rightarrow q, \Box_{n} \neg \Box_{n} p \rightarrow q \vdash_{\logic{L}} q$, i.e.\ if (and only if) $\Box_{n} p \vee \Box_{n} \neg \Box_{n} p$ is a theorem of $\logic{L}$.

  Conversely, suppose that $\logic{L}$ is semisimple. We first prove that $p \vee \Box_{1} \neg \Box_{n} p$ is a theorem of $\logic{L}$ for some $n \in \omega$. To this end, take $\Gamma \assign \set{q \rightarrow \Box_{k} p}{k \in \omega}$. Then $p, \Gamma \vdash_{\logic{L}} p \vee \Box_{1} \neg q$ and $\neg \Box_{k} p, \Gamma \vdash_{\logic{L}} \neg q \vdash_{\logic{L}} \Box_{1} \neg q \vdash_{\logic{L}} p \vee \Box_{1} \neg q$ for each $k \in \omega$, therefore $\Gamma \vdash_{\logic{L}} p \vee \Box_{1} q$ by the LEM. Because $q \rightarrow \Box_{k+1} p \vdash_{\logic{L}} q \rightarrow \Box_{k} p$, by finitarity there is some $n \in \omega$ such that $q \rightarrow \Box_{n} p \vdash_{\logic{L}} p \vee \Box_{1} \neg q$. Substituting $\Box_{n} p$ for $q$ now yields that $\Box_{n} p \vee \Box_{1} \neg \Box_{n} p$, hence also $p \vee \Box_{1} \neg \Box_{n} p$, is a theorem of $\logic{L}$.

  Consider now the function $f(k) = kn +1$ and take
\begin{align*}
  \Gamma & \assign \{ p \rightarrow q \} \cup \set{\Box_{f(k)} \neg \Box_{k} p \rightarrow q}{k \in \omega}.
\end{align*}
  Then $p, \Gamma \vdash_{\logic{L}} q$ and $\neg \Box_{k} p, \Gamma \vdash_{\logic{L}} q$ for each $k \in \omega$, therefore $\Gamma \vdash_{\logic{L}} q$ by the LEM. By finitarity there is some $m \geq 1$ such that
\begin{align*}
  p \rightarrow q, \set{\Box_{f(k)} \neg \Box_{k} p \rightarrow q}{k \leq m} \vdash_{\logic{L}} q.
\end{align*}
  Substituting the formula $\psi \assign p \vee \bigvee_{k \leq m} \Box_{f(k)} \neg \Box_{k} p$ for $q$ yields that $\psi$ is a theorem of $\logic{L}$. By Fact~\ref{fact: k cyclicity} the formula $p \vee \Box_{k} \neg \Box_{kn} p$ is a theorem of $\logic{L}$ for each $k \leq m$. Substituting $\neg \Box_{mn} p$ for $p$ in $\psi$ yields the theorem
\begin{align*}
  \neg \Box_{mn} p \vee \bigvee_{k \leq m} \Box_{f(k)} \neg \Box_{k} \neg \Box_{kn} \Box_{(m-k) n} p,
\end{align*}
  since $\Box_{mn} p$ is equivalent (in every context) to $\Box_{kn} \Box_{(m-k) n} p$ for each $k \leq m$. Applying the substitution instances $\Box_{(m-k) n} p \vee \Box_{k} \neg \Box_{kn} \Box_{(m-k) n} p$ of the theorem $p \vee \Box_{k} \neg \Box_{kn} p$ to the theorem displayed above yields
\begin{align*}
  \neg \Box_{mn} p \vee \bigvee_{k \leq m} \Box_{f(k)} \Box_{(m-k)n} p.
\end{align*}
  For our function $f(k) \assign kn + 1$ this is equivalent to $\neg \Box_{mn} p \vee \Box_{mn+1} p$, which implies $\Box_{mn} p \rightarrow \Box_{mn+1} p$. Of~course, if $p \vee \Box_{1} \neg \Box_{n} p$ is a theorem, so is $p \vee \Box_{1} \neg \Box_{mn} p$.

  (Let us also provide a semantic proof that $n$-cyclicity plus the complicated \mbox{theorem} $p \vee \bigvee_{k \leq m} \Box_{f(k)} \neg \Box_{k} p$ imply weak $n$-transitivity. We only state the proof for classical modal logic for the sake of simplicity, but the argument can be extended to the intuitionistic setting. These axiom of $n$-cyclicity and the theorem $p \vee \bigvee_{k \leq m} \Box_{f(k)} \neg \Box_{k} p$ are equivalent to so-called Sahlqvist formulas, therefore they correspond certain first-order conditions on Kripke frames and the logic $\logic{L}$ is complete with respect to the class of all Kripke frames satisfying these conditions (see e.g.~\cite[Sections~3.6~and~4.3]{blackburn+derijke+venema02}). Using the notation introduced in Subsection~\ref{subsec: running examples modal}, \mbox{$n$-cyclicity} corresponds to the condition that $uRw$ implies $w R^{\smallleq k}u$. The~more complicated theorem corresponds to the following condition:
\begin{align*}
  \text{for each $u \in W$ there is some $n \leq m$ such that $u R^{\smallleq nk+1} w$ implies $w R^{\smallleq n} u$}.
\end{align*}
  But by the first condition $w R^{\smallleq n} u$ implies $u R^{\smallleq nk} w$, therefore for each $u \in W$ there is some $n \leq m$ such that $u R^{\smallleq nk+1} w$ implies that $u R^{\smallleq nk} w$. It follows that $u R^{\smallleq mk + 1} w$ implies that $u R^{\smallleq mk} w$. Informally, all of the alternatives imply that each for each path of length $mk+1$ there is a shortcut of length at most $mk$. The formula $\Box_{mk} p \rightarrow \Box_{mk+1} p$ is therefore a theorem of $\logic{L}$.)
\end{proof}

\begin{corollary}
  A variety of modal Heyting algebras is semisimple if and only if it satisfies the equations $1 \equals x \vee \Box_{1} \neg \Box_{n} x$ and $\Box_{n} x \inequals \Box_{n+1} x$ for some $n \in \omega$.
\end{corollary}

\begin{corollary}[\cite{kowalski+kracht06}]
  A variety of modal algebras is semisimple if and only if it satisfies the equations $x \inequals \Box \Diamond_{n} x$ and $\Box_{n} x \inequals \Box_{n+1} x$ for some $n \in \omega$.
\end{corollary}

\section{Glivenko theorems and semisimple companions}
\label{sec: glivenko}

  In this section we study the (syntactic) semisimple companion of a logic, which for compact logics is the largest extension with the same inconsistent sets of formulas as the original logic. The rules valid in the semisimple companion can be described in terms of the antitheorems of the original logic. This leads to what we call antiadmissible rules, by analogy with the admissible rules of a logic. In fact, the semisimple companion can be thought of as a construction dual to the structural completion, which is the largest extension with the same theorems as the original logics.

  Under favourable circumstances, the models of the semisimple companion are precisely the semisimple models of the original logic. This fact can be applied to describe the semisimple algebras in a given variety, provided that we know how to axiomatize the semisimple companion of the corresponding logic. This~strategy works for logics with a well-behaved implication or disjunction (which satisfy the global DDT or the global PCP). For such logics the semisimple companion is axiomatized by an axiomatic form of the LEM.

  The description of antiadmissible will then allow us to establish a Glivenko-like connection between a logic and its semisimple companion, subsuming several existing Glivenko theorems in the literature. In~particular, we describe all extensions of $\FLe$ which enjoy a Glivenko-like connection to classical logic.

\subsection{The semisimple companion of a logic}
\label{subsec: semisimple companion}

  Let us start by recalling the relevant definitions from the previous section. A \emph{simple} theory of~$\logic{L}$ is a largest non-trivial theory of~$\logic{L}$, while a \emph{semisimple} theory is an intersection of simple theories. A~logic~$\logic{L}$ is called \emph{(syntactically) semisimple} if each theory of $\logic{L}$ is semisimple. The \emph{(syntactic) semisimple companion} of $\logic{L}$, denoted $\ssc{\logic{L}}$, is then defined as the logic determined by the simple, or equivalently semisimple, theories of~$\logic{L}$. That is, it is the logic determined by the matrices $\langle \FmAlg \logic{L}, T \rangle$ where $T$ ranges over the simple, or equivalently semisimple, theories of $\logic{L}$.  Note that it is not obvious from the definition whether this companion is a semisimple logic, although it is always complete with respect to its simple theories.

  The semisimple companion can only be expected to have a significant relation to the original logic if each non-trivial theory may be extended to a simple theory, a condition that we call \emph{(syntactic) coatomicity}. Otherwise, much information about the theories of the original logic may be lost. This condition will be assumed in almost all of the results of this section. Recall that in particular each compact logic is coatomic.

\begin{fact} \label{fact: largest extension with same simple theories}
  Let $\logic{L}$ be a coatomic logic. Then $\ssc{\logic{L}}$ is the largest logic with the same simple theories as $\logic{L}$. In particular, $\ssc{\logic{L}}$ is also coatomic.
\end{fact}

\begin{proof}
  If $\logic{L}'$ has the same simple theories as~$\logic{L}$, then each matrix of the form $\langle \FmAlg \logic{L}, T \rangle$ for some simple $\logic{L}$-theory $T$ is a model of $\logic{L}'$, hence $\logic{L}' \logleq \ssc{\logic{L}}$. Clearly each simple theory of $\logic{L}$ is a simple theory of $\ssc{\logic{L}}$. Conversely, each simple theory $T$ of $\ssc{\logic{L}}$ is a theory of $\logic{L}$, so by coatomicity it extends to a simple theory $T'$ of $\logic{L}$, which is by definition a theory of $\ssc{\logic{L}}$. But $T$ is a simple theory of $\ssc{\logic{L}}$, therefore $T' = T$ and $T$ is $\logic{L}$-simple.
\end{proof}

\begin{fact}
  Let $\logic{L}$ be a coatomic logic. An~extension $\logic{L}'$ of $\logic{L}$ has the same simple theories as $\logic{L}$ if and only if it satisfies the equivalence
\begin{align*}
  \Gamma \vdash_{\logic{L}} \Fm \logic{L} \iff \Gamma \vdash_{\logic{L}'} \Fm \logic{L}.
\end{align*}
\end{fact}

\begin{proof}
  If $\logic{L}$ and $\logic{L}'$ have the same simple theories and ${\Gamma \nvdash_{\logic{L}} \Fm \logic{L}}$, then by the coatomicity of $\logic{L}$ there is a simple $\logic{L}$-theory, hence also a simple $\logic{L}'$-theory, $T \supseteq \Gamma$. But $T \nvdash_{\logic{L}'} \Fm \logic{L}$, hence also $\Gamma \nvdash_{\logic{L}'} \Fm \logic{L}$. Of course $\Gamma \vdash_{\logic{L}} \Fm \logic{L}$ implies $\Gamma \vdash_{\logic{L}'} \Fm \logic{L}$ for each extension $\logic{L}'$ of $\logic{L}$.

  Conversely, if $\logic{L}'$ satisfies the equivalence and $T$ is a simple theory of $\logic{L}$, then by the equivalence $T \nvdash_{\logic{L}'} \Fm \logic{L}$. Let $T'$ be the $\logic{L}'$-theory generated by $T$. Then $T'$ is a non-trivial $\logic{L}$-theory extending $T$, hence $T' =T$ by the $\logic{L}$-simplicity of $T$. Thus $T$ is an $\logic{L}'$-theory, and in fact a simple $\logic{L}'$-theory because $\logic{L}'$ is an extension of $\logic{L}$. On the other hand, if $T$ is a simple theory of $\logic{L}'$, then it is a theory of $\logic{L}$. If $T'$ is a non-trivial theory of $\logic{L}$ which properly extends $T$, then by the assumed equivalence $T' \nvdash_{\logic{L}'} \Fm \logic{L}$, therefore $T' = T$ by the $\logic{L}'$-simplicity of $T$. Thus $T$ is a simple $\logic{L}$-theory.
\end{proof}

\begin{corollary}
  Let $\logic{L}$ be a coatomic logic. Then the semisimple companion $\ssc{\logic{L}}$ is the largest extension of $\logic{L}$ such that
\begin{align*}
  \Gamma \vdash_{\logic{L}} \Fm \logic{L} \iff \Gamma \vdash_{\ssc{\logic{L}}} \Fm \logic{L}.
\end{align*}
\end{corollary}

  If a logic fails to be coatomic, then pathological behaviour can occur. In the extreme case, it may lack simple theories altogether. For example, consider the logic with a single binary operation $x \cdot y$ axiomatized by the rules
\begin{align*}
  p \cdot q & \vdash p, & p \cdot (q \cdot r) & \vdash (p \cdot q) \cdot r, & (p \cdot q) \cdot r & \vdash p \cdot (q \cdot r).
\end{align*}
  The lattice of theories of this logic is isomorphic to the lattice of sets of words over the given variables closed under taking initial segments. This lattice has no coatoms, so the semisimple companion of this logic is the trivial logic. It~may therefore even happen that $\emptyset \in \Th \logic{L}$ while $\emptyset \vdash_{\ssc{\logic{L}}} \Fm \logic{L}$.

  Among logics which have an antitheorem, such pathological behaviour can only occur if the logic is not finitary, since each non-trivial finitary logic with an antitheorem is compact, hence coatomic. One fairly natural example is the following expansion of G\"{o}del's well-known super-intuitionistic logic.\footnote{This logic in fact lacks subdirectly irreducible models altogether (see~\cite[Section~3.2]{lavicka+noguera17}).} 

\newcommand{\Galg}{\alg{G}^{\scriptstyle\mathbb{Q}}}
\newcommand{\Glog}{\logic{GO}^{\scriptstyle\mathbb{Q}}}
\newcommand{\Glogsmall}{\logic{GO}^{\scriptscriptstyle\mathbb{Q}}}

  Consider the real interval $[0, 1]$ with the natural order viewed as a Heyting algebra (i.e.\ as a bounded distributive lattice with the binary operation $\rightarrow$ such that $x \rightarrow y = y$ if $y < x$ and $x \rightarrow y = 1$ otherwise) expanded by a constant~$c_{q}$ interpreted by $q$ for each rational $q \in [0, 1]$. We denote this algebra $\Galg$ and define the \emph{G\"{o}del logic of order with rational constants} $\Glog$ as the logic (over some given set of variables) in the signature of the algebra $\Galg$ such that
\begin{align*}
  \Gamma \vdash_{\Glogsmall} \varphi & \iff \bigwedge h[\Gamma] \leq h(\varphi) \text{ for each homomorphism } h\colon \FmAlg \Glog \to \Galg.
\end{align*}
  The presence of infinitely many constants is not an essential feature of this example. We could make do with the unary function $x \mapsto \frac{x}{2}$ instead.

\begin{example} \label{ex: goq}
  G\"{o}del logic of order with rational constants has no simple theories and no largest extension with the same antitheorems.
\end{example}

\begin{proof}
  Consider the logics $\Glog_{q}$ for $q \in [0, 1] \cap \mathbb{Q}$ defined (over the same algebra of formulas) as
\begin{align*}
  \Gamma \vdash_{\Glogsmall_{q}} \varphi & \iff \bigwedge h[\Gamma] \wedge q \leq h(\varphi) \text{ for each hom.\ } h\colon \FmAlg \Glog \to \Galg.
\end{align*}
  Clearly $\Glog_{1} = \Glog$. Moreover, $\Glog_{r}$ is a proper extension of $\Glog_{q}$ if $r < q$ because $\emptyset \vdash_{\Glogsmall_{r}} c_{q}$ if and only if $r < q$. Finally, $\Gamma$ is an antitheorem of $\Glog_{q}$ for $q > 0$ if and only if $\bigwedge h[\Gamma] = 0$ for each homomorphism $h\colon \FmAlg \Glog \to \Galg$.

  It follows that adding any finite set of axioms $\emptyset \vdash c_{q}$ for $q > 0$ to $\Glog$ does not change the set of antitheorems, since it results in a logic below some $\Glog_{q}$ with $q > 0$. However, adding all of these axioms results in the logic becoming trivial thanks to the antitheorem $\set{c_{q}}{q \in (0, 1] \cap \mathbb{Q}}$.

  If $T$ is a non-trivial theory of $\Glog$, then $T \nvdash_{\Glogsmall} c_{q}$ for some constant~$c_{q}$ with $q > 0$. Because $T$ is not an antitheorem of $\Glog$, it is also not an antitheorem of $\Glog_{q}$, hence the theory of $\Glog_{q}$ generated by $T$ is a non-trivial theory of $\Glog$ properly extending $T$ by virtue of containing $c_{q}$.
\end{proof}

  Contrast this with the situation for theorems rather than antitheorems: each logic~$\logic{L}$, finitary or not, is guaranteed to have a largest extension with the same set of theorems $\Thm \logic{L}$, called the \emph{structural completion} of $\logic{L}$, which is defined by the single matrix $\langle \FmAlg \logic{L}, \Thm \logic{L} \rangle$. The reader may observe a natural duality between the semisimple companion defined the maximal (non-trivial) theories and the structural completion defined by the unique minimal theory. Indeed, the semisimple companion may well be called the \emph{antistructural completion}: if~a~coatomic logic~$\logic{L}$ has an antitheorem, then $\ssc{\logic{L}}$ is precisely the largest extension of $\logic{L}$ with the same set of antitheorems.

\subsection{Antiadmissible rules}
\label{subsec: antiadmissible rules}

  The rules valid in the structural completion of a logic~$\logic{L}$ can be described in terms of the theorems of~$\logic{L}$: they are precisely the \emph{admissible rules} of $\logic{L}$, that is, rules $\Gamma \vdash \varphi$ such that
\begin{align*}
  \emptyset \vdash_{\logic{L} }\sigma[\Gamma] \implies \emptyset \vdash_{\logic{L}} \sigma(\varphi) \text{ for each substitution } \sigma.
\end{align*}
  Given the duality between the structural completion and the semisimple companion, we now wish to provide a similar description of the rules valid in the semisimple companion~$\ssc{\logic{L}}$ in terms of the antitheorems of~$\logic{L}$.

  Throughout this section only, we shall abuse notation and talk about the rule $\Gamma \vdash \tuple{\varphi}$, i.e.\ the set of rules $\Gamma \vdash \varphi_{i}$ for $\varphi_{i} \in \tuple{\varphi}$, where $\tuple{\varphi}$ is an $\alpha$-tuple for $0 < \alpha < \kappa$. This is not an essential point: the~reader is free to substitute $\varphi$ for $\tuple{\varphi}$ throughout this chapter.
  
\begin{definition}
  A rule $\Gamma \vdash \tuple{\varphi}$ will be called \emph{antiadmissible} in $\logic{L}$ if
\begin{align*}
  \sigma(\tuple{\varphi}), \Delta \vdash \Fm \logic{L} \implies \sigma[\Gamma], \Delta \vdash \Fm \logic{L} \text{ for each substitution } \sigma \text{ and } \Delta \subseteq \Fm \logic{L},
\end{align*}
  or equivalently if
\begin{align*}
  \sigma(\tuple{\varphi}), T \vdash \Fm \logic{L} \implies \sigma[\Gamma], T \vdash \Fm \logic{L} \text{ for each substitution } \sigma \text{ and } T \in \Th \logic{L}.
\end{align*}
\end{definition}

\begin{proposition} \label{prop: antiadmissible rules}
  A rule $\Gamma \vdash \tuple{\varphi}$ is antiadmissible in a coatomic logic $\logic{L}$ if and only if $\Gamma \vdash_{\ssc{\logic{L}}} \tuple{\varphi}$.
\end{proposition}

\begin{proof}
  If the rule $\Gamma \vdash \tuple{\varphi}$ is valid in $\ssc{\logic{L}}$, then ${\sigma(\tuple{\varphi}), \Delta \vdash_{\ssc{\logic{L}}} \Fm \logic{L}}$ implies $\sigma[\Gamma], \Delta \vdash_{\ssc{\logic{L}}} \Fm \logic{L}$ by structurality and cut. Applying the equivalence $\Gamma \vdash_{\logic{L}} \Fm \logic{L} \iff \Gamma \vdash_{\ssc{\logic{L}}} \Fm \logic{L}$ now yields the antiadmissibility of $\Gamma \vdash \tuple{\varphi}$. Conversely, suppose that $\Gamma \vdash \tuple{\varphi}$ is not valid in $\ssc{\logic{L}}$. Then there is a simple $\logic{L}$-theory $T$ such that $\Gamma \vdash \tuple{\varphi}$ fails in the matrix $\langle \FmAlg \logic{L}, T \rangle$, i.e.\ there is a substitution $\sigma$ such that $\sigma[\Gamma] \subseteq T$ and $\sigma(\tuple{\varphi}) \nsubseteq T$, hence $\sigma(\tuple{\varphi}), T \vdash_{\logic{L}} \Fm \logic{L}$ but $\sigma[\Gamma], T \nvdash_{\logic{L}} \Fm \logic{L}$.
\end{proof}

  A rule $\Gamma \vdash \varphi$ is admissible in a logic $\logic{L}$ if and only if the extension of $\logic{L}$ by this rule has the same theorems as $\logic{L}$. An analogous observation can be made about antiadmissible rules in coatomic logics with an antitheorem.

\begin{corollary}
  A rule $\Gamma \vdash \varphi$ is antiadmissible in a coatomic logic $\logic{L}$ with an antitheorem if and only if extending $\logic{L}$ by this rule does not add any antitheorems.
\end{corollary}

  Without the assumption of coatomicity, we can still show that at least the finitary antiadmissible rules form a finitary logic.

\begin{fact}
  The finitary antiadmissible rules of a logic form a finitary logic. More precisely, the set of all rules $\Gamma \vdash \varphi$ such that $\Gamma' \vdash \varphi$ is antiadmissible in $\logic{L}$ for some finite $\Gamma' \subseteq \Gamma$ is a finitary logic.
\end{fact}

\begin{proof}
  Reflexivity, monotonicity, structurality, and finitarity hold by definition. Moreover, antiadmissible rules are closed under finitary cut: if the rules $\Gamma \vdash \varphi$ and $\varphi, \Delta \vdash \psi$ are antiadmissible, then
\begin{align*}
  \sigma(\psi), \Phi \vdash \Fm \logic{L} \implies \sigma(\varphi), \sigma[\Delta], \Phi \vdash \Fm \logic{L} \implies \sigma[\Gamma], \sigma[\Delta], \Phi \vdash \Fm \logic{L},
\end{align*}
  therefore the rule $\Gamma, \Delta \vdash \psi$ is antiadmissible in $\logic{L}$. Full cut now follows from finitarity and finitary cut.
\end{proof}

  The above fact fails if we look at all antiadmissible rules. These are closed under reflexivity, monotonicity, structurality, and finitary cut, but not under full cut. For example, in the G\"{o}del logic of order with rational constants $\Glog$ (recall Example~\ref{ex: goq}) each of the rules $\emptyset \vdash c_{q}$ for $q \in (0, 1] \cap \mathbb{Q}$ is antiadmissible, as is the valid rule $\set{c_{q}}{q \in (0, 1] \cap \mathbb{Q}} \vdash c_{0}$, but the rule $\emptyset \vdash c_{0}$ is not antiadmissible, since this would imply that $\Glog$ is the trivial logic.

  It may be instructive to contrast this with the situation of admissible rules. Essentially, a rule is admissible if all of its instances lead only from theorems to theorems. If the terminal nodes of a proof (which is a well-founded tree) are theorems and all rules applied in the proof preserve theoremhood, then by induction on well-founded trees the conclusion must be a theorem too. However, the analogous argument for antiadmisible rules fails: if the conclusion of a proof is an antitheorem, and all rules reflect antitheoremhood in each context (if $\Gamma \vdash \varphi$ is an instance of the rule and $\varphi, \Delta$ is an antitheorem, then so is $\Gamma, \Delta$), then we cannot infer that the set of all terminal nodes of the proof forms an antitheorem: induction on well-founded trees does not work in this direction. Such an inference only works for finite proofs, hence the above fact. Note that a finitary logic may well have infinitary antiadmissible rules, therefore restricting to antiadmissible rules in finitary logics does not help.

  In the rest of this subsection, we show that in many cases the definition of an antiadmissible rule may be simplified. Indeed, the double quantification over contexts $\Delta$ and substitutions $\sigma$ makes the definition of an antiadmissible rule more complicated than the analogous definition of an admissible rule. Remarkably, it turns out that we can often do without the quantification over substitutions. This has no analogue in the case of admissible rules, where quantification over substitutions is of course essential.

  The proof will rely on the (parametrized) local ILs introduced in Section~\ref{sec: ils and ddts}. Recall that in Proposition~\ref{prop: swapping for antitheorems} we proved that the $\kappa$-ary (parametrized) local IL is equivalent in a logic with an antitheorem to (surjective) substitution swapping for antitheorems: for each (surjective) substitution~$\sigma$, each theory $T$ of a logic $\logic{L}$ with an antitheorem, and each $\alpha$-tuple of formulas $\tuple{\varphi}$ for $0 < \alpha < \kappa$
\begin{align*}
  \sigma(\tuple{\varphi}), T \vdash_{\logic{L}} \allset \implies \tuple{\varphi}, \sigma^{-1}[T] \vdash_{\logic{L}} \allset.
\end{align*}

  In the statement of the following proposition, an invertible substitution is a substitution~$\iota$ such that $\tau \circ \iota$ is the identity map on $\Fm \logic{L}$ for some substitution~$\tau$. Invertible substitutions are precisely those injective substitutions $\tau$ such that $\tau(p) \in \Var \logic{L}$ for each variable $p$.

\begin{proposition} \label{prop: invertible antiadmissibility}
  Let $\logic{L}$ be a coatomic logic with the unary parametrized local IL. Then $\Gamma \vdash \tuple{\varphi}$ is antiadmissible in $\logic{L}$ if and only~if
\begin{align*}
  \iota(\tuple{\varphi}), \Delta \vdash_{\logic{L}} \allset & \implies \iota[\Gamma], \Delta \vdash_{\logic{L}} \allset \text{ for each invertible substition } \iota \text{ and each } \Delta.
\end{align*}
\end{proposition}

\begin{proof}
  Suppose that $\Gamma \vdash \tuple{\varphi}$ satisfies the implication and $\sigma (\tuple{\varphi}), T \vdash_{\logic{L}} \allset$ for some theory $T$ of $\logic{L}$. Let us first construct an invertible substitution $\iota$ and a surjective substitution~$\tau$ such that $\sigma = \tau \circ \iota$.

  To define $\iota$ and $\tau$, pick $X \subseteq \Var \logic{L}$ such that $| {X} | = |{\Var \logic{L}}| = | {\Var \logic{L} \setminus X} |$, and an injective map $f\colon \Var \logic{L} \rightarrow X$ and a surjective map ${g\colon X \rightarrow \Var \logic{L}}$ such that $g \circ f$ is the identity map on $\Var \logic{L}$. The map $f$ extends to an invertible substitution $\iota$ with an inverse $\lambda$ such that $\lambda \circ \iota$ is the identity map on $\Fm \logic{L}$. We can then define $\tau$ such that the restriction $\tau\colon X \rightarrow \Fm \logic{L}$ is defined as $\tau (p) = (\sigma \circ \lambda) (p)$ for $p \in X$ and the restriction $\tau\colon \Var \logic{L} \setminus X \rightarrow \Var \logic{L}$ is surjective. Then $\tau$ is a surjective substitution such that $\sigma = \tau \circ \iota$.

  The assumption $\sigma(\tuple{\varphi}), T \vdash_{\logic{L}} \allset$, i.e.\ $(\tau \circ \iota) (\tuple{\varphi}), T \vdash_{\logic{L}} \allset$, implies ${\iota (\tuple{\varphi}), \tau^{-1}[T] \vdash_{\logic{L}} \allset}$ by surjective substitution swapping for antitheorems. The assumed implication now yields ${\iota[\Gamma], \tau^{-1}[T] \vdash_{\logic{L}} \allset}$. Finally, we may apply the substitution $\tau$ to obtain $(\tau \circ \iota)[\Gamma], \tau[\tau^{-1}[T]] \vdash_{\logic{L}} \allset$, hence $\sigma[\Gamma], T \vdash_{\logic{L}} \allset$.
\end{proof}

  For the purposes of the following proposition, a rule $\Gamma \vdash \tuple{\varphi}$ is said to \emph{omit enough variables} if $| {\Var \logic{L} \setminus \Var (\Gamma, \tuple{\varphi})} | = | {\Var \logic{L}} |$, where $\Var \Delta$ denotes the set of all variables which occur in $\Delta$. In other words, it leaves $| {\Var \logic{L}} |$ variables unused. This in particular covers all finitary rules $\Gamma \vdash \varphi$.
  
\begin{proposition}
  Let $\logic{L}$ be a coatomic logic with a $\kappa$-ary parametrized local IL. Then a rule $\Gamma \vdash \tuple{\varphi}$ which omits enough variables is antiadmissible in~$\logic{L}$ if and only~if
\begin{align*}
  \tuple{\varphi}, \Delta \vdash_{\logic{L}} \allset & \implies \Gamma, \Delta \vdash_{\logic{L}} \allset \text{ for each } \Delta \subseteq \Fm \logic{L}.
\end{align*}
\end{proposition}

\begin{proof}
  Let $\Gamma \vdash \tuple{\varphi}$ be a rule satisfying the implication, let $\iota$ be an invertible substitution, and let $\iota(\tuple{\varphi}), T \vdash_{\logic{L}} \allset$ for some theory $T$ of $\logic{L}$. By the previous proposition it suffices to show that $\iota[\Gamma], T \vdash_{\logic{L}} \allset$.

  The substitution $\iota$ is restricts to an injective map $f\colon \Var \logic{L} \rightarrow \Var \logic{L}$. Since $| {\Var \logic{L} \setminus \Var (\Gamma, \tuple{\varphi})} | = | {\Var \logic{L}} | = \kappa$, there is a \emph{bijective} map ${g\colon \Var \logic{L} \rightarrow \Var \logic{L}}$ which agrees with $f$ on $\Var (\Gamma, \tuple{\varphi})$. The bijection~$g$ extends to a bijective substitution $\sigma$ of $\FmAlg \logic{L}$ such that $\sigma$ agrees with $\iota$ on $\Gamma$ and $\tuple{\varphi}$. But $\iota(\tuple{\varphi}), T \vdash_{\logic{L}} \allset$ implies $\sigma(\tuple{\varphi}), T \vdash_{\logic{L}} \allset$, hence $\tuple{\varphi}, \sigma^{-1}[T] \vdash_{\logic{L}} \allset$ by surjective substitution swapping for antitheorems. The assumed implication yields $\Gamma, \tau^{-1}[T] \vdash_{\logic{L}} \allset$. We may apply the substitution $\tau$ to obtain $\tau[\Gamma], T \vdash_{\logic{L}} \allset$, i.e.\ $\sigma[\Gamma], T \vdash_{\logic{L}} \allset$.
\end{proof}

  In other words, for all but the most pathological antiadmissible rules which use up almost all variables, we may omit the quantification over substitutions~$\sigma$ from the definition of anti\-admissibility. If the logic enjoys a local IL, then no such technical condition is necessary.

\begin{proposition} \label{prop: simply antiadmissible rules}
  Let $\logic{L}$ be a coatomic logic with a simple local IL. Then a rule $\Gamma \vdash \tuple{\varphi}$ is antiadmissible in $\logic{L}$ if and only~if
\begin{align*}
  \tuple{\varphi}, \Delta \vdash_{\logic{L}} \allset & \implies \Gamma, \Delta \vdash_{\logic{L}} \allset \text{ for each } \Delta \subseteq \Fm \logic{L}.
\end{align*}
  Moreover, we may restrict $\Delta$ to range over the simple theories of $\logic{L}$.
\end{proposition}

\begin{proof}
  If $\sigma[\Gamma], \Delta \nvdash_{\logic{L}} \allset$, then by coatomicity $\sigma[\Gamma], \Delta$ extends to a simple theory $T$, hence $T \vdash_{\logic{L}} \sigma[\Gamma]$ and $\sigma^{-1}[T] \vdash_{\logic{L}} \Gamma$. By assumption $\Gamma, \sigma^{-1}[T] \nvdash_{\logic{L}} \allset$ implies $\varphi, \sigma^{-1}[T] \nvdash_{\logic{L}} \allset$. But $\sigma^{-1}[T]$ is simple by the simple local IL, therefore $\sigma^{-1}[T] \vdash_{\logic{L}} \varphi$ and $T \vdash_{\logic{L}} \sigma(\varphi)$. It follows that $\sigma(\varphi), \Delta \nvdash_{\logic{L}} \allset$.
\end{proof}

  This last result gives us a convenient syntactic handle on the valid rules of $\ssc{\logic{L}}$ in terms of the antitheorems of $\logic{L}$, which will later be leveraged to yield Glivenko theorems connecting $\logic{L}$ and $\ssc{\logic{L}}$.


\subsection{Semisimple theories and models of $\logic{L}$}

  Having linked the valid rules of the semisimple companion of $\logic{L}$ with the antitheorems of $\logic{L}$, we now wish to relate the theories and models of the semisimple companion of $\logic{L}$ with the semisimple theories and models of $\logic{L}$. More precisely, we want to identify sufficient conditions under which these classses of theories and models coincide.

\begin{proposition} \label{prop: semisimple companion is semisimple}
  Let $\logic{L}$ be coatomic with the simple local IL. Then the theories of $\ssc{\logic{L}}$ are precisely the semisimple theories of $\logic{L}$. In~particular, $\ssc{\logic{L}}$ is semisimple.
\end{proposition}

\begin{proof}
  Each semisimple theory of $\logic{L}$ is a theory of $\ssc{\logic{L}}$ because each simple theory of $\logic{L}$ is by definition a theory of $\ssc{\logic{L}}$ and the theories of $\ssc{\logic{L}}$ are closed under intersections. Conversely, let $T$ be a theory of $\ssc{\logic{L}}$ such that $T \nvdash_{\ssc{\logic{L}}} \varphi$. To prove the proposition it suffices to find a simple $\logic{L}$-theory $U$ extending $T$ such that $U \nvdash_{\logic{L}} \varphi$. But $T \nvdash_{\ssc{\logic{L}}} \varphi$ implies by Proposition~\ref{prop: simply antiadmissible rules} that there is an $\logic{L}$-theory $V$ such that $\varphi, V \vdash_{\logic{L}} \allset$ and $T, V \nvdash_{\logic{L}} \allset$. By coatomicity the set $T, V$ extends to a simple $\logic{L}$-theory $U$. Then $\varphi, V \vdash_{\logic{L}} \allset$ and $U, V \nvdash_{\logic{L}} \allset$, therefore $U \nvdash_{\logic{L}} \varphi$.
\end{proof}

  In fact, $\ssc{\logic{L}}$ inherits the local IL of $\logic{L}$ and upgrades it to a classical one.

\begin{proposition} \label{prop: theories of semisimple companion} \label{prop: cil for semisimple companion}
  Let $\logic{L}$ be coatomic with the $\kappa$-ary par.\ \mbox{local} IL w.r.t.~$\ilfamily$. Then $\ssc{\logic{L}}$ enjoys the $\kappa$-ary par.\ local IL w.r.t.~$\ilfamily$. If~$\logic{L}$ moreover has the simple local IL, then $\ssc{\logic{L}}$ enjoys the $\kappa$-ary classical par.\ local IL w.r.t.\ $\ilfamily$.
\end{proposition}

\begin{proof}
  Suppose that $\Gamma, \tuple{\varphi} \vdash_{\ssc{\logic{L}}} \allset$. Then $\Gamma, \tuple{\varphi} \vdash_{\logic{L}} \allset$, therefore by the IL $\Gamma \vdash_{\logic{L}} \ilset(\tuple{\varphi}, \tuple{\pi})$ for some $\ilset \in \ilfamily_{\alpha}$ and some $\tuple{\pi}$. It follows that $\Gamma \vdash_{\ssc{\logic{L}}} \ilset(\tuple{\varphi}, \tuple{\pi})$. On the other hand, if $\Gamma \vdash_{\ssc{\logic{L}}} \ilset(\tuple{\varphi}, \tuple{\pi})$ for some $\ilset \in \ilfamily_{\alpha}$ and some $\tuple{\pi}$, then $\Gamma, \tuple{\varphi} \vdash_{\ssc{\logic{L}}} \allset$ because $\tuple{\varphi}, \ilset(\tuple{\varphi}, \tuple{\pi}) \vdash_{\logic{L}} \allset$ implies $\tuple{\varphi}, \ilset(\tuple{\varphi}, \tuple{\pi}) \vdash_{\ssc{\logic{L}}} \allset$.

  Suppose now that $\logic{L}$ enjoys a simple local IL and $\Gamma \nvdash_{\ssc{\logic{L}}} \varphi$. Then by Proposition~\ref{prop: simply antiadmissible rules} there is some simple theory $T$ such that $T, \varphi \vdash_{\logic{L}} \allset$ but $T, \Gamma \nvdash_{\logic{L}} \allset$. The simple local IL w.r.t.\ $\ilfamily'$ yields $T \vdash_{\logic{L}} \ilset'(\varphi)$ for some $\ilset' \in \ilfamily'_{1}$, therefore $\Gamma, \ilset'(\varphi) \nvdash_{\logic{L}} \allset$ and $\Gamma, \ilset'(\varphi) \nvdash_{\ssc{\logic{L}}} \allset$.
\end{proof}

  We now show that, under suitable assumptions, the models of $\ssc{\logic{L}}$ are the semisimple models of $\logic{L}$.

\begin{lemma}
  Let $\logic{L}$ be a logic with the semantic simple local IL. If $h\colon \alg{A} \to \alg{B}$ is a homomorphism and~$\pair{\alg{B}}{G}$ is a simple model of $\logic{L}$, then so is $\pair{\alg{A}}{h^{-1}[G]}$.
\end{lemma}

\begin{proof}
  If $a \notin h^{-1}[G]$, then $h^{-1}[G], a \vdash_{\logic{L}}^{\alg{A}} A$:
\begin{align*}
  a \notin h^{-1}[G] & \implies h(a) \notin G \\
  & \implies G, h(a) \vdash_{\logic{L}}^{\alg{A}} B \text{ (by the simplicity of $\pair{\alg{B}}{G}$)} \\
  & \implies G \vdash_{\logic{L}} \ilset(h(a)) \text{ for some } \ilset \in \ilfamily_{1} \text{ (by the semantic simple IL)} \\
  & \implies h[\ilset(a)] \subseteq G \text{ for some } \ilset \in \ilfamily_{1} \text{ (because $h$ is a homomorphism)} \\
  & \implies \ilset(a) \subseteq h^{-1}[G] \text{ for some } \ilset \in \ilfamily_{1} \\
  & \implies h^{-1}[G], a \vdash_{\logic{L}}^{\alg{A}} A \text{ (by the semantic IL)}.
\end{align*}
  It follows that $h^{-1}[F]$ is a simple $\logic{L}$-filter on $\alg{A}$.
\end{proof}

\begin{lemma} \label{lemma: simple models of semisimple companion}
  Let $\logic{L}$ be a semantically coatomic logic with the semantic simple local IL. Then the simple models of $\ssc{\logic{L}}$ are precisely the simple models of $\logic{L}$.
\end{lemma}

\begin{proof}
  By the previous lemma, if $\pair{\alg{A}}{F}$ is a simple model of $\logic{L}$, then for each homomorphism $h\colon \FmAlg \logic{L} \to \alg{A}$ the set of formulas $h^{-1}[F]$ is a simple theory of~$\logic{L}$, and thus a theory of $\ssc{\logic{L}}$. It follows that $\pair{\alg{A}}{F}$ is a (simple) model of $\ssc{\logic{L}}$.

  Conversely, if $\pair{\alg{A}}{F}$ is a simple model of $\ssc{\logic{L}}$, then by the semantic coatomicity of $\logic{L}$ the filter $F$ extends to a simple filter $G$ of $\logic{L}$, and by the implication proved in the previous paragraph $G$ is a non-trivial filter of $\ssc{\logic{L}}$. Because $F$ was a simple filter of $\ssc{\logic{L}}$, it follows that $G = F$, therefore $F$ is a simple filter of $\logic{L}$.
\end{proof}

\begin{theorem} \label{thm: models of semisimple companion}
  Let $\logic{L}$ be a compact logic which enjoys the $\kappa$-ary local IL w.r.t.~a family $\ilfamily$ such that $|\ilfamily_{\alpha}| \leq |{\Var \logic{L}}|$ for each $\ilfamily_{\alpha}$ and each set in $\ilfamily_{1}$ has cardinality less than $\kappa$. Then the models of $\ssc{\logic{L}}$ are precisely the semi\-simple models of $\logic{L}$. In particular, $\ssc{\logic{L}}$ is semantically semisimple.
\end{theorem}

\begin{proof}
  By Lemma~\ref{lemma: simple models of semisimple companion} it suffices to show that $\ssc{\logic{L}}$ is semantically semisimple. We know that it enjoys the unary classical local IL w.r.t.~$\ilfamily$ by Proposition~\ref{prop: cil for semisimple companion}. Semantic semisimplicity follows by Theorem~\ref{thm: semisimplicity}.
\end{proof}

  In other words, whenever the above theorem applies, to describe the semisimple models of a logic $\logic{L}$ it suffices to axiomatize $\ssc{\logic{L}}$. We now show that under certain conditions this can be done in an entirely mechanical manner, since $\ssc{\logic{L}}$ turns out to be the extension of $\logic{L}$ by an axiomatic form of the~LEM. We~illustrate this on $\FLen$ and $\IKnfour$.

  This strategy works whenever the logic in question enjoys the global IL and has either a well-behaved implication or a well-behaved disjunction, obeying either a global DDT or a so-called global proof by cases property (PCP). Accordingly, we call the resulting sets of axioms the \emph{DDT--axiomatic form of the LEM} and the \emph{PCP--axiomatic form of the LEM}.

  In the following, $p \genrightarrow q$ and $\genbot$ shall denote sets of formulas in the same way that $\Gamma(p, q)$ and $\Delta$ do. Given these two sets of formulas, we define
\begin{align*}
  \genneg \varphi & \assign \varphi \genrightarrow \genbot = \bigcup_{\pi \in \genbot} \varphi \genrightarrow \pi.
\end{align*}
  If $\tuple{\varphi}$ is a finite tuple of formulas, we define the set of formulas $\tuple{\varphi} \genrightarrow \psi$ as follows:
\begin{align*}
  \emptyset \genrightarrow \psi & \assign \{ \psi \}, \\
  \langle \tuple{\varphi}, \varphi_{n+1} \rangle \genrightarrow \psi & \assign \set{\varphi_{n+1} \genrightarrow \chi}{\chi \in \tuple{\varphi} \genrightarrow \psi}.
\end{align*}

\begin{fact}
  If $\logic{L}$ enjoys the unary global DDT w.r.t.~$p \genrightarrow q$, then it enjoys the finitary global DDT w.r.t~the family $\{ p_{1}, \dots, p_{n} \} \genrightarrow q$.
\end{fact}

\begin{proposition} \label{prop: lem with gddt}
  Let $\logic{L}$ be a logic with a finite antitheorem $\genbot$ and a unary global DDT w.r.t.~a finite set $p \genrightarrow q$. Then $\logic{L}$ enjoys the unary LEM if and only if $(p \genrightarrow q) \genrightarrow ((\genneg p \Rightarrow q) \genrightarrow q)$ is a theorem of $\logic{L}$.
\end{proposition}

\begin{proof}
  We know that $\logic{L}$ enjoys the unary LEM if and only if it enjoys the unary global LEM w.r.t.~$\neg p$ (Propositions~\ref{prop: dil iff lem} and \ref{prop: il and dil imply cil}). By the global DDT this LEM amounts to the implication
\begin{prooftree}
  \def\fCenter{\vdash_{\logic{L}}}
  \Axiom$\Gamma \fCenter \varphi \genrightarrow \psi$
  \Axiom$\Gamma \fCenter \genneg \varphi \genrightarrow \psi$
  \BinaryInf$\Gamma \fCenter \psi$
\end{prooftree}
  This is equivalent to $\varphi \genrightarrow \psi, \genneg \varphi \genrightarrow \psi \vdash_{\logic{L}} \psi$ for all $\varphi$ and $\psi$, i.e.\ $p \genrightarrow q, \genneg p \genrightarrow q \vdash_{\logic{L}} q$. By the global DDT, this is equivalent to $\emptyset \vdash_{\logic{L}} (p \genrightarrow q) \genrightarrow ((\genneg p \Rightarrow q) \genrightarrow q)$.
\end{proof}

  We call $(p \genrightarrow q) \genrightarrow ((\genneg p \Rightarrow q) \genrightarrow q)$ the DDT--axiomatic form of the LEM for~$\logic{L}$. For example, the DDT-axiomatic form of the LEM for $\FLen$ and $\IKnfour$ is
\begin{align*}
  {(p \rightarrow_{n} q) \rightarrow_{n} ((\neg_{n} p \rightarrow_{n} q) \rightarrow_{n} q)},
\end{align*}
  where $x \rightarrow_{n} y \assign (1 \wedge x)^{n} \rightarrow y$ and $\neg_{n} x \assign \neg (1 \wedge x)^{n}$ in the case of $\FLen$ and $x \rightarrow_{n} y \assign \Box_{n} x \rightarrow y$ and $\neg_{n} x \assign \neg \Box_{n} x$ in the case of $\IKnfour$.

  Although a similar procedure works for the local DDT as well, its utility there appears to be limited. For example, suppose first that $\logic{L}$ is a logic which enjoys the local DDT of $\FLe$ w.r.t.~the family $\set{\{ p \rightarrow_{n} q \}}{n \in \omega}$, perhaps on account of being an axiomatic extension of~$\FLe$. Then, using this local DDT, $\logic{L}$~enjoys the local LEM w.r.t.~the family $\set{\{ \neg_{n} p \}}{n \in \omega}$ if and only if for each $m \in \omega$ and each function $f\colon \omega \to \omega$
\begin{prooftree}
  \def\fCenter{\vdash_{\logic{L}}}
  \Axiom$\Gamma \fCenter \varphi \rightarrow_{m} \psi$
  \Axiom$\Gamma \fCenter \neg_{n} \varphi \rightarrow_{f(n)} \psi$
  \BinaryInf$\Gamma \fCenter \psi$
\end{prooftree}
  This is equivalent to the fact that for each $m \in \omega$ and $f\colon \omega \to \omega$ and all $\varphi$ and $\psi$
\begin{align*}
  \varphi \rightarrow_{m} \psi, \set{\neg_{n} \varphi \rightarrow_{f(n)} \psi}{n \in \omega} \vdash_{\logic{L}} \psi.
\end{align*}
  By structurality this is equivalent to the fact that for each $m \in \omega$ and $f\colon \omega \to \omega$
\begin{align*}
  p \rightarrow_{m} q, \set{\neg_{n} p \rightarrow_{f(n)} q}{n \in \omega} \vdash_{\logic{L}} q.
\end{align*}
  One could go a step further and use finitarity and the local DDT to express this rule as the validity of some set of theorems. However, the resulting family of sets of theorems would be rather complicated.

  A more familiar form of the LEM can be obtained for logics which enjoy what Cintula \& Noguera~\cite{cintula+noguera13} call the proof by cases property (PCP). For example, the logic $\FLe$ and its axiomatic extensions satisfy the equivalence
\begin{align*}
  \Gamma, \varphi \vdash_{\FLe} \chi \text{ and } \Gamma, \psi \vdash_{\FLe} \chi & \iff \Gamma, (1 \wedge \varphi) \vee (1 \wedge \psi) \vdash_{\FLe} \chi,
\end{align*}
   while $\IKnfour$ and its axiomatic extensions satisfy the equivalence
\begin{align*}
  \Gamma, \varphi \vdash_{\IKnfour} \chi \text{ and } \Gamma, \psi \vdash_{\IKnfour} \chi & \iff \Gamma, \Box_{n} \varphi \vee \Box_{n} \psi \vdash_{\IKnfour} \chi.
\end{align*}
  The logic $\IK$ only satisfies a local form of this condition, namely
\begin{align*}
  \Gamma, \varphi \vdash_{\IK} \chi \text{ and } \Gamma, \psi \vdash_{\IK} \chi & \iff \Gamma, \Box_{n} \varphi \vee \Box_{n} \psi \vdash_{\IK} \chi \text{ for some } n \in \omega.
\end{align*}

\begin{definition}
  A logic $\logic{L}$ enjoys the \emph{binary global proof by cases property (PCP)} if there is a set of formulas $\pcpset(p, q)$ such that
\begin{align*}
  \Gamma, \varphi \vdash_{\logic{L}} \chi \text{ and } \Gamma, \psi \vdash_{\logic{L}} \chi & \iff \Gamma, \pcpset(\varphi, \psi) \vdash_{\logic{L}} \chi \text{ for some } n \in \omega.
\end{align*}
\end{definition}

\newcommand{\auxcit}{\cite[Lemma~3.10]{cintula+noguera13}}

\begin{fact}[\auxcit]
  Let $\sqcup(\Phi, \Psi) \assign \bigcup_{\varphi \in \Phi} \bigcup_{\psi \in \Psi} \sqcup(\varphi, \psi)$. If $\logic{L}$ enjoys the binary global PCP w.r.t.~$\pcpset(p, q)$,~then for all finite sets $\Phi$ and~$\Psi$
\begin{align*}
  \Gamma, \Phi \vdash_{\logic{L}} \chi \text{ and } \Gamma, \Psi \vdash_{\logic{L}} \chi \iff \Gamma, \pcpset(\Phi, \Psi) \vdash_{\logic{L}} \chi.
\end{align*}
\end{fact}

\begin{proposition} \label{prop: lem with pcp}
  Let $\logic{L}$ be a logic which enjoys the unary global IL w.r.t.~a finite set $\genneg p$ and the binary global PCP w.r.t.~$\pcpset(p, q)$. Then $\logic{L}$ enjoys the unary LEM if and only if $\pcpset(p, \genneg p)$ is a theorem of $\logic{L}$.
\end{proposition}

\begin{proof}
  We know that $\logic{L}$ enjoys the unary LEM if and only if it enjoys the unary global LEM w.r.t.~$\neg p$ (Propositions~\ref{prop: dil iff lem} and \ref{prop: il and dil imply cil}). By the global binary PCP this amounts to the implication
\begin{prooftree}
  \def\fCenter{\vdash_{\logic{L}}}
  \AxiomC{$\Gamma, \sqcup(\varphi, \genneg \varphi) \vdash_{\logic{L}} \psi$}
  \UnaryInfC{$\Gamma \vdash_{\logic{L}} \psi$}
\end{prooftree}
  This implication is equivalent to $\emptyset \vdash_{\logic{L}} \pcpset(\varphi, \genneg \varphi)$.
\end{proof}

  We call $\pcpset(p, \genneg p)$ the PCP--axiomatic form of the LEM for $\logic{L}$. For example, the PCP--axiomatic form of the LEM for $\FLen$ is
\begin{align*}
  (1 \wedge p) \vee (1 \wedge \neg (1 \wedge p)^{n}),
\end{align*}
  while the PCP--axiomatic form of the LEM for $\IKnfour$ is
\begin{align*}
  \Box_{n} p \vee \Box_{n} \neg \Box_{n} p,
\end{align*}
  which we have already shown to be equivalent to $p \vee \Box_{1} \neg \Box_{n} p$ in Fact~\ref{fact: cyclicity}.

  We now show that the DDT--axiomatic and PCP--axiomatic form of the LEM precisely describes the semisimple models of $\logic{L}$.

\begin{proposition} \label{prop: axiomatization of semisimple companion}
  Let $\logic{L}$ be a coatomic logic with a finite antitheorem and a unary global DDT w.r.t.~a finite set (alternatively, with a unary global IL w.r.t.~a finite set and a binary global PCP). Then $\ssc{\logic{L}}$ is precisely the extension of $\logic{L}$ by the DDT--axiomatic (alternatively, the PCP--axiomatic) form of the LEM.
\end{proposition}

\begin{proof}
  For the IL set, DDT set, and PCP set we use the notation $\genneg \varphi$, $\varphi \Rightarrow \psi$, and $\pcpset(\Phi, \Psi)$ introduced in Subsection~\ref{subsec: lem}. Let $\logic{L}'$ be the extension of $\logic{L}$ by the DDT--axiomatic form of the LEM. We first prove that $\logic{L}' \logleq \ssc{\logic{L}}$. Since $\ssc{\logic{L}}$ is complete w.r.t.\ the simple theories of $\logic{L}$, it suffices to prove that each simple theory $T$ of $\logic{L}$ is a theory of $\logic{L}'$. But for each simple theory $T$ of $\logic{L}$ either $T \vdash_{\logic{L}} \varphi$ or $T, \varphi \vdash_{\logic{L}} \allset$, in~which case $T \vdash_{\logic{L}} \genneg \varphi$. In either case $T, \varphi \genrightarrow \psi, \genneg \varphi \Rightarrow \psi \vdash_{\logic{L}} \psi$ by the DDT, therefore $T$ proves the DDT--axiomatic form of the LEM again by the DDT. Thus $T$ is a theory of $\logic{L}'$.

  Conversely, the logic $\logic{L}'$ is an axiomatic extension of $\logic{L}$, therefore it inherits the global DDT of $\logic{L}$. By Proposition~\ref{prop: lem with gddt} the DDT--axiomatic form of the LEM implies that $\logic{L}'$ is semisimple. It thus suffices to prove that each simple theory $T$ of $\logic{L}'$ is a theory of $\ssc{\logic{L}}$. But because $\logic{L}$ is coatomic, $T$ extends to a simple theory $T'$ of $\logic{L}$, which is a theory of $\logic{L}'$ by the previous paragraph. The simplicity of~$T$ with respect to $\logic{L}'$ now implies that $T' = T$, hence $T$ is a simple theory of $\logic{L}$. As~such, it is by definition a theory of $\ssc{\logic{L}}$.

  The proof in the PCP case is analogous. If $T$ is a simple theory of $\logic{L}$, then either $T \vdash_{\logic{L}} \varphi$ or $T \vdash_{\logic{L}} \genneg \varphi$, so $T \vdash_{\logic{L}} \sqcup(\varphi, \genneg \varphi)$. Conversely, the logic $\logic{L}'$ which extends $\logic{L}$ by the PCP--axiomatic form of the LEM inherits the binary global PCP as well as the unary global IL of $\logic{L}$. By Proposition~\ref{prop: lem with pcp} the PCP--axiomatic of the LEM implies that $\logic{L}'$ is semisimple. The same argument as in the DDT case now shows that each simple theory $T$ of $\logic{L}'$ is a theory of $\ssc{\logic{L}}$.
\end{proof}

\begin{theorem} \label{thm: semisimple models}
  Let $\logic{L}$ be a compact logic which enjoys a unary global DDT w.r.t.~a finite set (alternatively, a unary global IL w.r.t.~a finite set and a binary global PCP). Then a model of $\logic{L}$ is semisimple if and only if it validates the DDT--axiomatic (alternatively, the PCP--axiomatic) form of the LEM.
\end{theorem}

\begin{proof}
  This follows from the previous proposition and Theorem~\ref{thm: models of semisimple companion}.
\end{proof}

  If $\logic{L}$ is a (weakly) algebraizable and $\class{K}$ is its algebraic counterpart, then on each algebra the lattice of $\logic{L}$-filters is isomorphic to the lattice of $\class{K}$-congruences, therefore the above theorem tells us that the semisimple algebras of $\class{K}$ are precisely those which satisfy the equational translation of (either of the two forms of) the LEM. For $\FLen$ algebras and $\IKnfour$-algebras, this yields the following results.

\begin{fact}
  An $\FLen$-algebra is semisimple if and only if it satisfies the equation $1 \inequals x \sqcup \neg_{n} x$, or equivalently the equation $x \rightarrow_{n} y \inequals (\neg_{n} x \rightarrow_{n} y) \rightarrow_{n} y$.
\end{fact}

\begin{fact}
  An $\IKnfour$-algebra is semisimple if and only if it satisfies the equation $1 \equals x \sqcup \neg_{n} x$, or equivalently the equation $x \rightarrow_{n} y \inequals (\neg_{n} x \rightarrow_{n} y) \rightarrow_{n} y$.
\end{fact}

  The simpler form of the axiom of $n$-cyclicity (Fact~\ref{fact: cyclicity}) yields an alternative formulation of the last fact.

\begin{fact}
  A $\Knfour$-algebra is semisimple if and only if it satisfies $x \inequals \Box \Diamond_{n} x$. An $\IKnfour$-algebra is semisimple if and only if it satisfies $1 \equals x \vee \Box_{1} \neg \Box_{n} x$.
\end{fact}

  These facts may not be particularly difficult to prove directly. Our~point, \mbox{however}, is that the above theorem enables us to prove such results in a uniform~way for any logic which satisfies certain syntactic prerequisites. Moreover, it immediately extends to all (weakly) algebraizable fragments of these logics which contain sufficient syntactic material to express the IL and the DDT or the~PCP. (Although the two forms of the LEM are equivalent for $\FLen$ and $\IKnfour$, in such fragments only one of them may be expressible.)

\subsection{Glivenko theorems}
\label{subsec: glivenko}

  In this final section, we show that a Glivenko-like connection obtains between a logic and its semisimple companion, subsuming several known Glivenko theorems under a single umbrella.

\begin{proposition} \label{prop: glivenko}
  Let $\logic{L}$ be a coatomic logic with the $\kappa$-ary local IL w.r.t.\ a family $\ilfamily$ such that each $\ilset \in \ilfamily_{1}$ has cardinality less than~$\kappa$. Then $\Gamma \vdash_{\ssc{\logic{L}}} \varphi$ if and only if for each $\ilset \in \ilfamily_{1}$ there is some $\ilsettwo \in \ilfamily_{|\ilset|}$ such that $\Gamma \vdash_{\logic{L}} \ilsettwo(\ilset(\varphi))$.
\end{proposition}

\begin{proof}
  By Proposition~\ref{prop: antiadmissible rules} the valid rules of $\ssc{\logic{L}}$ coincide with the antiadmissible rules of $\logic{L}$. Using Proposition~\ref{prop: simply antiadmissible rules} and the local IL, the antiadmissible rules of $\logic{L}$ are precisely those rules $\Gamma \vdash \varphi$ such that $\Delta \vdash_{\logic{L}} \ilset(\varphi)$ implies $\Gamma, \Delta \vdash_{\logic{L}} \allset$ for each $\Delta$ and each $\ilset \in \ilfamily_{1}$. In other words, $\Gamma, \ilset(\varphi) \vdash_{\logic{L}} \allset$ for each $\ilset \in \ilfamily_{1}$. Another application of the local IL yields the desired conclusion.
\end{proof}

  The above equivalence may be called a \emph{local} Glivenko theorem connecting $\logic{L}$ and $\ssc{\logic{L}}$: it states that
\begin{align*}
  \Gamma \vdash_{\ssc{\logic{L}}} \varphi \iff \Gamma \vdash_{\logic{L}} \set{f(\ilset)(\ilset(\varphi))}{\ilset \in \ilfamily_{1}} \text{ for some } f\colon \ilfamily_{1} \to \bigcup_{0 < \alpha < \kappa} \ilfamily_{\alpha}
\end{align*}
  such that $f(\ilset) \in \ilfamily_{|\ilset|}$. For example, for $\FLew$ we get the local Glivenko theorem
\begin{align*}
  \Gamma \vdash_{\ssc{\FLew}} \varphi \iff \Gamma \vdash_{\FLew} \set{\neg (\neg \varphi^{n})^{f(n)} }{n \in \omega} \text{ for some } f\colon \omega \to \omega,
\end{align*}
  while for $\IK$ we get the local Glivenko theorem
\begin{align*}
  \Gamma \vdash_{\ssc{\IK}} \varphi \iff \Gamma \vdash_{\IK} \set{\neg \Box_{f(n)} (\neg \Box_{n} \varphi}{n \in \omega} \text{ for some } f\colon \omega \to \omega.
\end{align*}
  Of course, for such Glivenko theorems to be of interest we need to be able to identify the semisimple companions of these logics. For example, for $\BL$ we get the local Glivenko theorem
\begin{align*}
  \Gamma \vdash_{\Lukinfty} \varphi \iff \Gamma \vdash_{\BL} \set{\neg (\neg \varphi^{n})^{f(n)} }{n \in \omega} \text{ for some } f\colon \omega \to \omega.
\end{align*}
  If $\Gamma$ is finite, this theorem is superseded by the much simpler equivalence
\begin{align*}
  \Gamma \vdash_{\Luk} \varphi \iff \Gamma \vdash_{\BL} \neg \neg \varphi,
\end{align*}
  which was proved by Cignoli \& Torrens~\cite[Theorem~2.1]{cignoli+torrens03}.\footnote{Cignoli and Torrens prove the equivalence for theorems, but their proof extends to consequence relations. In fact, the equivalence extends to a connection between $\Lukinfty$ and the infinitary logic $\BLinfty$, defined as the logic of all $\BL$-algebras over the unit interval~$[0, 1]$. Using the decomposition of such $\BL$-algebras into ordinal sums of {\L}ukasiewicz and product components~\cite[Theorem~2.16]{metcalfe+olivetti+gabbay08}, one easily proves that $\Gamma \vdash_{\Lukinfty} \varphi \iff \Gamma \vdash_{\BL_{\infty}} \neg \neg \varphi$, from which $\Gamma \vdash_{\Luk} \varphi \iff \Gamma \vdash_{\BL} \neg \neg \varphi$ follows.}

  The most satisfactory results hold for logics which enjoys the global IL w.r.t.\ some set $\genneg x$, and moreover the LEM can be expressed in axiomatic form. For~such logics, $\ssc{\logic{L}}$ is the extension of $\logic{L}$ by the DDT--axiomatic or PCP--axiomatic form of the LEM, as shown in the previous subsection, and moreover we obtain a global rather than a local form of the Glivenko theorem:
\begin{align*}
  \Gamma \vdash_{\ssc{\logic{L}}} \varphi & \iff \Gamma \vdash_{\logic{L}} \genneg \genneg \varphi.
\end{align*}
  This immediately yields the following Glivenko theorems in the substructural and intuitionistic modal cases.

\begin{fact}
  Let $\logic{L}$ be an axiomatic extension of $\FLen$ and let $\logic{L} + \lem$ be the extension of $\logic{L}$ by the axiom $(1 \wedge p) \vee (1 \wedge \neg (1 \wedge p)^{n})$. Then
\begin{align*}
  \Gamma \vdash_{\logic{L} + \lem} \varphi \iff \Gamma \vdash_{\logic{L}} \neg (1 \wedge \neg (1 \wedge \varphi)^{n})^{n}.
\end{align*}
\end{fact}

\begin{fact} \label{fact: glivenko for iml}
  Let $\logic{L}$ be an axiomatic extension of $\IKnfour$ and let $\logic{L} + \lem$ be the extension of $\logic{L}$ by the axiom $\Box_{n} p \vee \Box_{n} \neg \Box_{n} p$ (or equivalently, $p \vee \Box_{1} \neg \Box_{n} p$). Then
\begin{align*}
  \Gamma \vdash_{\logic{L} + \lem} \varphi \iff \Gamma \vdash_{\logic{L}} \neg \Box_{n} \neg \Box_{n} \varphi.
\end{align*}
\end{fact}

  The familiar Glivenko theorem for intuitionistic logic
\begin{align*}
  \Gamma \vdash_{\CL} \varphi \iff \Gamma \vdash_{\IL} \neg \neg \varphi,
\end{align*}
  as well as the Glivenko theorem for $\Sfour$
\begin{align*}
  \Gamma \vdash_{\Sfive} \varphi \iff \Gamma \vdash_{\Sfour} \Diamond \Box \varphi,
\end{align*}
  are immediate special cases (see~\cite{glivenko29,matsumoto55}). We also obtain the Glivenko theorem of Bezhanishvili~\cite[Thm\,10]{bezhanishvili01} relating two different intuitionistic versions of the modal logic $\Sfive$, namely the logics known as $\MIPC$ and $\WSfive$. $\MIPC$ is the logic of so-called monadic Heyting algebras, i.e.\ modal Heyting algebras which satisfy $\Box \Box x \equals \Box x \inequals x$, $x \inequals \Diamond x \equals \Diamond \Diamond x$ and $x \inequals \Box \Diamond x$ and $\Diamond \Box x \inequals x$, and $\WSfive$ adds the axiom $\neg \Box \neg x \equals \Diamond x$, or equivalently $1 \equals \Box x \vee \neg \Box x$, to $\MIPC$~\cite[Thm\,23]{bezhanishvili98}.

\begin{fact}
  $\WSfive = \MIPC + \lem$.
\end{fact}

\begin{proof}
  If $1 \equals \Box x \vee \Box \neg \Box x$, then $1 \equals \Box x \vee \neg \Box x$, using the axiom $\Box x \leq x$. Conversely, if $1 \equals \Box x \vee \neg \Box x$, then $1 \equals x \vee \Box \neg \Box x$ is equivalent to $\neg \Box \neg \Box x \leq x$, but indeed $\neg \Box \neg \Box x = \Diamond \Box x \leq x$.
\end{proof}

  Theorem~\ref{thm: semisimple models} now immediately yields the known fact that $\WSfive$-algebras are precisely the semisimple monadic Heyting algebras \cite[Thm\,24]{bezhanishvili98}.

\begin{fact}[\cite{bezhanishvili01}]
  The following Glivenko theorem connects the intuitionistic modal logics $\MIPC$ and $\WSfive$:
\begin{align*}
  \Gamma \vdash_{\WSfive} \varphi \iff \Gamma \vdash_{\MIPC} \neg \neg \Box \varphi.
\end{align*}
\end{fact}

\begin{proof}
  In view of the general Glivenko theorem for axiomatic extensions of $\IKnfour$ (Fact~\ref{fact: glivenko for iml}), it suffices to prove that $\neg \neg \Box \varphi$ and $\neg \Box \neg \Box \varphi$ are equivalent in $\MIPC$. In~monadic Heyting algebras we have $\Box \neg \Box x \leq \neg \Box x$, thus $\neg \neg \Box x \leq \neg \Box \neg \Box x$. Conversely $\Diamond z \leq \neg \Box \neg z$ implies that $\neg \Box \neg \Box \neg z \leq \neg z$, so taking $z \assign \neg y$ yields $\neg \Box \neg \Box y \leq \neg \Box \neg \Box \neg \neg y \leq \neg \neg y$, and taking $y \assign \Box x$ yields $\neg \Box \neg \Box x = \neg \Box \neg \Box \Box x \leq \neg \neg \Box x$. 
\end{proof}

  Finaly, let us describe the subclassical substructural logics which are Glivenko related to classical logic. This extends the Glivenko equivalence between the fuzzy logic $\SBL$ (strict $\BL$) and classical logic $\CL$ due to Cignoli \& Torrens~\cite[Theorem~2.2]{cignoli+torrens03}. The following theorem is the only place in the paper where we use the constant $0$ in the signature of $\FL$. Let us use the notation $\negbs x \assign x \bs 0$ and $\negs x \assign 0 / x$ (the shape of the symbol indicates whether $x$ occurs to the left or to the right of the slash).

\begin{theorem}
  Let $\logic{L}$ be a compact logic such that $\FL \logleq \logic{L} \logleq \CL$. Then the following are equivalent:
\begin{enumerate}
\item $\Gamma \vdash_{\CL} \varphi \iff \Gamma \vdash_{\logic{L}} \negbs \negbs \varphi$ for all $\Gamma$ and $\varphi$,
\item $\Gamma \vdash_{\CL} \varphi \iff \Gamma \vdash_{\logic{L}} \negs \negs \varphi$ for all $\Gamma$ and $\varphi$,
\item $\Gamma \vdash_{\CL} \negbs \varphi \iff \Gamma \vdash_{\logic{L}} \negbs \varphi$ for all $\Gamma$ and $\varphi$,
\item $\Gamma \vdash_{\CL} \negs \varphi \iff \Gamma \vdash_{\logic{L}} \negs \varphi$ for all $\Gamma$ and $\varphi$,
\item $\Gamma, \varphi \vdash_{\logic{L}} 0 \iff \Gamma \vdash_{\logic{L}} \negbs \varphi$ for all $\Gamma$ and $\varphi$, and $\negbs (p \wedge \negbs q) \vdash_{\logic{L}} \negbs \negbs (p \bs q)$.
\item $\Gamma, \varphi \vdash_{\logic{L}} 0 \iff \Gamma \vdash_{\logic{L}} \negs \varphi$ for all $\Gamma$ and $\varphi$, and $\negs (p \wedge \negs q) \vdash_{\logic{L}} \negs \negs (q / p)$.
\end{enumerate}
\end{theorem}

\begin{proof}
  We prove the equivalence of (1), (3), and (5). The equivalence of (2), (4), and (6) is proved in an entirely analogous manner, and (3) and (4) are equivalent because $\negbs p \vdash_{\FL} \negs p$ and $\negs p \vdash_{\FL} \negbs p$. 

  (1) $\Rightarrow$ (3): if $\Gamma \vdash_{\CL} \negbs \varphi$, then $\Gamma \vdash_{\CL} \negs \varphi$, therefore $\Gamma \vdash_{\logic{L}} \negbs \negbs \negs \varphi$ by (1). But $\negbs p \vdash_{\FL} \negs p$ and $\negs \negbs \negs p \vdash_{\FL} \negs p$, therefore $\Gamma \vdash_{\logic{L}} \negs \negbs \negs \varphi$ and $\Gamma \vdash_{\logic{L}} \negs \varphi$. The rule $\negs p \vdash_{\FL} \negbs p$ now yields $\Gamma \vdash_{\logic{L}} \negbs \varphi$. Conversely, $\Gamma \vdash_{\logic{L}} \negbs \varphi$ implies $\Gamma \vdash_{\CL} \negbs \varphi$ because $\logic{L}$ is subclassical.

  (3) $\Rightarrow$ (5): the rule in question is valid in $\CL$, thus it is valid in $\logic{L}$ by~(3). If~$\Gamma \vdash_{\logic{L}} \negbs \varphi$, then $\Gamma, \varphi \vdash_{\logic{L}} 0$ because $p, \negbs p \vdash_{\FL} 0$. Conversely, if $\Gamma, \varphi \vdash_{\logic{L}} 0$, then $\Gamma, \varphi \vdash_{\CL} 0$, hence $\Gamma \vdash_{\CL} \negbs \varphi$, and $\Gamma \vdash_{\logic{L}} \negbs \varphi$ by (3).

  (5) $\Rightarrow$ (3): let $\logic{L}_{0}$ be the extension of $\logic{L}$ by the rule $0 \vdash \bot$. We claim that
\begin{align*}
  \Gamma \vdash_{\logic{L}} 0 \iff \Gamma \vdash_{\logic{L}_{0}} \allset.
\end{align*}
  The left-to-right direction is trivial, and conversely either the proof of $\Gamma \vdash_{\logic{L}_{0}} \bot$ does not use the rule $0 \vdash \bot$, in which case $\Gamma \vdash_{\logic{L}} \bot \vdash_{\logic{L}} 0$, or it does use the rule, in which case $\Gamma \vdash_{\logic{L}} 0$. Now if~$\ssc{\logic{L}_{0}} = \CL$, then
\begin{align*}
  \Gamma \vdash_{\CL} \varphi & \iff \Gamma, \negbs \varphi \vdash_{\CL} \allset \text{ by the IL of $\CL$} \\
  & \iff \Gamma, \negbs \varphi \vdash_{\logic{L}_{0}} \allset \text{ because } \ssc{\logic{L}_{0}} = \CL \\
  & \iff \Gamma \negbs \varphi \vdash_{\logic{L}} 0 \text{ by the above equivalence} \\
  & \iff \Gamma \vdash_{\logic{L}} \negbs \negbs \varphi \text{ by (5)}.
\end{align*}
  It will therefore suffice to prove that $\ssc{\logic{L}_{0}} = \CL$.

  The logic $\logic{L}_{0}$ enjoys the global IL of intuitionistic logic: in one direction, $\Gamma \vdash_{\logic{L}_{0}} \neg (\varphi_{1} \wedge \dots \wedge \varphi_{k})$ implies $\Gamma, \varphi_{1}, \dots, \varphi_{k} \vdash_{\logic{L}_{0}} \allset$, since $0 \vdash_{\logic{L}_{0}} \bot$. Conversely,
\begin{align*}
  \Gamma, \varphi_{1}, \dots, \varphi_{k} \vdash_{\logic{L}_{0}} \allset & \implies \Gamma, \varphi_{1}, \dots, \varphi_{k} \vdash_{\logic{L}} 0 \\
  & \implies \Gamma, \varphi_{1} \wedge \dots \wedge \varphi_{k} \vdash_{\logic{L}} 0 \\
  & \implies \Gamma \vdash_{\logic{L}} \neg (\varphi_{1} \wedge \dots \wedge \varphi_{k}) \text{ by (5)} \\
  & \implies \Gamma \vdash_{\logic{L}_{0}} \neg (\varphi_{1} \wedge \dots \wedge \varphi_{k}).
\end{align*}
  The logic $\ssc{\logic{L}_{0}}$ enjoys the classical global IL with respect to the same IL family by Proposition~\ref{prop: cil for semisimple companion}, therefore it enjoys the global DDT in the form
\begin{align*}
  \Gamma, \varphi \vdash_{\ssc{\logic{L}_{0}}} \psi & \iff \Gamma \vdash_{\ssc{\logic{L}_{0}}} \negbs (\varphi \wedge \negbs \psi).
\end{align*}
  Since $\logic{L}_{0}$ is compact and enjoys a finitary global IL w.r.t.~a family of finite sets,  the assumed rule $\negbs (\negbs p \wedge q) \vdash_{\logic{L}_{0}} \negbs \negbs (p \bs q)$ yields that $\negbs (p \wedge \negbs q) \vdash_{\ssc{\logic{L}_{0}}} p \bs q$ by the Glivenko theorem connecting $\logic{L}_{0}$ and $\ssc{\logic{L}_{0}}$, i.e.\ by Proposition~\ref{prop: glivenko}. Due~to this rule, $\ssc{\logic{L}_{0}}$ enjoys the DDT of intuitionistic logic:
\begin{align*}
  \Gamma, \varphi \vdash_{\ssc{\logic{L}_{0}}} \psi & \iff \Gamma \vdash_{\ssc{\logic{L}_{0}}} \varphi \bs \psi.
\end{align*}
  This DDT immediately implies that the axioms of Exchange, Contraction, and Weakening, i.e.\ the formulas $(p \bs (q \bs r)) \bs (q \bs (p \bs r)$, $(p \bs (p \bs q)) \bs (p \bs q),$ and $p \bs (q \bs p)$, are theorems of $\ssc{\logic{L}_{0}}$. Since adding these axioms to $\FL$ yields intuitionistic logic, $\IL \logleq \logic{L}_{0} \logleq \CL$. But each non-trivial theory of $\IL$ extends to a simple theory of $\IL$, which is a theory of $\CL$ and therefore also a theory of $\logic{L}_{0}$. It follows that the simple theories of $\IL$ and $\logic{L}_{0}$ coincide, therefore $\ssc{\logic{L}_{0}} = \ssc{\IL} = \CL$.
\end{proof}

  The associativity of $\FL$ was not used in the above proof, therefore the same theorem in fact holds for extensions of the non-associative Full Lambek calculus.


\end{document}